\documentclass[12pt,letterpaper,titlepage]{amsart}
\usepackage{amsmath, amssymb, amsthm, amsfonts,amscd,amsaddr}

\theoremstyle{plain}
\newtheorem{theorem}{Theorem}[section]
\newtheorem{proposition}[theorem]{Proposition}
\newtheorem{cor}[theorem]{Corollary}

\newtheorem{prop}[theorem]{Proposition}
\newtheorem{lemma}[theorem]{Lemma}

\theoremstyle{definition}
\newtheorem{ex}[theorem]{Example}
\newtheorem{rmk}[theorem]{Remark}
\numberwithin{equation}{section}
\newtheorem*{theoremA*}{Theorem A}
\newtheorem*{theoremB*}{Theorem B}
\newtheorem*{theorem1*}{Theorem A'}
\newtheorem*{theoremC*}{Theorem C}
\newtheorem*{theoremD*}{Theorem D}
\newtheorem*{theoremE*}{Theorem E}
\newtheorem*{theoremF*}{Theorem F}
\newtheorem*{theoremE2*}{Theorem E2}
\newtheorem*{theoremE3*}{Theorem E3}
\newcommand{\bs}{\backslash}

\newcommand{\C}{\mathbb{C}}
\newcommand{\G}{\mathbb{G}}
\newcommand{\A}{\mathcal{A}}
\newcommand{\Lc}{\mathcal{L}}
\newcommand{\E}{\mathcal{E}}
\newcommand{\Hc}{\mathcal{H}}
\newcommand{\Hb}{\mathbb{H}}
\newcommand{\Q}{\mathbb{Q}}
\newcommand{\Lb}{\mathbf{L}}

\newcommand{\Z}{\mathbb{Z}}
\newcommand{\Zc}{\mathcal{Z}}

\newcommand{\R}{\mathbb{R}}
\newcommand{\N}{\mathbb{N}}

\newcommand{\Dens}{\operatorname{Dens}}

\newcommand{\Sl}{\operatorname{SL}}
\newcommand{\SO}{\operatorname{SO}}

\newcommand{\SL}{\operatorname{SL}}

\newcommand{\SU}{\operatorname{SU}}
\newcommand{\tr}{\operatorname{tr}}
\newcommand{\im}{\operatorname{Im}}

\newcommand{\Ad}{\operatorname{Ad}}

\newcommand{\ad}{\operatorname{ad}}
\newcommand{\diag}{\operatorname{diag}}
\newcommand{\pr}{\operatorname{pr}}

\newcommand{\supp}{\operatorname{supp}}
\newcommand{\Span}{\operatorname{span}}

\newcommand{\rank}{\operatorname{rank}}

\newcommand{\re}{\operatorname{Re}}

\def\hat{\widehat}
\def\af{\mathfrak{a}}
\def\bfrak{\mathfrak{b}}
\def\e{\epsilon}
\def\gf{\mathfrak{g}}

\def\hf{\mathfrak{h}}
\def\kf{\mathfrak{k}}
\def\lf{\mathfrak{l}}
\def\mf{\mathfrak{m}}
\def\nf{\mathfrak{n}}

\def\sof{\mathfrak{so}}
\def\spin{\mathfrak{spin}}
\def\glf{\mathfrak{gl}}
\def\pf{\mathfrak{p}}
\def\qf{\mathfrak{q}}

\def\sl{\mathfrak{sl}}
\def\symp{\mathfrak{sp}}
\def\suf{\mathfrak{su}}

\def\uf{\mathfrak{u}}

\def\zf{\mathfrak{z}}
\def\la{\langle}
\def\ra{\rangle}
\def\1{{\bf1}}

\def\U{\mathcal{U}}

\def\D{\mathcal {D}}
\def\G{\mathcal{G}}
\def\sE{\mathsf {E}}
\def\Oc{\mathcal{O}}
\def\P{\mathcal{P}}

\def\M{\mathcal{M}}
\def\oline{\overline}
\def\F{\mathcal{F}}

\def\W{\mathcal{W}}

\def\tilde{\widetilde}
\def\Sphere{\mathbf{S}}
\def\sG{\mathsf{G}}
\usepackage[usenames]{color}

\title[Tempered spectrum]
{The tempered spectrum of a real spherical space}

\begin{document}
\date{July, 2017}

\begin{abstract}
Let $G/H$ be a unimodular real spherical space which is either absolutely spherical, i.e.~the real form of a complex spherical space, 
or of wave-front type. It is shown that every 
tempered representation of $G/H$ embeds into a 
twisted discrete series of 
a boundary degeneration of $G/H$. If $G/H$ is of wave-front type
it follows that the tempered representation is parabolically induced from a 
discrete series representation of a lower dimensional real spherical space.
\end{abstract}

\author[Knop]{Friedrich Knop}
\email{friedrich.knop@fau.de}
\address{Department Mathematik, Emmy-Noether-Zentrum\\
FAU Erlangen-N\"urnberg, Cauerstr. 11, 91058 Erlangen, Germany} 

\author[Kr\"otz]{Bernhard Kr\"{o}tz}
\email{bkroetz@gmx.de}
\address{Universit\"at Paderborn, Institut f\"ur Mathematik\\Warburger Stra\ss e 100, 
D-33098 Paderborn, Deutschland}
\thanks{The second author was supported by ERC Advanced Investigators Grant HARG 268105}
\author[Schlichtkrull]{Henrik Schlichtkrull}
\email{schlicht@math.ku.dk}
\address{University of Copenhagen, Department of Mathematics\\Universitetsparken 5, 
DK-2100 Copenhagen \O, Denmark}

\maketitle

\section{Introduction}
For a real reductive group $G$ one important step towards the Plan\-cherel theorem is the 
determination of the support of the Plancherel measure, i.e.~the portion of the 
unitary dual of $G$ which is contained weakly in $L^2(G)$. 
It turns out that these representations are the so-called tempered representations, i.e.
representations whose matrix coefficients satisfy a certain moderate growth 
condition. Further, one has the central theorem 
of Langlands: 

\bigskip {\bf Tempered Embedding Theorem} (\cite{Langlands}) 
{\it Every irreducible tempered representation $\pi$ is induced 
from discrete series, i.e.~there is a parabolic subgroup $P<G$ with Langlands 
decomposition $P=MAN$, a discrete series representation $\sigma$ of $M$,
and a unitary 
character $\chi$ of $A$,  such that 
$\pi$ is a subrepresentation of ${\rm Ind}_P^G (\sigma\otimes\chi\times\1)$.}
\bigskip

Thus (up to equivalence)
the description of the tempered spectrum is reduced to the classification of 
discrete series representations. 
A generalization with an analogous formulation 
was obtained by Delorme \cite{D} for symmetric spaces $G/H$.

\par In this article we consider the more general case of real spherical spaces,
that is, homogeneous spaces $Z=G/H$ on which the minimal parabolic subgroups of $G$
admit open orbits.  This case is more complicated and
several standard  techniques from the previous cases cannot be applied.
The result which in its formulation comes closest to the theorem above
is obtained for a particular class of real spherical spaces, said to be of
{\it wave-front} type. The spaces of this type feature a simplified large scale geometry 
\cite{KKSS2} and satisfy the wave-front lemma of Eskin--McMullen 
\cite{EM},\cite{KKSS2}. In particular, they are well suited for lattice counting problems \cite{KSS2}. 
The notion was originally introduced by Sakellaridis and Venkatesh \cite{SV}.  

\par Another large class of real spherical spaces $Z=G/H$ is obtained by taking real forms of 
complex spherical spaces $Z_\C=G_\C/H_\C$.  We call those $Z$ {\it absolutely spherical}. 
A generalization of Langlands' theorem will be obtained also for this class of spaces.
It should be noted that all symmetric spaces are both
absolutely spherical and of wave-front type, but there exist
real spherical spaces without any one or both of these properties.

Complex spherical spaces 
$Z_\C$ with $H_\C$ reductive were classified by Kr\"amer \cite{Kr} (for $G_\C$ simple) and Brion-Mikityuk
\cite{Brat}, \cite{Mik} (for $G_\C$ semi-simple). Recently we obtained a classification for all 
real spherical spaces 
$Z$ for $H$ reductive (see \cite{KKPS}, \cite{KKPS2}). 
In particular, if $G$ is simple, it turns out that  there are only few
cases which are neither absolutely spherical nor wave-front; concretely these are 
$\SL(n,\Hb)/\SL(n\!-\!1,\Hb)$ for $n\ge 3$,
$\SO(4,7)/ \SO(3) \times \mathrm{Spin}(3,4)$, and $\SO(3,6)/\mathsf{G}_2^1\times \SO(2)$.

\par Let now $Z=G/H$ be a unimodular real spherical space.  On the geometric 
level we attach to $Z$  a finite set of 
equal dimensional boundary degenerations $Z_I=G/H_I$. These boundary degenerations are real spherical 
homogeneous spaces parametrized by subsets $I$ of the set of spherical roots $S$ attached to $Z$.  

In the group case $Z=G \times G /\diag(G)\simeq G $ the boundary degenerations are of the form 
$Z_P =  G \times G /  [(N \times N^{\rm opp} ) \diag(M_0A)] $ 
attached to an arbitrary parabolic subgroup $P=MAN<G$. We 
remark that up to conjugation, parabolic subgroups are parametrized by finite subsets of the set of simple restricted roots.

To a Harish-Chandra module $V$ and a continuous $H$-invariant functional $\eta$ (on some completion of $V$) we 
attach a leading exponent $\Lambda_{V,\eta}$, and we provide a necessary and sufficient 
criterion, in terms of this leading exponent, 
for the pair $(V,\eta)$ to belong to a twisted
discrete series of $Z$, induced from a
unitary character of $H$.

\par In \cite{KKSS2} we defined tempered pairs $(V, \eta)$.  Under the assumption that $Z$ is absolutely spherical or wavefront,
the main result of this paper is then that 
for every tempered pair $(V, \eta)$ there exists a 
boundary degeneration $Z_I$ of $Z$ and an equivariant 
embedding  of $V$ into the twisted discrete series of $Z_I$.

\par In the special case of real spherical space of wave-front type our result gives rise to a tempered embedding 
theorem in the more familiar formulation of parabolic induction, see
Corollary \ref{temp-thm-wavefront}.
In the case of $p$-adic wave-front space such an embedding theorem 
was proved by Sakellaridis-Venkatesh \cite{SV}.

\subsection{The composition of the paper} There are two parts. A geometric part in Sections \ref{notation}-\ref{W-f},
and an analytic part in Sections \ref{power series}-\ref{tempspec}. 

\par The geometric part begins with the construction of the boundary 
degenerations $Z_I$ of $Z$. Here we choose an algebraic approach which gives 
the deformations of $\hf:={\rm Lie} (H)$ into $\hf_I:={\rm Lie}(H_I)$ via a simple 
limiting procedure. Then, in Section 3,  we relate the polar decomposition (see \cite{KKSS})
of $Z$ and $Z_I$ via the standard compactifications (see \cite{KK}). 
This is rather delicate, as representatives of open $P_{\rm min}$-orbits on $Z$ (and $Z_I$)  naturally enter the polar 
decompositions and these representatives need to be carefully chosen. 
Section 4 is concerned with real spherical spaces which we call induced: For any parabolic $P<G$ 
with Levi decomposition $P=G_P U_P$ we obtain an {\it induced} real spherical space $Z_P=G_P/H_P$
where $H_P<G_P$ is the projection of $P\cap H$ to $G_P$ along $U_P$.  
We conclude the geometric half with a treatment of wave-front spaces and elaborate 
on their special geometry. 

\par The analytic part starts with power series expansions for generalized matrix coefficients 
on $Z$.  We show that  the generalized matrix coefficients are solutions of a certain holonomic regular singular 
system of differential equations extending the results of  \cite{KS2}, Sect. 5.   

\par For every continuous $\hf$-invariant functional $\eta$ on $V$, we then construct $\hf_I$-invariant functionals $\eta_I$ on $V$
by extracting certain parts of the power series expansion. Passing from $\eta$ to $\eta_I$ 
can be considered as an algebraic version of the Radon transform. 
The technically most difficult part of this paper is then to establish the continuity of $\eta_I$. 
For symmetric spaces this would 
not be an issue in view of the automatic continuity theorem by  van den Ban, Brylinski and Delorme
\cite{BaD}, \cite{BrD}.   
For real spherical 
spaces such a result is currently 
not available.  {Under the assumption that $Z$ is either absolutely spherical or wave-front, we  settle this issue in Theorems \ref{lemma ub}, \ref{lemma ub wf} via quite delicate
optimal upper and lower bounds for the generalized matrix coefficients.
Already in the group case these bounds provide a significant improvement of the standard results
(see  Remark \ref{strong bound remark}).
The proof is based on the comparison theorems from  \cite{BO}, \cite{Brat}, and
\cite{AGKL}.

\smallskip
{\it Acknowledgement:} It is our pleasure to thank Patrick Delorme for
many invaluable comments to earlier versions of this paper.

\section{Real spherical spaces}\label{notation}

A standing convention of this paper is that (real) Lie groups will be denoted 
by upper case Latin letters, e.g $A$, $B$ etc., 
and their Lie algebras by lower case  German  letters, e.g. $\af$, $\mathfrak b$ etc.  
The identity component of a Lie group $G$ will be denoted by $G_e$. 

\par Let $G$ be an algebraic  real reductive group by which we understand 
a union of connected components 
of the real points of a connected complex 
reductive algebraic group $G_\C$. Let $H<G$ be a closed 
connected subgroup such that there is a complex algebraic subgroup 
$H_\C <G_\C$ such that $H=(G\cap H_\C)_e$. Under these assumptions 
we refer to $Z=G/H$ as a real algebraic homogeneous space.  We set $Z_\C=G_\C/H_\C$ 
and note that there is a natural $G$-equivariant morphism 
$$Z\to Z_\C , \ \ gH \mapsto gH_\C\, .$$
Let us denote by $z_0=H$ the standard base point of $Z$.

\par We assume that  $Z$ is {\it real spherical}, i.e.~we assume that 
a minimal parabolic subgroup $P_{\rm min}<G$ admits an open orbit on $Z$.
It is no loss of generality to request that $P_{\rm min} H \subset G$ is open, or equivalently 
$$ \gf  = \pf_{\rm min} + \hf \, .$$

\par If $L$ is a real algebraic group, then we denote by $L_{\rm n}\triangleleft L$ the connected 
normal subgroup generated by all unipotent elements.

\par According to the local structure theorem of \cite{KKS} there exists a unique   
parabolic subgroup $Q\supset P_{\rm min}$ (called
$Z$-adapted),
with a Levi decomposition $Q=LU$ such that:

\begin{itemize}
\item $P_{\rm min}\cdot z_0 = Q\cdot z_0$. 
\item   $L_{\rm n} < Q\cap H < L$. 
\end{itemize}

\par We let $K_L A_L N_L=L$ be an Iwasawa decomposition of $L$ with 
$N_L<P_{\rm min}$, set 
$A:=A_L$ and obtain a Levi-decomposition $P_{\rm min}=MAN= MA \ltimes N$ where $M=Z_{K_L}(A)$
and $N=N_LU$. 
\par We set $A_H:= A \cap H$ and further 
$$ A_Z:= A/A_H\, .$$
The dimension of $A_Z$ is an invariant of the real spherical space, called 
its {\it real rank}; in symbols $\rank_\R(Z)$. 
Observe that it follows from the fact that $\lf_{\rm n}\subset\hf$
that 
$$\af=\af\cap\zf(\lf)+\af\cap\hf$$ 
where $\zf(\lf)$ denotes the center of $\lf$. In particular,
it follows that $\lf\cap\hf$ is $\Ad(A)$-invariant.

\par We extend $K_L$ to a maximal compact subgroup $K$ of $G$ and denote by 
$\theta$ the corresponding 
Cartan involution. Further we put 
$\oline{\nf}:= \theta(\nf)$, $\oline\uf :=\theta(\uf)$. 
Later in this article in will be convenient to
replace $K$ by $K^a:=aKa^{-1}$ for a suitable element $a\in A$.
Then $\theta$ becomes replaced by $\theta^a:=\Ad(a)\circ \theta\circ \Ad(a)^{-1}$,
and for this reason it is important to monitor the dependence
of our definitions on $\theta$. For example, we note that $M$,
$\oline{\nf}$, and $\oline{\uf}$ are unaltered by such a change.

We fix an invariant non-degenerate bilinear form $\kappa$
on $\gf$. Associated with that we have the inner product
$\langle X,Y\rangle :=\kappa(X,\theta Y)$ on $\gf$, called the 
Cartan-Killing form. Note that it depends on the Cartan involution.

\par We write $\Sigma=\Sigma(\gf, \af)\subset \af^*\bs \{0\}$ for the restricted root system attached to the pair 
$(\gf, \af)$.  For $\alpha\in \Sigma$ we denote by $\gf^\alpha$ the corresponding root space and write
$$\gf=\af \oplus\mf \oplus\bigoplus_{\alpha\in \Sigma} \gf^\alpha$$
for the decomposition of $\gf$ into root spaces. Here, as usual, 
$\mf$ is the Lie algebra of $M$.

\subsection{Examples of real spherical spaces}

If $\hf<\gf$ is a subalgebra 
such that there exists a minimal parabolic subalgebra $\pf_{\rm min}$ such that 
$\gf=\hf +\pf_{\rm min}$, then we call $(\gf, \hf)$ a {\it real spherical pair} 
and $\hf$ a  {\it real spherical subalgebra}
of $\gf$. 
\par A subalgebra $\hf<\gf$ is called {\it symmetric} if there exists an involutive automorphism 
$\tau:\gf\to\gf$ with fixed point set $\hf$.  We recall that every symmetric subalgebra is reductive 
and that every symmetric subalgebra is real spherical. Symmetric subalgebras have been classified by 
Cartan and Berger.

\par A pair $(\gf,\hf)$ of a complex Lie algebra and a complex subalgebra
is called {\it complex spherical}  or simply {\it spherical} if it is real spherical
when regarded as a pair of real Lie algebras.
Note that in this case the minimal parabolic subalgebras 
of $\gf$ are precisely the Borel subalgebras.

\begin{lemma} Let $\hf<\gf$ be a subalgebra such that $(\gf_\C, \hf_\C)$ is a complex spherical pair.  
Then $(\gf, \hf)$ is a real spherical pair. 
\end{lemma}

\begin{proof} Let $\bfrak_\C$ be a Borel subalgebra of $\gf_\C$  which is contained in $\pf_{{\rm min}, \C}$. 
We claim 
that there exists a $g\in G$ such that $\hf_\C+\Ad(g)\bfrak_\C =\gf_\C$. 
This follows immediately from the fact that 
$(\gf_\C, \hf_\C)$ is a spherical pair, since the
set of elements $g\in G_\C$, for which 
$\hf_\C+\Ad(g)\bfrak_\C=\gf_\C$, is then non-empty, Zariski open, and defined over $\R$.

Let $g\in G$ be such an element. It follows that $\hf_\C + \Ad(g) \pf_{{\rm min}, \C}=\gf_\C$,
and taking real points this  implies that $\hf +\Ad(g) \pf_{\rm min}=\gf$. 
\end{proof}

We note that the converse of the lemma is not true if $\gf$ is not quasi-split.
For example $(\gf, \nf)$ is a real spherical pair, but the 
complexification $(\gf_\C, \nf_\C)$ is not spherical unless $\gf$ is quasisplit. 
The real spherical pairs $(\gf, \hf)$ obtained from complex spherical pairs $(\gf_\C, \hf_\C)$ are called  {\it absolutely 
spherical} or {\it real forms}. 

\subsubsection{Examples of absolutely spherical pairs with $\hf$ reductive}

Complex spherical pairs $(\gf_\C, \hf_\C)$ with $\hf_\C$ reductive have been classified. For $\gf_\C$ simple this goes back 
to Kr\"amer \cite{Kr}, and it was extended to the semi-simple case  by Brion \cite{Brion1} and Mikityuk \cite{Mik}. 
For convenience we recall the non-symmetric cases of Kr\"amer's list in the table below.

\begin{table}[h]
\centering 
\begin{tabular} {|c | c|}  
\hline
\multicolumn{2}{|c|} {\tt The non-symmetric cases of Kr\"amer's list}\\
\hline
$\gf_\C$ & $\hf_\C$ \\
\hline
\hline
$\sl(n,\C) $ & $\sl(p,\C) \oplus\sl(n-p,\C),\, 2p\neq n$ \\
$\sl(2n+1,\C)$ & $\symp(n,\C)\oplus \C$\\
$\sl(2n+1,\C)$ & $\symp(n,\C)$\\ 
\hline 
$\sof(2n+1,\C)$ & $\glf(n,\C)$\\
$\sof(9,\C)$   & $\spin(7,\C)$  \\
$\sof(7,\C)$   & $G_2$\\
\hline 
$\symp(2n,\C)$ & $\symp(n-1,\C)\oplus\C$\\ 
\hline
$\sof(2n,\C)$ & $\sl(n,\C),\, n$ odd\\
$\sof(10,\C)$ & $\sof(2,\C) \oplus \spin(7,\C)$\\
$\sof(8,\C) $ & $G_2$\\
\hline
$G_2$ & $\sl(3,\C)$ \\ 
$E_6$ &  $\sof(10,\C)$\\ 
\hline
\hline
\end{tabular}
\end{table}

\bigskip

The pairs in the table feature plenty of non-compact real forms, classified in \cite{KKPS} and \cite{KKPS2}.
For example the pairs $(\sl(2n+1,\C), \symp(n,\C))$, 
$(\sof(2n+1,\C), \glf(n,\C))$ and  $(\sof(7,\C), G_2)$
have the following non-compact real forms: 
\begin{eqnarray}
\label{NonWF1}& &(\suf(2p,2q+1), \symp(p,q)) \quad \hbox{and} \quad  (\sl(2n+1,\R), \symp(n,\R))\\
\label{NonWF2}& &(\sof(n,n+1), \glf(n,\R))\\
& &(\sof(3,4), \sG_2^1) \quad\hbox{with $\sG_2^1$ the split real form of $G_2$}.\end{eqnarray}

From the list of irreducible complex spherical pairs $(\gf_\C,\hf_\C)$ with $\gf_\C$ non-simple 
(see \cite{Brion1}, \cite{Mik}) we highlight the Gross-Prasad cases: 
\begin{eqnarray}
\label{GP1}& & (\sl(n+1,\C) \oplus \sl(n,\C), \glf(n,\C)), \\ 
\label{GP2}& & (\sof(n+1,\C) \oplus \sof(n,\C), \sof(n,\C))\end{eqnarray}
which are ubiquitous in automorphic forms  \cite{GP}. 

\subsubsection{Real spherical pairs which are not absolutely spherical}

Prominent examples are constituted by the triple spaces 
$(\gf, \hf)= (\hf \times\hf\times\hf, \diag \hf)$ for $\hf=\sof(1,n)$,
which are real spherical for $n\ge 2$ but not  absolutely spherical when $n\ge 4$
(see \cite{BR}, \cite{DKS}).  This example will be discussed later in the context 
of Levi-induced spaces in Section \ref{triple spaces}. 

Other interesting examples for $\gf$ simple are (see \cite{KKPS}): 

\begin{itemize}
\item $(\sE_7^2, \sE_6^2),(\sE_7^2, \sE_6^3), (\sE_6^4, \sl(3,\Hb))$,
\item $ (\sl(n,\Hb), \sl(n-1,\Hb))$,  $( \sof(2p, 2q), \suf(p,q))$ for $p\neq q$,
\item $(\sof(6,3), \sof(2) \oplus \sG_2^1), (\sof(7,4), \spin(4,3)+\sof(3))$.
\end{itemize}

\subsubsection{Non-reductive examples}
We begin with a general fact (see \cite{Brion1}, Prop. 1.1 for a slightly weaker statement in the complex case). 

\begin{prop} Let $P< G$ be a parabolic subgroup and $H<P$ an algebraic subgroup. Let $P=L_P \ltimes U_P$ 
be a Levi-decomposition of $P$. Then the following statements are equivalent: 
\begin{enumerate}
\item\label{1rst} $Z=G/H$ is real spherical. 
\item\label{2ond} $P/H$ is an $L_P$-spherical variety, i.e.~the action of a  minimal parabolic subgroup of $L_P$ admits 
an open orbit on $P/H$. 
\end{enumerate}
\end{prop}
\begin{proof} Let $P^{\rm opp}<G$ be the parabolic subgroup
with $P^{\rm opp}\cap  P= L_P$, and let $P_{\rm min}<P^{\rm opp}$ be a minimal 
parabolic subgroup of $G$. 
By the Bruhat decomposition $P_{\rm min}P$ is the only open
double $P_{\rm min}\times P$-coset in $G$. Moreover, 
$S:=P_{\rm min}\cap P$ is a 
minimal parabolic subgroup of $L_P$. 
 
\par Assume (\ref{1rst}) and 
let $g\in G$ be such that $P_{\rm min}gH$ is open in $G$.
Then $P_{\rm min}gP$ is also open in $G$, and hence 
$g\in P_{\rm min}P$. We may thus assume $g\in P$.
Then $P_{\rm min}gH\cap P=SgH$ is open in $P$, proving
(\ref{2ond}).

Assume (\ref{2ond}) and 
let $p\in P$ be such that $SpH$ is open in $P$.
Then $P_{\rm min}pH=P_{\rm min}SpH$ is open in $P_{\rm min}P$
and in $G$, proving (\ref{1rst}). 
\end{proof}

\begin{ex} We consider the group $G=\SU(p,q)$ with $3\leq p\leq q$. This group has real rank $p$ 
with restricted root system $BC_p$ or $C_p$ (if $p=q$). Let $P=P_{\rm min}=MAN$ be a minimal subgroup. 
Let $N_0<N$ be the subgroup with Lie algebra 
$$\nf_0:= \bigoplus_{\alpha\in \Sigma^+\atop \alpha \neq \e_1-\e_2, \e_2-\e_3}\gf^\alpha\, .$$
Let $Z(M)$ be the center of $M$. Then the proposition shows that $H:=Z(M)AN_0$ is a real spherical subgroup. 
In case  $q-p>1$ it is not absolutely spherical.
\end{ex}

\section{Degenerations of a real spherical subalgebra}

\subsection{The compression cone}

We recall that the local structure theorem (see \cite{KKS})
implies that $\gf=\hf+\qf$ and hence
\begin{equation}\label{lst-decomp}
\gf=
\hf\oplus (\lf\cap\hf)^{\perp_\lf} \oplus\uf.
\end{equation}
Here we use the Cartan-Killing form of $\gf$ restricted 
to $\lf$ to define $\perp_\lf$. Note that since $\lf\cap\hf$
is $\Ad(A)$-invariant, then so is its orthocomplement.
In particular, the decomposition (\ref{lst-decomp})
will not be affected by our later conjugation of the Cartan involution
by an element from $A$.

We define the linear 
operator $T: \oline \uf \to   
(\lf \cap \hf)^{\perp_\lf}\oplus \uf\subset \pf_{\rm min}$
as minus the restriction of the projection
along $\hf$, according to (\ref{lst-decomp}).  Then
\begin{equation}\label{LST} 
\hf = \lf \cap \hf\,\oplus\,\mathcal G(T)=\lf \cap \hf\,\oplus\,\{ \oline X + T(\oline X)\mid \oline X \in \oline\uf\}.
\end{equation}

\par Write $\Sigma_\uf$ for the space of $\af$-weights of the $\af$-module $\uf$. 
Let $\alpha\in \Sigma_\uf$ and let $X_{-\alpha}\in\gf^{-\alpha}\subset
 \oline{\uf}$. 
Then 
\begin{equation}\label{T}  T(X_{-\alpha}) = \sum_{\beta\in \Sigma_\uf\cup \{0\}} X_{\alpha, \beta}\end{equation}
with $X_{\alpha, \beta} \in \gf^{\beta}\subset \uf$ a root vector for $\beta\neq 0$ and the 
convention that $X_{\alpha, 0}\in (\lf \cap \hf)^{\perp_\lf}$. 
Let ${\mathcal M}\subset \N_0[\Sigma_\uf]$ be the monoid (additive semi-group with zero)
generated by 
$$\alpha +\beta \qquad (\alpha\in\Sigma_\uf, \beta\in \Sigma_\uf\cup\{0\}, \exists 
X_{-\alpha}\in\gf^{-\alpha}\ \hbox{such that}\  X_{\alpha, \beta}\neq 0)\, .$$  

Note that elements of ${\mathcal M}$ vanish on  $\af_H$ so that 
${\mathcal M}$ is naturally a subset of $\af_Z^*$. 

\par We define the compression cone of $Z$ to be 
$$\af_Z^-:=\{ X\in \af_Z\mid (\forall \alpha\in {\mathcal M}) \ 
\alpha(X)\leq 0\}$$
which is a closed convex cone in $\af_Z$ with non-empty interior. 
 
\subsubsection{Limits in the Grassmannian}
We recall from \cite{KKSS} Lemma 5.9 the following property of the 
compression cone. 
Let $\af_Z^{--}$ be the interior of $\af_Z^-$ and let
$$\hf_{\rm lim}:= \lf \cap \hf + \oline{\uf}\, .$$
Note that $d:=\dim \hf =\dim \hf_{\rm lim}$ and that $\hf_{\rm lim}$ is a real spherical 
subalgebra. Then $X\in \af_Z^{--}$ if and only if 
\begin{equation} \label{compress} \lim_{t\to \infty} e^{t \ad X} \hf = \hf_{\rm lim} \end{equation} 
holds in the Grassmannian ${\rm Gr}_d(\gf)$ of $d$-dimensional 
subspaces of $\gf$. 

\subsubsection{Limits in a representation}\label{Liar}
We recall  another description of $\af_Z^-$, which was used
as its definition in \cite{KKSS}, Def.~5.1 and Lemma 5.10. 

Consider an irreducible finite
dimensional real representation $(\pi, V)$ with $H$-semi-spherical
vector $0\neq v_H\in V$, that is there is an algebraic character 
$\chi $ of $H$ such that $\pi(h)v_{H} = \chi(h)v_{H}$ for all $h\in H$.
Let $\R^+ v_0$ be a lowest weight ray which is stabilized by $\oline Q$. 
Let $\af_{\pi,\chi}^{--}$ be the open cone in $\af_Z$ defined by the property: 
$X\in \af_{\pi,\chi}^{--}$ if and only if 
\begin{equation} \label{pi limit} \lim_{t\to \infty} [\pi(\exp(tX))\cdot v_H] = [v_0] \end{equation} 
holds in the projective space  ${\mathbb P} (V)$. 
We denote by $\af_{\pi,\chi}^-$ the closure of $\af_{\pi,\chi}^{--}$ and record 
that $\af_Z^- \subset \af_{\pi,\chi}^-$.
Moreover,  we have 
\begin{equation*} 
\af_Z^- =\af_{\pi, \chi}^- \quad \hbox{iff $\pi$ is regular}\, .\end{equation*}
Here $\pi$ is called regular if $\oline Q$ is the stabilizer of $\R^+ v_0$; in particular, 
the lowest weight is strictly anti-dominant with respect to the roots of $\uf$.

\subsection{Spherical roots}

Let $C$ be the convex cone spanned by $\mathcal M$. Then, according to
\cite{KK}, Cor.~12.5 and Cor.~10.9, $C$ is simplicial, that is, there exists a linearly 
independent set $S\subset \af_Z^*$ such that 
\begin{equation} \label{C-cone}C=\bigoplus_{\sigma \in S} \R_{\geq 0} \sigma\,. \end{equation}
In particular, we record 
$$
\mathfrak{a}^-_Z=\{X\in\mathfrak{a}_Z\mid(\forall\sigma\in S)\
\sigma(X)\le0\}\, .
$$
The elements of $S$, suitably normalized (see \cite{vS} for an overview on some commonly used  normalizations), are referred as {\it spherical roots} for $Z$. 
In this paper we are not very specific about the normalization of $S$ and just request that 
\begin{equation}  \label{S normalization} {\mathcal M} \subset \N_0[S]\end{equation}
is satisfied.  We note that the "maximal length" normalization of \cite{KK} (11.4)
fulfills this. We also note that 
(\ref{S normalization}) implies
\begin{equation}\label{S rational}
S\subset \Q_{\ge 0}[\Sigma_\uf].
\end{equation}
To see this let $\sigma\in S$. Then the ray $\R_{\ge 0}\sigma$ is extreme in $C$,
hence spanned by some $\gamma\in\M$. Thus $\sigma=c\gamma$
for some $c>0$. It now follows from (\ref{S normalization}) that $1/c$ is an 
integer and hence (\ref{S rational}) holds.

By slight abuse of common terminology we will henceforth call any set $S$ which satisfies 
(\ref{C-cone}) and (\ref{S normalization}) 
{\it a set of spherical roots for $Z$.} Let us now fix such a choice.

\par Given a closed convex cone $C$ in a finite dimensional real 
vector space, we call $E(C):= C\cap -C$ the edge of $C$; it is the largest 
vector subspace of $V$ which is contained in $C$. 

\par We are now concerned with the edge $\af_{Z,E}:=E(\af_Z^-)$ of $\af_Z^-$. By our 
definition of $\af_Z^-$ we have 
\begin{equation*}
\af_{Z,E} = \{ X\in \af_Z\mid 
(\forall \alpha\in S)\, \alpha(X)=0\}\,.
 \end{equation*}
It is immediate from (\ref{T}) that $\af_{Z, E}$ is contained in $N_\gf(\hf)$, the normalizer 
of $\hf$ in $\gf$. In this context it is good to keep in mind that 
$N_G(\hf)/ H A_{Z, E}$ is a compact  group (see \cite{KKSS}).

Let $e:= \dim \af_{Z,E}$, $r:=\rank_\R(Z)=\dim \af_Z$ and 
$s:=\# S$.  
Then $\mathfrak{a}^-_Z/\mathfrak{a}_{Z,E}$ is a
simplicial cone with 
\begin{equation*} 
s=r-e\, 
\end{equation*}
generators.

\begin{ex} Let $H=\oline{N}$. This is a spherical subgroup. In this 
case $\hf_{\rm lim}=\hf$,  ${\mathcal M}=\{0\}$, $S=\emptyset$, 
and $\af_Z^-=\af_{Z,E}=\af$.
\end{ex}

\subsection{Boundary degenerations}

For each subset $I\subset S$ we choose an element $X=X_I\in\af_Z^-$ with
$\alpha(X)=0$ for all $\alpha\in I$ and $\alpha(X)<0$ for all 
$\alpha\in S\bs I$. Then we define 
\begin{equation}\label{defi h_I}
\hf_I:=\lim_{t\to \infty} e^{\ad tX} \hf
\end{equation}
with the limit taken in the Grassmannian ${\rm Gr}_d(\gf)$ as in 
(\ref{compress}).
In particular, 
$\hf_\emptyset=\hf_{\rm lim}$ and $\hf_S=\hf$.

\par To see that the limit exists we recall the explicit 
description of $\hf$ in (\ref{T}).
Let $\la I\ra\subset \N_0[S]$ be 
the monoid generated by $I$. Within the notation of (\ref{T}) we set 
$X_{\alpha,\beta}^I:=X_{\alpha,\beta}$ if $\alpha+\beta\in \la I\ra$ 
and zero otherwise. 
Let $\uf_I\subset \uf$ be the subspace spanned by all $X_{\alpha,\beta}^I$ and define a linear operator 
$$T_I: \oline \uf \to   (\lf \cap \hf)^{\perp_\lf}\oplus \uf_I$$
by 
\begin{equation} \label{T_I}  
T_I(X_{-\alpha}) = \sum_{\beta\in \Sigma_\uf\cup\{0\}} X_{\alpha,\beta}^I\, .\end{equation}
In particular, $T_\emptyset=0$ and $T_S=T$.
Now observe that 
\begin{equation}\label{t-compr} 
e^{t\ad X} ( X_{-\alpha}+ T(X_{-\alpha}) )
= e^{-t\alpha(X)} ( X_{-\alpha}+ \sum_\beta e^{t (\alpha(X)+\beta(X))}
X_{\alpha, \beta})
\end{equation}
from which we infer that the limit in (\ref{defi h_I}) is given by
\begin{equation}\label{expression h_I}
\hf_I = \lf \cap \hf +\mathcal G(T_I)=\lf \cap \hf + \{ \oline X + T_I(\oline X)\mid \oline X \in \oline\uf\}\, ,
\end{equation} 
and in particular, it is thus independent of the choice of element $X_I$.

Let $H_I<G$ be the connected subgroup of $G$ corresponding to 
$\hf_I$.
We call $Z_I:=G/H_I$ the {\it boundary degeneration} of $Z$ attached 
to $I\subset S$, and summarize its basic properties as follows.

\begin{prop}\label{properties of Z_I} Let 
 $I\subset S$. Then
\begin{enumerate} 
\item\label{eins} $Z_I$ is a real spherical space. 
\item\label{drei} $Q$ is a $Z_I$-adapted parabolic subgroup.  
\item\label{zwei} $\af\cap\hf_I=\af\cap\hf$ and $\rank_\R Z_I =\rank_\R Z$.   
\item\label{vier} $I$ is a set of spherical roots for $Z_I$.
\item\label{funf} $\af_{Z_I}=\af_Z$ and 
$\af_{Z_I}^-=\{ X\in\af_Z\mid (\forall\alpha\in I)\, \alpha(X)\leq 0 \}$.
\end{enumerate}
\end{prop}

\begin{proof} It follows from (\ref{defi h_I})
that $\hf_I$ is algebraic, and from (\ref{expression h_I})
that  $\hf_I+\pf_{\min} =\gf$. Hence (\ref{eins}).
The statements (\ref{drei})-(\ref{vier}) all follow easily from 
(\ref{expression h_I}), and (\ref{funf}) is a consequence of (\ref{zwei}) and (\ref{vier}).
\end{proof}

The boundary degeneration $Z_I$ admits  non-trivial
automorphisms when $I\neq S$.
Set 
\begin{equation*} 
\af_I:=\{ X\in \af_Z\mid (\forall \alpha\in I) \alpha(X)=0\}\, ,\end{equation*}
then we see that $A_I$ acts by $G$-automorphisms of $Z_I$
from the right. 

\par In the sequel we realize $\af_Z$ as a subspace of $\af$ via the identification $\af_Z=\af_H^{\perp}$.
Likewise we view $A_Z$ as a subgroup of $A$.

It is then immediate from the definitions that: 
\begin{eqnarray} 
\label{ai1}
&&\af_S=\af_{Z,E}\subset \af_I=\af_{Z_I,E} 
\subset \af_\emptyset=\af_Z\,, \\
\label{ai2}
&&[\af_I+\af_H, \hf_I]\subset \hf_I\,.
\end{eqnarray}

\begin{rmk}\label{Z_I abs sph} If $Z$ is absolutely spherical 
then so are all $Z_I$.  Indeed, let $(\gf,\hf)$ be absolutely spherical with complex spherical 
complexification $(\gf_\C,\hf_\C)$.  Then the compression cones for  $(\gf,\hf)$ and 
$(\gf_\C,\hf_\C)$ are compatible (see \cite{KK}, Prop. 5.5 (ii)) in the obvious sense. 
From that the assertion follows.
\end{rmk}

\begin{ex} \label{symmetric space case}
If $G/H$ is a symmetric space for an involution $\sigma$, 
that is, $H$ is the connected component 
of the fixed point group of an 
involution $\tau$ on $G$ 
(which we may assume commutes with $\theta$), then
$$Z_I=G/(L_I\cap H)_eN_I$$ 
where $P_I=L_IN_I$ is a $\tau\theta$-stable parabolic subgroup
with $\tau$- and $\theta$-stable Levi part $L_I$. 
\end{ex}

\subsection{Polar decomposition}\label{poldecsec}

The compression cone $\af_Z^-$ of $Z$ determines the large scale behaviour of $Z$. 
In \cite{KKSS} we obtained a polar decomposition of a real spherical space. Our concern here 
is to obtain polar decompositions for all spaces $Z_I$ in a uniform way. For that it is 
more convenient to use standard compactifications of $Z$  (see \cite{KK}) rather then the simple
compactifications from \cite{KKSS}. 

\par For a real spherical subalgebra $\hf<\gf$ we set $\hat\hf:= \hf+\af_{Z,E}$. Note that $\hf\triangleleft \hat \hf$ 
is an ideal. 
We denote by $\hat H_{\C,0}$ the connected algebraic 
subgroup of $G_\C$ with Lie algebra $\hat \hf_\C$ and set $\hat H_0 :=\hat H_{\C, 0}\cap G$. 
More generally, let $\hat H_\C$ be some complex algebraic subgroup of $G_\C$ with 
Lie algebra $\hat \hf_\C$, and let  
$\hat H= G \cap \hat H_\C$. Then $\hat H_0$ and $\hat H$ both have Lie algebra $\hat\hf$, and 
$\hat H_0\triangleleft \hat H$ is a normal subgroup.

\par Further we set $\hat \hf_I =\hf_I+\af_I$ for each $I\subset S$ and note that $\hat \hf_S= \hat \hf$
and $\hat\hf_\emptyset=\hf_{\lim}+\af_Z$. 
Recall the element $X_I\in\af_I\cap \af_Z^-$ and set for $s\in \R$ 
$$a_{s, I}:= \exp(sX_I)\in A_I\, .$$

\par Let $\hat Z=G/\hat H$. We first describe the basic structure of a standard compactification $\oline Z$ for 
$\hat Z$.  There exists a finite dimensional  real representation $V$ of $G$ with an $\hat H$-fixed 
vector $v_{\hat H}$ such that 
$$ \hat Z \to {\mathbb P}(V), \ \ g\cdot \hat z_0 \mapsto [ g\cdot v_{\hat H}]$$
is an embedding and $\oline Z$ is the closure of $\hat Z$ in the projective space ${\mathbb P}(V)$. 
Moreover, $\oline Z$ has the following properties: 
{\renewcommand{\theenumi}{\roman{enumi}}%
\begin{enumerate}
\item\label{romani} The limit $ \hat z_{0,I}=\lim_{s\to \infty} a_{s,I} \cdot \hat z_0$ exists for every $I\subset S$,
and the stabilizer $\hat H_I$ of $\hat z_{0,I}$ is an algebraic group with Lie algebra $\hat\hf_I$. 
\item $\oline Z$ contains the unique closed orbit $Y= G\cdot \hat z_{0,\emptyset}$. 
\end{enumerate}}
\par Note that  $\hat H_{I}\supset \hat H_{I,0}$. The inclusion
can be proper, also if we choose $\hat H =\hat H_0$. 
This is the reason why we need to consider more general algebraic subgroups $\hat H$ than $\hat H_0$.

\par In the next step we explain the polar decomposition for $\hat Z$ and derive from that a polar 
decomposition for $Z$. 

In order to do that we recall the description of the open $P_{\rm min} \times \hat H$ double cosets
of 
$G$ from \cite{KKSS}, Section 2.4. We first treat the case of $\hat H_0$.
Every open double coset of  $P_{\rm min} \times \hat H_0$  has a representative
of the form 
\begin{equation} \label{th-form} 
w=th, \qquad (t\in T_Z=\exp(i\af_Z), h\in \hat H_{\C,0})\, .\end{equation}
This presentation is unique in the sense that if 
$t'h'$ is another such representative of the same double coset, 
then 
there exist  $f\in T_Z\cap \hat H_{\C,0}$ and $h''\in \hat H_0$ such that  
$t'=t f$ 
and $h'=f^{-1} h h''$. We let 
$$\F=\{ w_1, \ldots, w_k\}\subset G$$
be a minimal set of representatives of the open $P_{\rm min} \times \hat H_0$-cosets
which are of the form (\ref{th-form}). 

The map $w\mapsto P_{\rm min}w\hat H$ is surjective from $\F$ onto of 
open $P_{\rm min}\times \hat H$ cosets in $G$. We let 
$$\hat \F=\{ \hat w_1, \ldots, \hat w_m\}\subset \F$$
be a minimal set of representatives of these cosets. 
Note that every $\hat w\in\hat \F$ allows a presentation $\hat w=th$ as in (\ref{th-form}),
which is then unique  in the sense that
if $t'h'\in P_{\rm min} \hat w\hat H$ is another such representative, then 
$t'=tf$ and $h'= f^{-1} h h''$
for some  $f\in T_Z\cap \hat H_\C$ and $h''\in \hat H$. 
We observe the following relations
on the $T_Z$-parts:
\begin{equation*} 
\{ \hat t_1, \ldots, \hat t_m\} (T_Z\cap \hat H_\C) 
= \{ t_1, \ldots, t_k\} (T_Z\cap \hat H_\C)\, .
\end{equation*}

With that notation the polar decomposition for $\hat Z= G/\hat H$ is obtained as in 
\cite{KKSS}, Th. 5.13: 
\begin{equation}\label{poldechat}
\hat Z = \Omega A_Z^- \hat \F\cdot \hat{z_0}\, .
\end{equation}
Here $A_Z^-=\exp(\af_Z^-)$ and $\Omega\subset G$ is a compact subset which is of the form 
$\Omega = \F''K $ with $\F''\subset G$ 
a finite set. 

\par For $G/\hat H_0$ the polar decomposition (\ref{poldechat})
can be rephrased as $G= \Omega A_Z^- \F \hat H_0$. 
From the fact that $\hat H_{\C,0}$ is connected we infer 
\begin{equation} \label{H-normalizer} 
\hat H_0< N_G(H)\, .\end{equation}
Let $\F'\subset \hat H_0$ be a minimal set of representatives 
of the finite group $\hat H_0 / HA_{Z,E}$ (observe that $HA_{Z,E}$ is 
the identity component of $\hat H_0$). Note that $\F'$ is in the normalizer of $H$ by (\ref{H-normalizer}).  
We then record the obvious decomposition: 
\begin{equation}\label{Hnot} 
\hat H_0 = A_{Z,E} \F' H \, .\end{equation}

\par Note that since $A_{Z,E}$ is connected 
the open $P_{\rm min} \times H$ double cosets in $G$ are identical to the open 
$P_{\rm min}\times A_{Z,E}H$  double cosets. Hence $\W:= \F \F'\subset G$ is a (not necessarily minimal)
set of representatives for all open $P_{\rm min} \times H$-double cosets in $G$.

The next lemma guarantees that we can slide $A_{Z,E}$ past $\W$. 

\begin{lemma} \label{F-set} Let $w\in \W$. Then there exist for all $a\in A_{Z,E}$ an element $h_a\in H$ such that 
$a^{-1}wa = wh_a$. In particular, $\W A_{Z,E} \subset A_{Z,E} \W H$. 
\end{lemma}

\begin{proof} Let $w=th \in \W$ with $t\in T_Z$ and $h\in \hat H_{\C,0}$.  For 
$a\in A_{Z,E}$ the element $a^{-1} wa$ represents the same open double $P_{\min}\times \hat H_0$ coset 
as $w$. Further note that $a^{-1} wa =  t a^{-1} h a$. We infer from the uniqueness
of the presentation $w=th$ (as a representative of an open $P_{\min} \times \hat H_0$ double coset) 
that there exists an element $h_a\in \hat H_0$ such that 
$a^{-1} ha = h h_a$.  Note that $\hat H_{\C,0} = A_{Z,E,\C}H_\C $ and therefore we 
can decompose 
$ h = bh_1$ with $b\in A_{Z,E, \C}$, $h_1\in H_\C$. It follows
that $a^{-1} h_1 a =h_1 h_a$. Hence $h_a\in \hat H_0 \cap H_\C$ and as $h_a$ varies continuously 
with $a$ we deduce that it belongs to $H$. 
\end{proof} 

Observe that $A_{Z,E} \subset A_Z^-$.  Thus putting (\ref{poldechat}), (\ref{Hnot}) and Lemma \ref{F-set} together 
we arrive at the polar decomposition for $Z$:
\begin{equation}\label{poldec} Z = \Omega A_Z^- \W\cdot z_0\, .\end{equation} 

\begin{rmk} \label{symmsp} 
Consider the case where $Z=G/H$ is a symmetric space as in
Example \ref{symmetric space case}. 
We choose $\af$ such that it is $\tau$-stable and such that the 
$-1$-eigenspace $\af_{pq}$ of $\tau$ on $\af$ is maximal. 
Then $\af_Z=\af_H^\perp=\af_{pq}$
and  the set of $P_{\rm min} \times H$ open double cosets
is naturally identified with a quotient of Weyl-groups: $W_{pq} / W_{H\cap K}$
where $W_{pq}= N_K(\af_{pq}) / Z_K(\af_{pq})$ and 
$W_{H\cap K} = N_{H\cap K} (\af_{pq}) / Z_{H\cap K}(\af_{pq})$
(see \cite{Ross}, Cor.~17).
Moreover, in this case (\ref{poldec}) is valid with $\Omega=K$
(see \cite{Mogens}, Thm.~4.1).
\end{rmk} 

For $g\in G $ we set $\hf_g:= \Ad(g)\hf$ and $H_g=gHg^{-1}$, and note that
if $P_{\rm min} gH$ is open then 
$Z_g=G/H_g$ is a real spherical space. In particular this applies when
$g\in\W$.

\begin{lemma}\label{hf} Let $w=th\in \W$ with $t\in T_Z$ and $h\in\hat H_{\C,0}$. Then
$$\hf_w = \lf\cap \hf  + {\mathcal G}(T_w),$$ where $T_w: \oline \uf \to \uf + (\lf \cap \hf)^\perp$ is a linear map with 
\begin{equation} \label{Tw} 
T_w(X_{-\alpha}) = \sum_\beta \e_{\alpha, \beta}(w)  X_{\alpha, \beta}\end{equation} 
in the notation from (\ref{T}), and
where $\e_{\alpha, \beta} (w)= t^{\alpha+\beta}\in \{-1, 1\}$. 
\end{lemma}
\begin{proof} We first observe that 
\begin{equation}\label{hf-eq}
\hf_w =\Ad(t)\hf_\C\cap\gf.
\end{equation}
Now by (\ref{LST}) an arbitrary element $X\in \hf_\C$ can be uniquely written as 
$$X=X_0 + \sum_{\alpha\in \Sigma_\uf} c_\alpha (X_{-\alpha} + \sum_{\beta\in\Sigma_\uf\cup\{0\}} X_{\alpha, \beta})$$
with $X_0\in (\lf\cap \hf)_\C$ and with coefficients $c_\alpha\in \C$.
Hence 
$$\Ad(t) X= X_0 + \sum_{\alpha \in \Sigma_\uf}  t^{-\alpha} c_\alpha (X_{-\alpha} + 
\sum_{\beta\in\Sigma_\uf\cup\{0\}} t^{\alpha+\beta} X_{\alpha, \beta})\, .$$
We conclude that $\Ad(t)X\in \gf$ if and only if $X_0\in \lf\cap \hf$, $c_\alpha t^{-\alpha}\in \R$  
for all $\alpha$
and $t^{\alpha+\beta}\in \R$ for all $\alpha,\beta$, that is $t^{\alpha+\beta}\in \{-1, 1\}$. 
\end{proof}

\begin{cor}\label{Zw}
Let $w\in\W$. Then
\begin{enumerate} 
\item $Q$ is the $Z_w$-adapted parabolic subgroup,  
\item $\af_Z^-$ is the compression cone for $Z_w$. 
\end{enumerate}  
\end{cor} 

\begin{proof} Immediate from Lemma \ref{hf}.\end{proof}

\subsubsection{The sets $\F$ and $\W$ for the boundary degenerations}\label{FI-sub}
Let  $I\subset S$. We define $\hat H_I$ as in (\ref{romani}),
with the assumption
$\hat H=\hat H_0$, and set
$$\hat Z_I:= G/\hat H_I= G \cdot \hat z_{0,I}.$$ 
We wish to construct a set $\F_I$ of representatives of open 
$P_{\rm min} \times \hat H_I$
double cosets in $G$, analogous to the previous set $\F$ for $P_{\rm min} \times \hat H_0$.
Recall that possibly $\hat H_{I,0}\subsetneq \hat H_I$.

Notice that the $\hat Z_I$-adapted parabolic subgroup
is $Q$ and that $L_{\rm n} \subset L\cap \hat H_I$. The local structure theorem 
for $\hat Z_I$ then implies that $Q\times_L (L/L\cap \hat H_I) \to \hat Z_I$ is an open immersion 
onto the $P_{\rm min} $-orbit $P_{\rm min} \cdot \hat z_{0,I}$. Moreover $A_{Z_I}= A_Z/A_I$. 
Realize $\af_{Z,I}\subset \af_Z$ via $\af_I^{\perp_{\af_Z}}$ and set $T_{Z_I}=\exp(i\af_{Z_I})$. Let 
$\hat \F_I$ be a minimal set of representatives 
of the open $P_{\min}\times \hat H_I$ double cosets
which are of the form $\hat w_I= \hat t_I \hat h_I \in \hat \F_I$, $\hat t_I \in T_{Z_I}$, $\hat h_I \in \hat H_{I,\C}$. 
As before we also have a minimal set $\F_I$ of representatives for the open cosets for the smaller
group $P_{\min}\times \hat H_{I,0}$, such that
$$\hat \F_I\subset \F_I= \{ w_{1,I}, \ldots, w_{k_I,I}\}$$
where $w_{j,I} = t_{j,I} h_{j,I}$ with $t_{j,I} \in T_{Z_I}$ and $h_{j,I}\in \hat H_{I,\C, 0}$. 

\par Finally in analogy to $\F'$ we choose $\F_I'$
as a minimal set of representatives for $\hat H_{I,0}/H_IA_I$,
and set $\W_I:=\F_I \F_I'$.  
Then the polar decomposition of 
$Z_I$ is given by 
\begin{equation*} 
Z_I = \Omega A_{Z_I}^- \W_I \cdot z_{0,I}\, .\end{equation*} 
with $\Omega\subset G$  a compact subset of the form $\F_I''K$ for a finite set $\F_I''\subset G$.

\subsection{Relating $\W_I$ to $\W$}

We start with a general lemma: 

\begin{lemma} \label{doublecoset}Let $Z=G/H$ be a real spherical space and 
$g\in G$ be such that $P_{\min}gH_I$ is open in $G$. Then there exists 
$s_0>0$ such that $P_{\min}g a_{s,I}H$ is open and equal to
$P_{\min}ga_{s_0,I}H$ for all $s\ge s_0$.
\end{lemma}

\begin{proof} 
If there were a sequence $s_n>0$ tending to infinity such that $\pf + \Ad(ga_{s_n,I})\hf \subsetneq \gf$,
then $\lim_{n\to \infty} (\pf +\Ad(g)\Ad(a_{s_n,I})\hf) = \pf +\Ad(g)\hf_I$ would be a subspace 
of $\gf$ with positive co-dimension, which contradicts the assumption on $g$.
Hence $P_{\min}g a_{s,I}H$ is open for all $s\ge s_0$, for some $s_0$.
By continuity this implies the sets are equal.
\end{proof}

Fix an element $w_I\in \W_I$ and observe that $P_{\min}w_I H_I$ is open in $G$. 
Lemma  \ref{doublecoset} then gives an element $w\in \W$ and an $s_0>0$ 
such that $P_{\min}w_I a_{s,I} H = P_{\min}wH$ for all $s\geq s_0$. 
We say that $w$ {\it corresponds} to $w_I$ but note that $w$ is not necessarily unique. 

\par  

\begin{lemma} \label{wwI} Let $w_I\in\W_I$ and let
$w\in \W$ correspond to $w_I$. With $s_0>0$ as above there exist for each $s\ge s_0$
elements $u_s\in U$, $b_s\in A_Z$, $m_s \in M$ and $h_s\in H$,
each depending continuously on $s\ge s_0$, such that 
\begin{enumerate}
\item\label{first} $w_I a_{s,I} = u_s b_s m_s w h_s$.
\end{enumerate}
Moreover,
\begin{enumerate}
\setcounter{enumi}{1}
\item \label{first2} the elements $u_s$ and $b_s$ are unique and depend analytically on $s$, 
\item\label{second} $\lim_{s\to \infty}  (a_{s,I} b_s^{-1})=\1$, 
\item\label{third} $\lim_{s\to \infty} u_s =\1$,
\item\label{fourth} $m_s$ can be chosen such that $\lim_{s\to \infty} m_s $ exists in $M$.  
\end{enumerate}
\end{lemma}

\begin{proof}  By Corollary \ref{Zw}, the map 
$$ U \times A_Z \times M/ M\cap H_w \to P_{\rm min} w\cdot z_0, \ \ (u, a, m)\mapsto uamw\cdot z_0$$
is a diffeomorphism (local structure theorem for $Z_w$). 
As $w_Ia_{s,I} \in P_{\rm min} w\cdot z_0$ for $s\geq s_0$, 
this gives (\ref{first}) and (\ref{first2}).

\par After enlarging $G$ to $G\times\R^\times$ we can, via the affine cone construction 
(see \cite{KKS}, Cor. 3.8),  assume that $Z=G/H$ is quasi-affine. 

Let us denote by $\Gamma$ the set of (equivalence classes) of finite dimensional 
irreducible $H$-spherical and $K$-spherical representations. To begin with we recall 
a few facts from Section \ref{Liar} and from \cite{KKSS}. 

\par For $(\pi,V) \in \Gamma$ we denote 
by $\lambda_\pi\in \af^*$ its highest weight. 
Let $(\pi, V) \in \Gamma$ and $0\neq v_H \in V$ an $H$-fixed vector which we expand 
into $\af$-eigenvectors: 
$$ v_H = \sum_{\nu\in \Lambda_\pi} v_{-\lambda_\pi +\nu}\, .$$
Here $\Lambda_\pi \subset \N_0\Sigma_u$ is such that $-\lambda_\pi  +\Lambda_\pi$ is 
the $\af$-weight spectrum of $v_H$. Note that $v_{-\lambda_\pi}$ is a lowest weight vector 
which is fixed by $M$. 

\par As $\Lambda_\pi|_{\af_Z^-}\leq 0$ (see \cite{KKSS}, Lemma 5.3) we deduce that the limit 
$$ v_{H,I}:= \lim_{s\to \infty}  a_{s,I}^{\lambda_\pi} \pi (a_{s,I}) v_H$$
exists. Moreover with 
$\Lambda_{\pi, I}:=\{ \nu\in \Lambda_\pi\mid \nu (X_I)=0\}$
we obtain that 
$v_{H,I}= \sum_{\nu\in \Lambda_{\pi, I}} v_{-\lambda_\pi+\nu}$. 
Note that $v_{H,I}$ is $H_I$-fixed. 

Let $w=th$ and $w_I=t_Ih_I$ with our previous notation. From (\ref{first}) 
we obtain 
\begin{equation} 
a_{s,I}^{\lambda_\pi} \pi (t_I h_Ia_{s,I}) v_H = a_{s,I}^{\lambda_\pi} \pi (u_s m_s b_s t)v_H\, ,\end{equation} 
and hence, by passing to the limit  $s\to \infty$,
\begin{equation} \label{ww-ident3} 
\pi(t_I) v_{H,I}  = \lim_{s\to \infty} a_{s,I}^{\lambda_\pi} \pi (u_s m_s b_s t)v_H\, .\end{equation}

Let $v^*\in V^*$ be a highest weight vector in the dual representation
and apply it to (\ref{ww-ident3}). 
Since  $v^*(v_H)=v^*(v_{H,I})=v^*(v_{-\lambda_\pi})\neq 0$ 
we get 
$$ t_I^{-\lambda_\pi}=     t^{-\lambda_\pi} \lim_{s\to \infty} (a_{s,I} b_s^{-1})^{\lambda_\pi}$$
and therefore $\lim_{s\to \infty} (a_{s,I} b_s^{-1})^{\lambda_\pi}=1$. Since $Z$ is quasi-affine, it follows 
that $\{ \lambda_\pi\mid \pi \in \Gamma\}$ spans $\af_Z^*$ (see \cite{KKS}, Lemma 3.4), hence (\ref{second}). 

\par We move on to the fourth assertion. We first show that $(u_s)_s$ is bounded in $U$
when $s\to\infty$. For that, let $X_1, \ldots, X_n$ be a basis 
for $\uf$ consisting of root vectors $X_j$ with associated roots $\alpha_j$.
The map 
$$\R^n \to U, \ (x_1, \ldots, x_n)\mapsto \exp(x_n X_n)\cdot  \ldots \cdot \exp(x_1 X_1)$$
is a diffeomorphism. Let $(x_1(s), \ldots, x_n(s))$ be the coordinate vector of $u_s\in U$,
which we claim is bounded.

We fix an ordering of $\Sigma_\uf$ with the property
that if a root $\alpha$ can be expressed as a sum
of other roots $\beta$, then only roots $\beta\le\alpha$ will occur.
It suffices to show, for any given index $j$, that if $x_i(s)$ is bounded for all $i$ with
$\alpha_i<\alpha_j$, then so is $x_i(s)$ for each $i$ with $\alpha_i=\alpha_j$.

We now fix $\pi$ such that it is regular, 
that is, the highest weight $\lambda=\lambda_\pi$ satisfies
$\lambda(\alpha^\vee)>0$ for all $\alpha\in \Sigma_\uf$.
Then the map $X\mapsto d\pi(X) v_{-\lambda}$ is injective
from $\uf$ into $V$.

We compare weights vectors of weight $-\lambda+\alpha_j$
on both sides of (\ref{ww-ident3}). 
On the left side we have $t_I^{-\lambda_\pi +\alpha_j}v_{-\lambda+\alpha_j}$
if $\alpha_j\in\Lambda_{\pi,I}$, and otherwise $0$. 
By applying the Taylor expansion of $\exp$ 
we find on the other side 
$$  \lim_{s\to \infty}  a_{s,I}^{\lambda}\left(
\sum_{m,\nu} 
(b_st)^{-\lambda+\nu}  \frac {x(s)^{m}}{m!} 
d\pi( X_n)^{m_n}\cdots d\pi( X_1)^{m_1} \pi(m_s)v_{-\lambda+\nu}
\right), $$ 
where the sum extends over all multi-indices $m=(m_1,\dots,m_n)$ and all $\nu\in\Lambda_\pi$
for which $\alpha_j=\sum m_i\alpha_i+\nu$. 

Notice that by (\ref{second}) the product 
$a_{s,I}^{\lambda_\pi}(b_s)^{-\lambda+\nu}=(a_{s,I} b_s^{-1})^{\lambda-\nu}a_{s,I}^\nu$
remains bounded when $s\to\infty$. Likewise, by our assumption on the index $j$
all the terms with $m_i\neq 0$ for some $i$ with $\alpha_i\neq\alpha_j$
(and hence $m_i=0$ for all $i$ with $\alpha_i=\alpha_j$), are bounded. 
The remaining terms are those of the form 
$$a_{s,I}^{\lambda}(b_st)^{-\lambda} x_i(s)
d\pi( X_i)v_{-\lambda}$$
where $\alpha_i=\alpha_j$.
It follows by linear independence 
that $x_i(s)$ is bounded for each of these $i$ as claimed, i.e.~$(u_s)_s$ is bounded. 

Finally we show that $u_s$ converges to $\1$. Otherwise there exists $u\neq \1$ and a
sequence $s_k$ of positive numbers tending to infinity such that $u_k:=u_{s_k}\to u$. 
We may assume in addition that $m_k:=m_{s_k}$ is convergent with a limit $m$. 
We apply (\ref{first}) to $\hat z_0\in \hat Z$
$$ w_I a_{s_k,I} \cdot \hat z_0 = u_k m_k b_k t \cdot \hat z_0$$
and take the limit: 
\begin{equation} \label{stab} t_I \cdot \hat z_{0,I} =  um t \cdot \hat z_{0,I}\, .\end{equation} 
The local structure theorem for $\hat Z_I= G/\hat H_I$, then implies $u=\1$ which completes the proof 
of (\ref{third}).  

\par The proof of (\ref{third}) shows as well that 
the limit $m$ of every converging subsequence of $m_s$ satisfies $t_I\cdot \hat z_{0,I}=mt\cdot \hat z_{0,I}$,
and hence determines a unique element in  $M/ M\cap \hat H_I$. Thus $\lim_{s\to \infty}m_s(M\cap \hat H_I)$ exists.
\par Notice that $(M\cap H_w)_e$ has finite index in $M\cap H_I$ as the Lie algebras of both groups 
coincide. By continuity with respect to $s$ it follows that
$m_s (M\cap H_w)_e$ converges in $M/(M\cap H_w)_e$. Now (\ref{fourth}) follows
by trivializing this bundle in a neighborhood of the limit point.
\end{proof}

\begin{rmk} With the assumption and notation of the preceding lemma
let $m:= \lim_{s\to \infty}  m_s$. Then (\ref{stab}) implies the relation 
\begin{equation*}
(\hat\hf_I)_{w_I} = \Ad (m) \widehat{\hf_w}_I\,  \end{equation*} 
where $\widehat{\hf_w}_I:=(\hf_w)_I+\af_I$. To see this, first note that
$\af_I\subset(\hat\hf_I)_{w_I}$ by Lemma \ref{F-set}, hence
(by dimension) it suffices to show 
\begin{equation}\label{mhwI}
\Ad (m) (\hf_w)_I\subset (\hat\hf_I)_{w_I} .
\end{equation}
Let $X\in(\hf_w)_I$ and choose
$X(s)\in \Ad(a_{s,I})\hf_w$ with $X(s)\to X$
for $s\to\infty$. The fundamental vector field on $\hat Z_w$
corresponding to $\Ad(m)X(s)$ has a zero at $ma_{s,I}w\cdot \hat z_0$,
and from 
$$t \cdot \hat z_{0,I} =t\cdot \lim_{s\to \infty} a_{s,I} \cdot \hat z_0=
\lim_{s\to \infty} a_{s,I} w \cdot \hat z_0$$
we deduce that the fundamental vector field 
corresponding to $\Ad(m)X$ then has a zero at 
$m t \cdot \hat z_{0,I}$.
Now (\ref{mhwI}) follows from (\ref{stab}) and the fact
that $(\hat H_I)_{w_I}$ is the stabilizer of $t_I\cdot \hat z_{0,I}
=w_I \cdot\hat z_{0,I}$.
\end{rmk}

\subsection{Unimodularity}

For a moment let $G$ be a an arbitrary Lie group and $H<G$ a closed subgroup. 
We call the homogeneous space $Z=G/H$ {\it unimodular} provided that $Z$ carries a
$G$-invariant positive Borel measure, and recall that this is the case if and only if
the attached modular character 
\begin{equation}\label{defi Delta_Z}
\Delta_Z: H \to \R, \ \ h\mapsto \frac{|\det \Ad_\hf (h)|}{|\det\Ad_\gf(h)|} =|\det \Ad_{\gf/\hf}(h)|^{-1}
\end{equation}
is trivial.

After these preliminaries we return to our initial set-up of a real spherical space 
$Z=G/H$ and its boundary degenerations. In this context we record:

\begin{lemma}\label{Z_I is unimodular}  
Let $Z$ be a  real spherical space which is unimodular. Then all boundary degenerations
$Z_I$ are unimodular.
\end{lemma}

\begin{proof} 
The fact that $X \mapsto \tr (\ad X)$ is trivial in $\hf^*$, 
is a closed condition on $d$-dimensional Lie subalgebras in $\gf$. 
Now apply (\ref{defi h_I}).
\end{proof}

\section {Levi-induced spherical spaces}\label{section induced}

Let $Z=G/H$ be a real spherical space. Let $P<G$ be a parabolic subgroup 
and $P=G_P U_P$ a Levi-decomposition. Then $G_P\simeq P/U_P$. We write 
$$\pr_P : P \to G_P $$ for 
the projection homomorphism. Define $H_P:= \pr_P(H\cap P)$
and set 
$$Z_P:= G_P/ H_P\, .$$
Note that $H_P<G_P$ is an algebraic subgroup.  

\begin{prop}\label{propind} The space $Z_P$ is real spherical. 
\end{prop}

\begin{proof} Let $Q_{\rm min}<P$ be a minimal parabolic subgroup of $G$. According to 
\cite{KS1}, the number of $Q_{\rm min}$-orbits in $Z$ is finite. In particular the number of 
$Q_{\rm min}$-orbits in $P/P\cap H \subset Z$ is finite. 
Observe that $Q_{\rm min, P}:= \pr_P(Q_{\rm min})$ is a minimal parabolic subgroup of $G_P$. 
It follows that the number of $Q_{\rm min, P}$-orbits in $Z_P$ is finite. In particular 
there exists open orbits, i.e.~$Z_P$ is real spherical. 
\end{proof}
 
We call $Z_P$ the {\it Levi-induced real spherical space} attached to $P$.

\subsection{Induced parabolics with respect to open $P$-orbits}
In the sequel we are only interested in parabolic subgroups 
containing the fixed minimal parabolic 
subgroup $P_{\min}$.  
We recall the parametrization of these.
Recall that $\Sigma=\Sigma(\gf, \af)\subset \af^*$ is the root system attached to 
the pair $(\gf, \af)$.  Let $\Sigma^+\subset \Sigma$ be the positive system attached to $N$ and 
$\Pi \subset \Sigma^+$ the associated set of simple roots.
The parabolic subgroups $P\supset P_{\rm min}$ are in 
one-one correspondence with subsets $F\subset \Pi$.
The parabolic subgroup $P_F$ attached to $F\subset \Pi$ has  Levi-decomposition $P_F = G_F U_F$ where 
$G_F = Z_G(\af_F)$ with 
$$\af_F:=\{ X\in \af\mid  (\forall \alpha \in F) \ \alpha(X)=0 \}$$
and 
$$\uf_F:=\bigoplus_{ \alpha \in \Sigma^+\bs\la F\ra} \gf^\alpha\, .$$
In these formulas $\gf^\alpha\subset \gf$ is the root space attached to $\alpha\in \Sigma$ and 
$\la F \ra\subset \Sigma$  denotes the root system generated  by  $F$. 

The space $\af$ decomposes orthogonally 
as $\af = \af_F \oplus \af^F$ with 
$$ \af^F :=\Span\{ \alpha^\vee \mid \alpha \in \Pi\bs F\}\, $$
where $\alpha^\vee \in \af$ is the co-root associated to $\alpha$. Observe that 
$A^F$ is a maximal split torus of the semi-simple commutator group $[G_F, G_F]$,
and that $P_{\rm min, F}:= P_{\rm min} \cap G_F$ is a minimal parabolic of $G_F$
with unipotent radical $U^F$ where 
$$\uf^F:= \bigoplus_{\alpha\in \la F\ra^+} \gf^\alpha\, .$$

Denote $\pr_F=\pr_{P_F}$, $H_F=H_{P_F}$, and 
$$Z_F:= Z_{P_F}=G_F/H_F,$$
the Levi-induced homogeneous space attached to $P_F$.

\par We write $F_Q\subset \Pi$ for the set which corresponds to $Q$. 
In the sequel our interest is particularly
with those parabolic subgroups 
$P_F$ which contain $Q$, that is for which $F\supset F_Q$. 
For later reference we note that
\begin{equation} \label{FQ}
\gf^\alpha\subset \hf , \quad \alpha\in\la F_Q\ra.
\end{equation}

\subsection{Examples of induced spaces}

\subsubsection{Symmetric spaces}

Assume as in Example \ref{symmetric space case}
that $Z$ is a symmetric space.
The $Z$-adapted parabolic subgroup $Q$ is 
$\tau\theta$-stable, 
as are also
all parabolic subgroups $P_F \supset Q$. In particular, 
we have 
$$ P_F \cap H = G_F \cap H = H_F$$ 
and $Z_F=G_F/H_F$ is a symmetric space, which
embeds into $Z$.
 
\subsubsection{Triple spaces}\label{triple spaces}
 
For a general real spherical space  it is an unfortunate fact that basic properties of $Z$ 
typically do not inherit to $Z_F$. For example if $Z$ is affine/unimodular/has trivial automorphism group, then one  cannot expect the same for induced spaces $Z_F$. 
This is all well illustrated in the basic example of triple spaces. 
Let $\G:= \SO_e(1, n)$ for $n\geq 2$ and set 
$$ G := \G \times \G \times \G\, .$$
Then 
$$ H:= \Delta_3(\G):= \{ (g, g, g)\mid g\in \G\}$$ 
is a real spherical subgroup of $G$. Let 
$$ P_{\rm min} := \P_1 \times \P_2 \times \P_3$$ 
be a minimal parabolic subgroup of $G$, that is each $\P_i<\G$ is a minimal 
(and maximal) parabolic subgroup of $\G$.  The condition that $HP_{\rm min}\subset G$ 
is open means that all $\P_i$ are pairwise different (see \cite{DKS}). Note that $Q=P_{\rm min}$ in this case.  Note that $\Sigma=A_1 \times A_1 \times A_1$ and thus $\Pi=\{ \alpha_1, \alpha_2, \alpha_3\}$. 
There are six proper parabolic subgroups $P_F$ containing $Q$. For example  if $|F|=1$, say $F=\{\alpha_3\}$ one has 
$$ P_{\{\alpha_3\}}= \P_1 \times \P_2\times \G, $$
whereas for $|F|=2$, say $F=\{ \alpha_2, \alpha_3\}$ one has 
$$P_{\{ \alpha_2, \alpha_3\}}:=\P_1 \times \G\times \G\, .$$
Let $A<P_{\rm min}$ be a maximal split torus. Then 
$A= {\mathcal A}_1\times {\mathcal A}_2\times {\mathcal A}_3$. Further we let 
$\mathcal{M}_i<\P_i$ a maximal compact subgroup which commutes with ${\mathcal A}_i$. 
Denote by $p_i: \P_i \to \mathcal{M}_i {\mathcal A}_i$
the projection along $\mathcal{N}_i$. 
The real spherical subgroups $H_F$ for our above choices of $F$ are given by: 
\begin{align*} H_{\{\alpha_3\}} & =\{(p_1(g), p_2(g), g)\mid g\in \P_1\cap \P_2\} \simeq \P_1\cap \P_2\\
H_{\{\alpha_2, \alpha_3\}}& = \{ (m_1 a_1, m_1 a_1 n_1, m_1 a_1 n_1)\mid m_1 a_1 n_1 \in \mathcal{M}_1 \A_1\mathcal {N}_1\}\\
&=\Delta_3(\mathcal{M}_1 \A_1) \Delta_2(\mathcal{N}_1) \simeq \P_1\, .\end{align*} 

Of special interest is the case $\G=\SO_e(1,2) \simeq \mathrm{PSL}(2,\R)$.  Here in the three cases 
with $|F|=1$ one has that $H_F$ is reductive (a split torus) while this is not the case for $|F|=2$. Even more, for 
$|F|=2$ the spaces $Z_F$ are not even unimodular and have non-trivial automorphism groups.
We remark that the fine polar geometry of this example is described in \cite{DKS}, and that
trilinear functionals related to $Z$ were studied by Bernstein and Reznikov \cite{BR}.

One might think that there is always a Levi-decomposition $P_F = G_F' U_F$ for which one has 
$P_F \cap H < G_F'$. The triple cases with $|F|=2$ show that this is not the case in general.
Hence, unlike to the symmetric situation, 
we cannot expect to have embeddings $Z_F \hookrightarrow Z$ in general.

\subsection{Induced adapted parabolics}

For $F\supset F_Q$ we let 
$$Q_F= Q\cap G_F=\pr_F(Q),$$ 
which is a parabolic subgroup of $G_F$. It has the Levi decomposition
$Q_F=L_FU_{Q,F}$ where $L_F=L$ and $U_{Q,F}=U\cap G_F$.

\begin{lemma}\label{Qind} The following assertions hold: 
\begin{enumerate} 
\item $Q_FH_F=P_{\rm min, F}H_F$ is open in $G_F$.
\item $\lf\cap \hf= \qf_F \cap \hf_F$. 
\item $Q_F$ is the $Z_F$-adapted parabolic subgroup of $G_F$ containing $P_{\min,F}$.
\end{enumerate}
\end{lemma}

\begin{proof} As $Q\subset P_F$ and $\pr_F: P_F \to G_F$ is a homomorphism
we obtain:  
\begin{equation*}Q_F H_F=
\pr_F( (QH)\cap P_F)=\pr_F ( (P_{\rm min} H)\cap P_F)=
P_{\rm min, F} H_F 
\, .\end{equation*}
This is an open set since $\pr_F: P_F\to G_F$ is an open map.
Further we note that the Lie algebra of $Q_F$ is given by 
\begin{equation} \label{Q1} \qf_F=\qf \cap \gf_F = \lf + \uf^F\, .\end{equation} 
We recall (\ref{LST}). It follows that
$$\hf\cap \pf_F= \lf \cap \hf +\{ \oline X + T(\oline X)\mid \oline X \in \oline\uf^F\}\, .$$
This in turn gives that 
\begin{equation} \label{Q2} 
\hf_F = \pr_F(\hf\cap \pf_F) =\lf \cap \hf +\{ \oline X + (\pr_F\circ T)(\oline X)\mid \oline X \in \oline\uf^F\}\, .
\end{equation}
The combination of (\ref{Q1}) and (\ref{Q2}) results in
the second assertion. 
The last assertion now follows, as
$L_{F,n}=L_n\subset H$ (see Section \ref{notation}).
\end{proof}

Observe that the lemma implies that  $\af\cap\hf=\af\cap\hf_F$ and hence that
there is an equality of real ranks
\begin{equation} \label{rankequality} \rank_\R(Z) = \rank_\R(Z_F)\, . \end{equation}
Furthermore, $A_{Z_F}=A_Z$.

\subsection{Induced compression  cones}

We are interested in the behaviour of the compression cone under induction. 
Note that there is a natural action of $A$ on $A_Z=A/A_H$. 

\begin{prop}\label{prop-cvc}
Let $F\supset F_Q$. Then 
\begin{equation*} 
A_F\cdot A_Z^-= A_F \cdot A_{Z_F}^-
\end{equation*}
for the induced 
spherical space $Z_F=G_F/H_F$.
\end{prop}

\begin{proof} We shall prove that
\begin{equation}\label{string of inclusions}
A_Z^-\subset A_{Z_F}^-\subset A_F \cdot A_Z^-.
\end{equation}
Let $(\pi, V)$ be a regular irreducible real $H$-semi-spherical representation 
as considered in (\ref{pi limit}).  This induces a natural $H_F$-semi-spherical representation 
\def\fdsp{Y}
$(\pi_\fdsp, \fdsp)$ 
of $G_F$ as follows. Set $\fdsp:= V/ \uf_F V$. Clearly $\fdsp$ is a 
$G_F$-module. 
Note that $w_0:= v_0 + \uf_F V \in \fdsp$ is a lowest weight vector and hence generates an irreducible 
submodule, say $\fdsp_0$ of $\fdsp$. As 
$$V= \U(\uf) v_0=\U(\uf^F) \U(\uf_F)v_0 
\subset \U(\uf^F)(\fdsp_0 + \uf_F V)$$
we conclude that $\fdsp_0=\fdsp$
is irreducible. 
As $\pi$ is regular with respect to $\oline{Q}$, we infer that $w_0$
is  regular with respect to $\oline{Q}_F$. 
Likewise $w_H:= v_H +\uf_F V$ 
is an $H_F$-semi-spherical vector in $\fdsp$. 
As 
$$V= {\mathcal U}(\gf)v_H={\mathcal U}(\qf)v_H$$
we conclude that $v_H\notin \uf V$. Since $\uf_F\subset\uf$ it follows that $w_H\neq 0$.

\par Now if $X\in \af_\pi^{--}$, then by (\ref{pi limit})
$$ \lim_{t\to \infty} [\pi_\fdsp (\exp(tX)) w_H] = [w_0]\, .$$
This shows the first inclusion in (\ref{string of inclusions}). 

\par In the construction from above 
we realized $\fdsp$ in a quotient of $V$, but it is also possible to 
realize it as a subspace.
Set $\tilde \fdsp:= 
{\mathcal U}(\gf_F)v_0$. Then $\tilde \fdsp$ is an irreducible 
lowest weight module for $G_F$ with 
lowest weight $v_0$, hence $\tilde \fdsp\simeq \fdsp$. 
Let us describe an explicit isomorphism. Write 
$p: V\to \fdsp$ for  
the $G_F$-equivariant projection.
Then the restriction of $\tilde p:=p|_{\tilde\fdsp}$ 
establishes an isomorphism of $\tilde p: \tilde \fdsp \to \fdsp$.  
Then $\tilde w_H :=\tilde p^{-1}(w_H)\in \tilde\fdsp$ is a 
non-zero $H_F$-semi-spherical vector in $\tilde \fdsp$.  
Then $v_H= \tilde w_H + \tilde w_H^\perp$ with $\tilde w_H^\perp \in \ker p =
\uf_F V$.
Let now $X\in \af_{Z_F}^{--}$. 
By adding a suitable element 
$X'\in \af_F$ to $X$ we obtain that $\alpha(X+X')<0$ for all roots
$\alpha$ of $\uf_F$.  Hence 
$$\lim_{t\to \infty}[\pi(\exp(t(X+X'))) v_H]= 
\lim_{t\to \infty}[\pi(\exp(t(X+X'))) \tilde w_H]=[v_0]$$
and the second inclusion in (\ref{string of inclusions}) is established.   
\end{proof}

\subsection{Unimodularity issues under induction}\label{usi}

Let $P<G$ be a parabolic subgroup for which $PH$ is open. 

\begin{lemma} \label{lemuni} If $Z$ is unimodular then so is
$P/P\cap H$. 
\end{lemma}

\begin{proof} As  $PH$ is open
we can identify $P/P\cap H$ as an open subset of $Z$. The $G$-invariant measure 
on $Z$ then induces a $P$-invariant measure on $P/P\cap H$. 
\end{proof}

Next we observe the basic isomorphism 
\begin{equation*} 
Z_P= G_P/H_P\simeq P/(P\cap H) U_P\,.
\end{equation*}
which together with Lemma \ref{lemuni}
allows us to compute the associated modular character 
$\Delta_P=\Delta_{Z_P}$ 
(see (\ref{defi Delta_Z})):  

\begin{lemma} Suppose that $Z$ is unimodular. 
Then 
\begin{enumerate} 
\item The modular function for $Z_P=G_P/H_P$ is given by
$$\Delta_{P}(hu)=|\det \Ad_{\uf_P/\uf_P\cap\hf}(h)|,
\quad (h\in P\cap H, u\in U_P).$$
\item In particular if $U_P\cap H=\{1\}$ then 
$$\Delta_{P}(h)=|\det \Ad_{\uf_P}(h)|,
\quad (h\in H_P).$$
\end{enumerate}
\end{lemma}

\begin{proof} Since the modular function is trivial
on the nilpotent group $U_P$, this follows from 
the exact sequence
$$0\to \uf_P/\uf_P\cap\hf \to \pf/\pf\cap\hf \to
\pf/(\pf\cap\hf+\uf_P) \to 0$$
of $P\cap H$ modules.
\end{proof}

\section{Wave-front spaces and interlaced spherical subgroups}\label{W-f}

The following notation will be used throughout.
Recall the set 
of spherical roots $S:=\{ \sigma_1, \ldots, \sigma_s\} \subset 
\af_Z^*$ 
such that 
$$\af_Z^-=\{ X\in \af_Z\mid \sigma_j(X)\leq 0 \quad (1\leq j\leq s)\}\, .$$
We let $\omega_1, \ldots, \omega_s\in \af_Z$ be the dual basis to 
$S$, 
i.e.~we request 
\begin{itemize}
\item $\sigma_i(\omega_j) =  \delta_{ij} \qquad (1\leq i, j\leq s)$. 
\item $\omega_i \perp \af_{Z,E}$. 
\end{itemize}
This gives us the following coordinates of $\af_Z$: 
\begin{equation}\label{I-coordinates}
\R^s \times \af_{Z,E}\to \af_Z, (t, X)\mapsto \sum_{j=1}^s t_j \omega_j +X 
\end{equation}
with $\af_Z^-$ corresponding to pairs $(t, X)$ with $t_j\leq 0$.
Moreover,
\begin{equation}\label{aI basis}
\af_I=\Span\{\omega_j\mid j\notin I\}+\af_{Z,E},\qquad (I\subset S).
\end{equation}

\subsection{Wave-front spaces}\label{sub wf}
Let 
$$ \af^-:=\{ X\in \af\mid (\forall \alpha \in \Pi) \alpha(X)\leq 0\}$$
be the closure of the negative Weyl chamber. 
Then the projection to $\af_Z$
along $\af_H$ maps $\af^-$ into $\af_Z^-$.
We recall that $Z$ is called {\it wave-front} provided the projection is onto,
that is, provided 
$\af_Z^-= (\af^- +\af_H)/\af_H$. 

\begin{ex} (a) All symmetric spaces are wave-front. 
\par\noindent (b) The Gross-Prasad spaces  (\ref{GP1}) and (\ref{GP2}) are wave-front. 
\par\noindent (c) The triple space (see Section \ref{triple spaces}) 
$$G/H=\Sl(2,\R) \times \Sl(2,\R) \times \Sl(2,\R)/\Sl(2,\R)$$ 
is wave-front. 
\par\noindent (d) $(\sof(3,4), G_2(\R))$ is wave-front. 
\par\noindent (e) The series (\ref{NonWF1}) and (\ref{NonWF2}) are {\it not} wave-front.  
\end{ex}

A simple, but important feature of wave-front spaces is the  fact
which we record in the next lemma. 

\begin{lemma} If $G/H$ is wave-front then so is $G/H_IA_I$ for all $I\subset S$. 
\end{lemma}

\begin{proof}
It follows from Prop.~\ref{properties of Z_I}(\ref{zwei}),(\ref{funf})
and equation (\ref{aI basis}) that the wave-front-ness of $G/H_IA_I$ amounts to 
$\af_Z^-+\af_I+\af_H=\af^-+\af_I+\af_H$.
\end{proof}

We continue by collecting a few but important facts of wave-front spaces. 
Recall from (\ref{S rational}) that $S\subset \Q_{\geq 0}[\Pi]$
and thus every $\sigma\in S$ has a unique 
presentation $\sigma=\sum_{\alpha\in \Pi} n_\alpha \alpha$ with $n_\alpha\in \Q_{\geq 0}$. 
Accordingly we define the support of $\sigma$ by 
$$\supp(\sigma):=\{ \alpha\in \Pi \mid n_\alpha>0\}\, .$$
We denote by 
$\{ \omega_\beta'\}_{\beta\in \Pi}$ be the basis of $\af$ dual to $\Pi$, i.e.
$\alpha(\omega_\beta') =  \delta_{\alpha,\beta}$. 
For every $\sigma\in S$ we let $\Pi_\sigma\subset \Pi$ be the set of 
$\alpha$ for which 
$\R^+ \omega_\alpha' +\af_H = \R^+ \omega_\sigma$.
The following is then a reformulation of the definition of wave-front.

\begin{lemma} \label{wf-easy}For a real spherical space $Z$  the following conditions
are equivalent:
\begin{enumerate}   
\item\label{wf1} $Z$ is wave-front. 
\item\label{wf2} $\Pi_\sigma\neq\emptyset $ for all $\sigma\in S$. 
\end{enumerate}
Furthermore, for each $\sigma\in S$,
\begin{equation} \label{Pisigma} 
\Pi_\sigma = 
\supp(\sigma) \bs \bigcup_{\sigma'\neq \sigma} \supp(\sigma').
\end{equation} 
\end{lemma}
\begin{proof} 
It follows from the assumption of wave-front 
that the extremal rays of the compression cone $\af_Z^-$ 
are generated by extremal rays of $\af^-$. Hence
$\Pi_\sigma\neq \emptyset$. Together with $(\af^-+\af_H)/\af_H  \subset \af_Z^-$ we thus 
obtain the equivalence of (\ref{wf1}) and (\ref{wf2}). 
\par Let $\sigma\in S$. It follows from the definition that
$\alpha\in \Pi_\sigma$ if and only if $\sigma(\omega'_\alpha)>0$
and $\sigma'(\omega'_\alpha)=0$ for all $\sigma'\neq\sigma$. 
Since $\alpha\in\supp(\sigma')$ if and only if  
$\sigma'(\omega'_\alpha)>0$ we obtain (\ref{Pisigma}).
\end{proof}

\begin{lemma} \label{wfq}Suppose that $Z$ is wave-front. Then $\Pi_\sigma\not\subset F_Q$
for all $\sigma\in S$.
\end{lemma}

\begin{proof}  (cf. \cite{SV}, proof of Lemma 2.7.2). 
We argue by contradiction and assume there exists a $\sigma\in S$ with $\Pi_\sigma\subset F_Q$.  
Let $\alpha\in\Pi_\sigma$. Then $\alpha\in F_Q$ and hence 
$\alpha^\vee \in \af_H$ by (\ref{FQ}). Thus $\la \alpha, \sigma'\ra =0$ for all 
$\sigma'\in S$. Moreover for $\sigma'\neq \sigma$ and all $\beta\in \supp(\sigma')$ 
we have $\beta\neq \alpha$ by (\ref{Pisigma}) and hence 
$\la \alpha, \beta\ra \leq 0$. Hence 
\begin{equation} \label{ortho} \alpha\perp \supp(\sigma') 
\qquad (\alpha\in \Pi_\sigma, \sigma'\in S\bs\{\sigma\}) \, .\end{equation} 
 
\par Let $\sigma= \sum_{\alpha\in \Pi} n_\alpha \alpha=\sigma_1+\sigma_2$ where
$$\sigma_1:= \sum_{\alpha\in \Pi_\sigma} n_\alpha \alpha,\quad
\sigma_2:= \sum_{\alpha\in \supp(\sigma)\bs \Pi_\sigma} n_\alpha \alpha.$$ 
With (\ref{Pisigma}) and (\ref{ortho}) 
we conclude now $\la \alpha, \sigma_2\ra=0$
for all $\alpha\in \Pi_\sigma$. As $\la \alpha, \sigma\ra =0$ we thus have $\la\alpha, \sigma_1\ra =0$ for all 
$\alpha\in \Pi_\sigma$. Hence $\la \sigma_1,\sigma_1\ra=0$ which
contradicts that $\Pi_\sigma\neq\emptyset$. 
This proves the lemma. 
\end{proof}

\subsection{Interlaced subgroups}
Let $P<G$ be a parabolic subgroup. We say that $H$ is {\it interlaced} by $P$ if
$$ U_P \subset H \subset P$$
where $U_P$ is the unipotent radical of $P$. Note that 
with $\oline{P}=\theta(P)$ it follows that $PH$ is open
if $H$ is interlaced by $\oline{P}$. 

Observe also that if $H$ is interlaced by $P$, then 
$H_P=G_P\cap H$ and the Levi induced
spherical space is $Z_P=G_P/H_P\simeq P/H$.

\par We now show that $H_I$ is non-trivially  interlaced for all $I\subsetneq S$ 
in case $Z$ is wave-front. With notation from above we observe that we can use 
\begin{equation}\label{X_I}
X_I=-\sum_{j\notin I} \omega_j
\end{equation}
for the element $X$ in (\ref{defi h_I}).

Assume that $Z$ is wave-front. 
To $I\subset S$ we now attach a minimal set $F=F_I\subset \Pi$ such that $F_Q\subset F$ and
$H_I$ is interlaced by $\oline{P_F}$.  We define $F_I$ to be the complement
of $J_I\subset \Pi$ where 
$$J_I:= \bigcup_{\sigma\not \in I} (\Pi_\sigma\bs F_Q)\, .$$

\begin{prop}  \label{sandwich1}
Assume that $Z$ is wave-front. Let $I\subset S$ and $F=F_I$ as defined above. 
Then there exists an element $Y_I\in \af_F$ such that $Y_I+\af_H = X_I$
and $\alpha(Y_I) <0$ for all $\alpha \in \Sigma_\uf \bs \la F\ra $.  In particular, 
$$\oline{\pf_F} = \mf +\af +
\bigoplus_{\genfrac{}{}{0pt}{}{\alpha\in \Sigma}{\alpha(Y_I)\leq 0}}
 \gf^\alpha
$$
and $G_F= Z_G(Y_I)$. 
\end{prop}

\begin{proof}  It follows from Lemma \ref{wfq} and 
the definition of $\Pi_\sigma$ that 
$$\sum_{\alpha\in \Pi_\sigma\bs F_Q} \omega'_\alpha +\af_H$$
is a positive multiple of $\omega_\sigma$ for each $\sigma\notin I$. 
Thus for suitable constants $c_\alpha>0$, the element
$Y_I = -\sum_{\alpha\in J_I} c_\alpha \omega_\alpha'$
has the desired property. 
\end{proof}

\begin{cor}  \label{sandwich2}
Let $F=F_I\subset\Pi$ be as above. Then 
$\la I\ra = \Q_{\geq 0}[F]\cap \la S \ra$,
$\af_I=\af_F+\af_H$
and 
\begin{equation}\label{interlaced} (G_F\cap H)_e \,\oline {U_F} \subset H_I\subset \oline{P_F}\, .\end{equation}
In particular, $H_I$ is interlaced by $\oline {P_F}$.
\end{cor}

\begin{proof} 
The first two statements follow 
immediately from Proposition \ref{sandwich1}. 
It also follows that
\begin{equation}\label{<F>}
\alpha+\beta\in \la I\ra
\Leftrightarrow 
\alpha,\beta\in \la F\ra
\end{equation}
for all $\alpha\in\Sigma_\uf$ and $\beta\in \Sigma_\uf\cup\{0\}$ 
with $\alpha+\beta\in \N_0 [S]$.

As $F_Q\subset F$ we have $\lf\cap\hf\subset \gf_F$.
It then follows from (\ref{<F>}) and
the descriptions of $\hf$ and $\hf_I$
by means of the maps $T$ and $T_I$ that
\begin{equation*}
\gf_F\cap\hf + \oline{\uf_F}\subset \hf_I\subset\oline{\pf_F}.
\end{equation*}
This implies (\ref{interlaced}). 
\end{proof}

\begin{rmk} The property of being interlaced may hold
also in cases where $Z$ is not wave-front. For example,
the spherical subgroup $\oline N$ is 
interlaced by $P_{\rm min}$, but $Z=G/\oline N$ is not wave-front.
\end{rmk}

\section{Power series expansions of generalized 
matrix coefficients on 
the compression cone}\label{power series}

\par Given a Harish-Chandra module $V$ for the pair $(\gf, K)$ we write 
$V^{\infty}$ for its unique smooth moderate growth Fr\'echet completion
and $V^{-\infty}$ for the strong dual
of $V^\infty$. Both $V^\infty$ and $V^{-\infty}$ are 
$G$-modules. Attached to a pair $(v, \eta)\in 
V^{\infty}\times V^{-\infty}$ we associate the generalized 
matrix coefficient
$$m_{v, \eta}(g):= \eta(g^{-1} \cdot v),\qquad (g\in G),$$
which is a smooth function on $G$.

\par Let $Z=G/H$ be a real spherical space.  We say that $V$ is {\it $H$-spherical}  
provided that $(V^{-\infty})^H\neq \{0\}$. We recall (see \cite{KS2}, Thm.~3.2) that 
\begin{equation*}
\dim(V^{-\infty})^H\le \dim(V/(\lf\cap\hf+\bar\uf)V)<\infty .
\end{equation*}
The group $A_{Z,E}$ naturally acts on the vector space 
$(V^{-\infty})^H$.  Note however, that the action need not be semi-simple. 
In the sequel we only consider $\eta\in (V^{-\infty})^H$ which are 
eigenvectors of $A_{Z,E}$. 

In this case we call $(V,\eta)$ an $H$-{\it spherical pair}. We then
regard $m_{v, \eta}$ as an element of $C^\infty(G/H)$.

\begin{rmk} For our objective, namely to understand tempered representations, the assumption
that $\eta$ is an $A_{Z,E}$-eigenvector is no loss of generality.  In order to justify that we recall the abstract Plancherel theorem for the 
left regular action $L$ of $G$ on $L^2(Z)$: 
\begin{equation}\label{absPlanch}
L^2(Z) \simeq \int_{\hat G}^\oplus \M_\pi \otimes  \Hc_\pi  \ d\mu(\pi)
\end{equation}  
with $\M_\pi\subset (\Hc_\pi^{-\infty})^H$, the multiplicity space
equipped with some inner product 
(see \cite{KKSS2}, Sect. 5.1 for the notation).
Our interest lies in those linear forms $\eta$ which belong to 
$\M_\pi$ and the corresponding matrix coefficients 
$m_{v,\eta}\in L^2(Z)$ obtained from (\ref{absPlanch}) with
$v\in \Hc_\pi^\infty$.

\par If $\mu_Z$ denotes the Haar-measure on $Z$, then there exists a positive character $\psi$ of $A_{Z,E}$, 
the restriction of the modular function $\hat\Delta$ of (\ref{defi Delta_Z}) for $\hat Z=G/HA_{Z,E}$,
such that 
$\mu_Z (Ea) =\psi(a) \mu_Z(E)$ for all measurable sets $E\subset Z$ and $a\in A_{Z,E}$. 
Hence we obtain
a unitary representation $R$ of $A_{Z,E}$ by 
$$ (R(a) f)(z) = \sqrt{\psi(a)} f( za) \qquad (f\in L^2(Z), z\in Z, a\in A_{Z,E})\, .$$   
As the representations $R$ and $L$ commute
we obtain from (\ref{absPlanch}) for 
$\mu$-almost all $\pi$ that the natural action of $A_{Z,E}$
on $(V^{-\infty})^H$ restricts to 
a unitary action on $\M_\pi$.  
In particular, 
this action on $\M_\pi$ is semi-simple. Moreover if $\eta\in\M_\pi$ is an
eigenvector with eigenvalue $a^\chi$ then $f=m_{v,\eta}$ 
is an eigenvector for $R(a)$ with eigenvalue $\psi(a)^{1/2}a^\chi$.
Hence 
\begin{equation}
 \label{sqrt psi}
e^{\re\chi} = \hat\Delta(a)^{-1/2}
\end{equation}
for all 
$\chi\in \af_{Z,E,\C}^*$ appearing in the spectrum of the $\af_{Z,E}$-module $\M_\pi$.
\end{rmk}

\par Recall the set $\W=\F\F'$ with $\F'\subset \hat H_0$ a set of representatives of the 
component group $\hat H_0/ HA_{Z,E}$. For $w\in \W$ we set $\eta_w:= w\cdot \eta= \eta(w^{-1}\cdot)$ and observe that
$\eta_w$ is $H_w$-fixed.  We thus have an isomorphism
$$ (V^{-\infty})^H \to (V^{-\infty})^{H_w}, \ \ \eta\mapsto \eta_w\, .$$
We recall that $L\cap H<H_w$ for all $w\in \W$. 
Observe the obvious identity 
\begin{equation}\label{mvetaw}
m_{v, \eta} (gw\cdot z_0)= 
m_{v, \eta_w}(g\cdot z_w) \qquad (g\in G)\, .
\end{equation}
where $z_w$ denotes the origin of the homogeneous space $Z_w=G/H_w$.

\begin{lemma} \label{w-character} Let $\eta\in (V^{-\infty})^H$ be an $A_{Z,E}$-eigenvector
corresponding to a character $\chi\in \af_{Z,E,\C}^*$, i.e.~$a\cdot \eta= a^\chi  \eta$ for $a\in A_{Z,E}$.
Then for all $w\in \W$: 
$$a\cdot  \eta_w = a^{\chi}\eta_w \qquad (a\in A_{Z,E})\, .$$ 
\end{lemma}

\begin{proof} Recall from Lemma \ref{F-set} that $aw= wah $ for some $h\in H$, 
hence the assertion. 
\end{proof}

{}From this lemma we obtain the identity 
\begin{equation}\label{chi on m}
m_{v,\eta}(abw \cdot z_0)= a^{\chi}  m_{v, \eta}(bw\cdot z_0) 
\end{equation}
for all $a\in A_{Z,E}, b\in A_Z$ and all $ w\in\W.$

In the sequel we realize $\af_Z$ inside of $\af$ as $\af_H^\perp$.
Let $a_0\in A$ be such that $a_0 A_H$ lies in the interior of $A_Z^-$.
It follows from Lemma 5.1 in \cite{KS2} that 
we can choose this element $a_0$ such that in addition there is a
relatively compact open neighbourhood $U_A$ of $\1$ in $A_Z$ with 
\begin{equation} \label{p-id} 
\kf + \Ad(a) \hf_w +\af_Z =\gf\end{equation} 
for all $a\in a_0 U_A\cdot A_Z^-$ and 
$w\in\W$.
What we have fixed so far are the tori $A$ and $A_Z$. These are invariant
under conjugation by $a_0$, and hence 
we are free to replace $K$ by $\Ad(a_0^{-1}) K$. We 
may thus assume that (\ref{p-id}) holds for all 
$a\in  U_A\cdot A_Z^-$. We note that such a change
leaves the neighborhood $U_A$ unchanged. Moreover, we 
still have the freedom of making further replacements
of this type, but now only by elements from $A_Z^-$. 

What we actually use is the conjugated form 
\begin{equation}\label{p-id2}\Ad(a)^{-1} \kf +\af_Z +\hf_w=\gf\end{equation}
and we express elements from the Lie algebra (and via Poincar\'e-Birkhoff-Witt also elements from 
$\U(\gf)$) via this decomposition. We then let $\U(\gf)$ act on smooth functions on 
$G$ via right differentiation.  Functions on $Z_w=G/H_w$ will be considered as 
functions on $G$ which are right $H_w$-invariant.  The functions of our concern are 
the smooth matrix coefficients $g\mapsto m_{v,\eta_w}(g)$ when restricted 
to functions on $U_A \cdot A_Z^-$. 

\subsection{Holomorphic decompositions of the universal enveloping algebra}

The fact that matrix coefficients $m_{v,\eta}$ admit power series expansions when restricted 
to $A^-A_H/ A_H\subset A_Z^-$ rests on several decomposition theorems for the universal algebra $\U(\gf)$ of $\gf$. 
These results were in provided in \cite{KS2}, Section 5, and 
are parallel to the special case where $Z$ is a group (see \cite{Kn}, Ch. VIII, \S 7, 8).

We develop here the framework for a theory, which extends the results from \cite{KS2}
by allowing the expansions to take place on the full
compression cone $A_Z^-$.

\par Recall the coordinates (\ref{I-coordinates}) of $\af_Z$ and fix a basis 
$\nu_1, \ldots, \nu_e$ of $\af_{Z,E}^*$. 
We write $D$ for the open unit disc in $\C$ and write $D^\times =D\bs\{0\}$. Let $\D=D^s \times (\C^*)^e$ 
and $\D^\times =(D^\times)^r\times (\C^*)^e$. For $r>1$ we realize $A_Z^-$ in $r\D^\times$ via the map 
$$ A_Z^-\to r\D^\times, \ \ a\mapsto (a^{\sigma_1}, \ldots, a^{\sigma_s}, a^{\nu_1}, \ldots, a^{\nu_e})\, .$$
It is no loss of generality to assume that $U_A$ and $r$ are chosen
such that $r\D\cap (\R_{>0})^n =  U_A \cdot A_Z^-$.

We write $\Oc(\D)$ for the ring of holomorphic functions on $\D$ and $\Oc_1(\D)$ 
for the subring 
of functions which are independent of the variable from $(\C^*)^e$.

\par Let $(W_l)_l$ be a basis for $\hf_{\lim}= \lf\cap\hf+\oline\uf$ 
such that either $W_l\in \lf\cap \hf$  or 
$W_l=X_{-\alpha}\in \gf^{-\alpha}$ for some $\alpha\in \Sigma_\uf$. 
Further let $(V_k)_k$ be a basis for $\af_Z$ and 
$(U_j)_j$ a linear independent set in $\kf$ which complements 
$(W_l, V_k)_{l,k}$ to a basis for $\gf$.

For $a\in A_Z$ we set $W_l(a):= W_l$ if $W_l\in \hf\cap\lf$, and with the notation of (\ref{T}),
$$W_l(a) := X_{-\alpha} + \sum_{\beta} a^{\alpha+\beta} X_{\alpha, \beta} 
=a^\alpha \Ad(a) (X_{-\alpha}+T(X_{-\alpha}))$$
if $W_l=X_{-\alpha}$.
Then $(W_l(a))_l$ is a basis of $\Ad(a)\hf$ for all $a\in A_Z$, and 
the transition between $(W_l(ab))_l$ and $(\Ad(b)W_l(a))_l$
is by a diagonal matrix depending only on $b\in A_Z$.
In particular, $W_l(a) =W_l(ab)$ for $b\in A_{Z,E}$. 

Now note the following, as a consequence of our normalization (\ref{S normalization}) of $S$: 
The set of vectors $(W_l(a), V_k, U_j)_{l,k,j}$ is expressable
in terms of the basis $(W_l, V_k, U_j)_{l,k,j}$ through a transition matrix of the form 
$\1+P(a)$, where each entry of $P(a)$ extends to a polynomial in $\Oc_1(r\D)$ 
without constant term.  It follows that there exists $a_0\in A_Z^-$ such that 
this matrix is invertible for all $a\in a_0 U_A A_Z^-$, and such that
the entries of the inverse matrix again belong to $\Oc_1(r\D)$
as functions of $a_0^{-1}a$. 
 
\begin{lemma}\label{u-deco} 
There exists $a_0\in A_Z^-$ with the following property.
For every $X\in \gf$ there exist functions $f_j, g_k, h_l\in \Oc_1(r\D)$ 
such that 
$$ X= \sum_{j} f_j (a) U_j + \sum_{k} g_k(a) V_k + \sum_l h_l (a) W_l(a_0a) 
$$
for all $a \in U_A \cdot A_Z^-$. Moreover, the values at zero of these functions are given by
$$ X= \sum_{j} f_j (0)  U_j + \sum_{k} g_k(0) V_k + \sum_l h_l (0) W_l.$$
In particular, $g_k(0)=0$ for all $k$ if $X\in\uf$.
\end{lemma}

\begin{proof} Clear by the preceeding remarks. The statement for $X\in \uf$
follows from the fact that $\uf\subset \kf+\hf_{\lim}= \operatorname{span}\{U_j,W_l\}$.
\end{proof} 

Using induction on the degree, an  immediate consequence of that is (see \cite{KS2}, Lemma 5.3):

\begin{cor}\label{cor u-deco} For every $u\in \U(\gf)$ there exist
${\bf U}_j\in \U(\kf),  {\bf V}_j \in \U(\af_Z)$,  ${\bf W}_j(a) \in \U(\Ad(a)\hf)$,  
and $f_j\in \Oc_1(r\D)$ such that 
\begin{equation}\label{sta-deco} u =\sum_{j=1}^p f_j (a_0^{-1}a) {\bf U}_j 
{\bf V}_j {\bf W}_j(a)  
\qquad (a \in  a_0U_A\cdot A_Z^-)\, .\end{equation}
Moreover each $\mathbf W_j$ is a product of $W_i$'s, as a function of $a$, and 
$$ \deg ({\bf U}_j) + \deg ({\bf V}_j)+ \deg ({\bf W}_j(a))\leq \deg (u)\, .$$
\end{cor}

\par For every $n\in\N_0$ we denote by $\U(\gf)_n$ the subspace of $\U(\gf)$ consisting of elements of degree
at most $n$. 

We use $\Zc(\gf)$,  the center of $\U(\gf)$, to replace the ${\bf V_j}$ in the decomposition (\ref{sta-deco}).  
The formal statement is as follows:
 
\begin{lemma}\label{udeco2}  There exists a finite dimensional subspace 
$\mathcal{ Y}\subset \U(\af_Z)$ with the following property. For all $n\in\N_0$ there exists 
an element $a_n\in  A^-_Z$ such that 
for every $u\in \U(\gf)_n$ there exist ${\bf U}_j\in \U(\kf),  {\bf Y}_j \in \mathcal{Y}, 
z_j\in \Zc(\gf)$, ${\bf W}_j(a) \in \U(\Ad(a)\hf)$,
and $f_j\in \Oc_1(r\D)$ such that 
\begin{equation} \label{udeco2a} u =\sum_{j=1}^p f_j (a_n^{-1}a) {\bf U}_j {\bf Y}_j 
{\bf W}_j(a)  z_j \qquad (a \in a_nU_A\cdot A_Z^-)\, .
\end{equation}
Moreover each $\mathbf W_j$ is a product of $W_i$'s, as a function of $a$.
\end{lemma}

\begin{proof} The proof, which goes by induction on the index $n$, is analogous 
to \cite{KS2}, Lemma 5.5.  As explained there we can reduce to the case where $u=Xv$
with $X\in\uf$ and $v\in \U(\gf)_{n-1}$ (note that, by equation (5.9) of \cite{KS2},
we have $T_{m''}\in\hf\cap\lf$ for the element denoted like that,
and hence this element is a fixed linear combination of the $W_l(a)$ for all $a$).
In fact, it is more convenient to
consider $u=vX$, which is equivalent as the induction hypothesis applies to $[X,v]$.
We use Lemma \ref{u-deco} and obtain for $a\in a_0U_A \cdot A_Z^-$
$$u=\sum_j f_j(a_0^{-1}a)\, vU_j +\sum_k g_k(a_0^{-1}a)\, vV_k + \sum_l h_l(a_0^{-1}a)\,vW_l(a).$$
The terms with $W_l$ obtain the desired form (for any $a_n\in a_{n-1}A_Z^-$)
immediately from the induction hypothesis
applied to $v.$  By replacing $vU_j$ with $U_jv$, which we are allowed to do by
the induction hypothesis, we can say the same for the terms with $U_j$. Finally, since
$X\in\uf$ we have from Lemma \ref{u-deco} that $g_k(0)=0$ for all $k$. This allows
us to use the argument in the final lines of the proof of \cite{KS2}, Lemma 5.5,
and obtain (\ref{udeco2a}) with an appropriate $a_n$.
\end{proof}

The main conclusion then is: 

\begin{prop}\label{prop-punch} Let $u\in \U(\af_Z)_n$ and let $a_n\in A^-_Z$ be as above. 
Then there exist
${\bf U}_j\in \U(\Ad(a_n)^{-1}\kf)$, $ {\bf Y}_j \in \mathcal{Y}$, 
$z_j\in \Zc(\gf)$ and $f_j\in \Oc_1(r\D)$ such that 
\begin{equation}\label{expansion of u}
u =\sum_{j=1}^p f_j (a) (\Ad(a)^{-1} {\bf U}_j)  {\bf Y}_j   z_j \mod \U(\gf)\hf \qquad (a \in U_A\cdot A_Z^-)\, .
\end{equation}
\end{prop}

\begin{proof}  Let $a\in  U_A\cdot A_Z^-$ then
$u =\sum_{j=1}^p f_j (a) {\bf U}_j {\bf Y}_j 
{\bf W}_j(a_na)  z_j  $
by (\ref{udeco2a}). Apply $\Ad(a_na)^{-1}$ 
and observe that $\Ad(a_na)^{-1}u =u$. 
\end{proof}

\begin{rmk}\label{p-rem} It is clear that Proposition \ref{prop-punch} holds as well if $\hf$ is replaced by
$\hf_w$ for $w\in \W$. 
\end{rmk}

By a hypersurface in $A_Z$ we understand a level set of a non-zero real analytic  function on $A_Z$. 
Notice that a hypersurface is a closed subset of measure zero.  

Then there is 
the following counterpart of Proposition \ref{prop-punch} which does not depend on the filtration degree $n$
but carries no information on the coefficients: 

\begin{lemma} \label{Hyperplane avoidance}  There exists a countable union of hypersurfaces ${\mathcal S}\subset A_Z$ 
such that 
\begin{equation}  \U(\gf) = \U (\Ad(a)^{-1}\kf) {\mathcal Y}\Zc(\gf) \U (\hf)     \qquad (a \in A_Z \bs {\mathcal S})\, . \end{equation}
\end{lemma}
\begin{proof} The proof of  \cite{KS2}, Lemma 5.5, shows that for every $n\in\N$ there exists 
a finite union of 
hypersurfaces ${\mathcal S}_n\subset A_Z $ such that  
$$ \U(\gf)_n\subset  \U (\Ad(a)^{-1}\kf) {\mathcal Y}\Zc(\gf) \U (\hf)     
\qquad (a \in A_Z \bs {\mathcal S}_n)\,.\qedhere$$
\end{proof} 

We let now $n_0:=  1+ \max \{ \deg Y \mid Y \in  {\mathcal Y}\}$. 
As described in (\ref{p-id}), 
we are free to replace $K$ by $\Ad(a_0^{-1})K$ without disturbing
other choices. We do this with the above 
element $a_{n_0}$. Then we obtain that (\ref{expansion of u})
holds for all $u \in \U(\af)_{n_0}$ with 
${\bf U}_j\in \U(\kf)$.
Further Lemma \ref{Hyperplane avoidance}  allows us to request in addition that 
\begin{equation} \label{Cas-Osb}  \U(\gf) = \U (\kf) {\mathcal Y}\Zc(\gf) \U ((\hf_w)_I) \, . \end{equation}
for all $I\subset S$ and $w \in \W$.

We recapitulate that we have shown there exists a
Cartan involution for which (\ref{p-id}) holds
for all $a\in U_A\cdot A_Z^-$, for which
(\ref{expansion of u}) is valid for $u \in \U(\af)_{n_0}$ with 
${\bf U}_j\in \U(\kf)$, and for which 
(\ref{Cas-Osb}) holds for all $I\subset S$ and $w \in \W$.
We shall refer to the corresponding maximal compact subgroup $K$ 
as {\it regular}.

\par We recall that (\ref{Cas-Osb}) reduces to the Casselman-Osborne Lemma 
(\cite{W} Prop. 3.7.1)
in the case of $\hf=\nf$
and implies the following finiteness result (see \cite{AGKL}, Prop. 4.2): 

\begin{lemma} \label{finite generation} With a regular choice of $K$, every Harish-Chandra module $V$ is finitely generated 
as $\U((\hf_w)_I)$-module, for all $w\in \W$ and $I\subset S$. 
In particular, $H_0(V, (\hf_w)_I)= V/(\hf_w)_I V$ is finite dimensional. 
\end{lemma}

We now fix once and for all a Cartan involution for which
$K$ is regular.

\subsection{Power series expansions for $K$-finite vectors}

Following \cite{CM} and \cite{vdB} we developed in
\cite{KS2}, Section 5, a theory of power series expansions
for $K$-finite matrix coefficients $m_{v,\eta_w}$, which we briefly summarize. 
Attached to an $H$-spherical pair $(V,\eta)$ 
there exists a number 
$d\in \N$ and for each $v\in V$  a finite set 
$$ \mathcal{E}=
\{ \Lambda_1, \ldots, \Lambda_{l}\} \subset \af_{Z, \C}^*$$
such that for every $w\in \W$ 
one has an expansion
\begin{equation} \label{asymp'}
m_{v,\eta} (aw\cdot z_0) =  \sum_{j=1}^{l} \sum_{\alpha\in \N_0[S]} 
a^{\Lambda_j+\alpha} q_{\alpha,j,w}(\log a)\end{equation}
with absolute convergence for all $a\in U_A A^-_Z$. 
The expansion (\ref{asymp'}) is derived
in analogy to \cite{KS2}, Cor.~5.7, when we use Proposition \ref{prop-punch} 
(with the regular choice of $\kf$) and Remark \ref{p-rem} 
instead of \cite{KS2}, Lemma 5.5.
Here we also used that $\eta$ is an $A_{Z,E}$-eigenvector  to separate off the $A_{Z,E}$-behavior of $m_{v, \eta_w}$ in 
advance (see 
(\ref{chi on m})), i.e.   
\begin{equation} \label{leadchi}\Lambda_j|_{\af_{Z,E}} = \chi\qquad (1\leq j\leq l)\, .\end{equation}

\par We can arrange that no mutual difference between two elements 
from $\mathcal E$ belongs to $\Z[S]$,
so that the individual terms in (\ref{asymp'}) are unique.

The  set of exponents of $v$ along $A_Z^-$ is then the set 
$$\Xi_{v}=\{ \xi\in\af_{Z,\C}^* \mid (\exists \alpha,j,w)\,\,
\xi=\Lambda_j+\alpha,\, q_{\alpha,j,w}\neq0\}$$
and the set of leading exponents is
\begin{equation*}
\mathcal{E}_{{\rm lead},v}=\{ \xi\in\Xi_{v}\mid \forall \alpha\in\N_0[S]\bs\{0\}:\, \xi-\alpha\notin\Xi_{v} \}.
\end{equation*}

Let $\af\oplus\bfrak\subset\gf$ be a full Cartan subalgebra,
and let  $W_\C$ denote the corresponding Weyl group for the root
system of $(\af\oplus\bfrak)_\C$ in $\gf_\C$.
As in \cite{Kn}, Thm.~8.33 (see also \cite{vdB} Thm 2.4 and Prop.~4.1) 
we find, in case $V$ admits an infinitesimal character,  that 
\begin{equation}\label{big Weyl}
\mathcal{E}_{{\rm lead},v}\subset \{u\Lambda|_{\af_Z}\mid u\in W_\C\}+\rho_P.
\end{equation}
where $\Lambda\in (\af\oplus\bfrak)_\C^*$ is the infinitesimal character of $V$.
As every Harish-Chandra module has a finite composition series we conclude 
that $\mathcal{E}_{{\rm lead},v}$ is finite.

Notice that the set on the right side of the inclusion (\ref{big Weyl})
is independent of $v$. We define
\begin{equation} \label{defXi} \Xi:=\cup_{v}\, \Xi_{v}\end{equation}
and 
\begin{equation}\label{defElead}
\mathcal{E}_{\rm lead}:=\{ 
 \xi\in\Xi\mid \forall \alpha\in\N_0[S]\bs\{0\}:\, \xi-\alpha\notin\Xi \}. 
\end{equation}
It follows that
$$\mathcal{E}_{\rm lead}
 \subset \{u\Lambda|_{\af_Z}\mid u\in W_\C\}+\rho_P,$$
in particular, it is a finite set.
Now
$$\Xi\subset \mathcal{E}_{\rm lead}+\N_0[S],$$
and
the expansion (\ref{asymp'}) reads
\begin{equation} \label{asympXi}
m_{v,\eta} (aw\cdot z_0) = \sum_{\xi\in \Xi} 
a^{\xi} q_{\eta,v,\xi,w}(\log a),
\end{equation}
with polynomials $q_{\eta,v,\xi,w}$ of degree $\le d$ which depend linearly
on $\eta$ and $v$. To simplify notation we shall 
write $q_{\xi,w}$ instead of
$q_{\eta,v,\xi,w}$, whenever it is clear from 
the context that $\eta$ and $v$ have been fixed. 
{}From (\ref{mvetaw}) we see that
\begin{equation} \label{qetaw}
q_{\eta_w,v,\xi,\1}=q_{\eta,v,\xi,w}
\end{equation}
for all $w\in \W$.
Note that all $\sigma\in S$ vanish on $\af_{Z,E}$, hence by (\ref{leadchi}) 
\begin{equation} \label{xionedge}
\xi|_{\af_{Z,E}}= \chi  \qquad (\xi\in\Xi)\, .\end{equation}
To the pair $(V,\eta)$  we  attach now an element $\Lambda_{V, \eta}\in \af_Z^*$ as follows.
Recalling (\ref{I-coordinates}) we let
\begin{equation} \label{Lambdadef} \Lambda_{V,\eta} (\omega_j):=\min_{\xi\in \Xi}
\re \xi(\omega_j)= \min_{\lambda\in \mathcal{E}_{\rm lead}}
\re \lambda(\omega_j)\end{equation}
and define
\begin{equation} \label{Lambda on edge}
\Lambda_{V, \eta}|_{\af_{Z,E}}= \re \chi.
\end{equation}
\begin{rmk} The exponent $\Lambda_V$ defined in \cite{KSS}, Th. 5.8, does not depend on the particular $\eta$. 
The exponent $\Lambda_{V,\eta}$ defined here is an invariant of the pair $(V,\eta)$. 
\end{rmk} 

\subsection{Upper and lower bounds on matrix coefficients}\label{optimal bound}
We regard $\Lambda_{V,\eta}$ as an element in $\af^*$ trivial on $\af_H$.
Hence (\ref{asympXi}) yields, after possibly shrinking $U_A$, for all $v\in V$ a 
constant $C_v>0$ such that 
\begin{equation}\label {ub} |m_{v,\eta}(aw\cdot z_0)| 
\leq C_v (1 +\|\log a\|)^d a^{\Lambda_{V,\eta}}
\qquad (a\in U_A\cdot A_Z^-)\end{equation} 
for all $w\in\W$.

A sharper upper bound is possible along rays in $A_Z^-$.
By  (\ref{I-coordinates}) and (\ref{xionedge})-(\ref{Lambda on edge})
we find for $a\in A_Z^-$ that
\begin{equation}\label{weta}
\mathbf{w}_\eta(a):= \max_{\xi \in \Xi} a^{\re \xi}\le a^{\Lambda_{V,\eta}},\qquad (a\in A_Z^-).
\end{equation}
Let  $\Sphere(\af_Z^-)$ be the intersection of the unit 
sphere with $\af_Z^-$, then
$${\bf w}_\eta(\exp(tX))= e^{t\lambda_X} \qquad (t\geq 0)\,
$$
where
$$\lambda_X:= \max_{\xi \in \Xi} \re \xi(X), \qquad X\in \Sphere(\af_Z^-).$$
Observe also that ${\bf w}_\eta$ is continuous and a weight on $A_Z^-$, that is
for all compact subsets $C_A\subset A_Z^-$ there exists a constant $c>0$ such that 
$$\frac1c {\bf w}_\eta (a) \leq {\bf w}_\eta (ba) \leq c {\bf w}_\eta(a) \qquad 
(a\in A_Z^-, b\in C_A)\, .$$

Furthermore, let 
$$\Xi_X=\{ \xi \in \Xi\mid \re \xi(X) 
=\lambda_X\}$$  and
define
$$d_X =\max_{v\in V,\xi\in\Xi_X,w\in\W} 
\deg(t\mapsto q_{v,\xi, w}(tX))$$
and 
$${\bf d}_\eta(\exp (tX)):= t^{d_X}   \qquad (t\geq 0)\, . $$

\begin{lemma}\label{radial bound}
Let $X\in \Sphere(\af_Z^-)$. There exists for each $v\in V$ 
a constant $C=C_{v,X}>0$ such that
\begin{equation} \label{ub-exact}  
|m_{v,\eta}(\exp(tX) w\cdot z_0)| 
\leq C {\bf w}_\eta(\exp(tX)) {\bf d}_\eta(\exp(tX))
\end{equation}
for all $t\ge 1$ and $w\in \W$. 
\end{lemma}

\begin{proof}
With (\ref{asympXi}) we write
\begin{equation*} 
m_{v,\eta} (a_tw\cdot z_0) = \sum_{\xi\in \Xi_X} 
a_t^{\xi} q_{v,\xi,w}(tX)+\sum_{\xi\in \Xi\bs\Xi_X} 
a_t^{\xi} q_{v,\xi,w}(tX)
\end{equation*}
where $a_t=\exp(tX)$. 
The polynomials $q_{v,\xi,w}$ in the first sum have degree at most $d_X$,
and hence the asserted bound follows for this part. In the second sum
the polynomials may have higher degrees (up to $d$), but as the exponentials are
bounded by $e^{\lambda_X't}$ for some
$\lambda'_X<\lambda_X$,
this term is in fact of the form $o(e^{t\lambda_X})$ for $t\to\infty$.
\end{proof}

The bound in (\ref{ub-exact}) is essentially optimal. The formal statement will be given in Proposition \ref{exact lb}.
However, we first need a lemma.

\begin{lemma}\label{degree along ray} 
Let $X\in \Sphere(\af_Z^-)$ and let $\lambda_X$
and $d_X$ be defined as above. Then the set
$$\Lambda=\{ \xi(X)\mid \xi \in \Xi_X\}\subset \lambda_X+i\R.$$
is finite. Then for each $Y_0\in\af_Z$ we have 
\begin{equation}\label{deg1} 
\deg(t\mapsto q_{v,\xi, w}(Y_0+tX))\leq d_X,
\end{equation}
for all $v\in V$, $\xi\in\Xi_X$, and $w\in\W$,
and there exist 
$\lambda\in\Lambda$, $v\in V$, and $w\in\W$
such that
\begin{equation}\label{deg2} 
\deg(t\mapsto 
\sum_{\genfrac{}{}{0pt}{}{\xi\in\Xi}{\xi(X)=\lambda}}
e^{\xi(Y_0)}q_{v,\xi,w}(Y_0+tX))=d_X.
\end{equation}
\end{lemma}

\begin{proof}
We start with a general remark.  For $Y\in \af_Z$ and $f$ a differentiable function 
on $\af_Z$ we let $Yf$ be the derivative of $f$ with respect to $Y$. 
Then the coefficients of $m_{v, \eta}$ and $m_{Yv,\eta}$ are related as follows:
\begin{equation}\label{c-relations} 
e^{\xi}q_{Yv,\xi,w}=-Y[e^{\xi}q_{v,\xi,w}]\, .
\end{equation}

Let  $\xi\in\Xi_X$ and $w\in \W$. We consider 
for each $v\in V$
\begin{equation*}%
 e^{\xi(Y)} q_{v,\xi,w}(Y+tX)
\end{equation*}
as an analytic function of $Y\in \af_Z$
into the space of polynomials in $t$ of degree $\leq d$.
The value of this function at $Y=0$ belongs to the subspace of
polynomials of degree $\leq d_X$, and by (\ref{c-relations})
the same is valid for all its derivatives at $0$.
It follows that it has degree $\leq d_X$ for all $Y$. This proves (\ref{deg1}).

Likewise, for $\lambda\in\Lambda$ and $w\in\W$ we consider for each $v\in V$
\begin{equation}\label{exppol}
e^{-t\lambda(X)}
\sum_{\genfrac{}{}{0pt}{}{\xi\in\Xi}{\xi(X)=\lambda}}
 e^{\xi(Y+tX)} q_{v,\xi,w}(Y+tX)
\end{equation}
as an analytic function of $Y\in \af_Z$
into the space of polynomials in $t$ of degree $\leq d_X$.
If (\ref{deg2}) fails to hold, then
(\ref{exppol}) belongs to the polynomials of degree $< d_X$ 
for every $\lambda$, $v$ and $w$ and for $Y=Y_0$.

Then again by (\ref{c-relations}) this will be the case for the
derivatives at $Y=Y_0$, and hence it will have degree
$< d_X$ at every $Y$. By the linear independence of the involved
exponential polynomials of $Y$ this implies that the individual
terms $q_{v,\xi,w}(Y+tX)$ also must have degree $< d_X$ at every $Y$,
hence in particular at $Y=0$, and thus a contradiction with the 
definition of $d_X$ will be 
reached.
\end{proof}

\begin{proposition}\label{exact lb} Assume $\eta\neq0$.
Let $Y_0\in \af_Z^-$ and $X\in \Sphere(\af_Z^-)$. Then there exist $w\in \W$, $v\in V$, 
a constant $C=C_{Y_0,X}>0$, an $\e>0$ and a sequence $t_1<t_2< \ldots $ with $t_n\to \infty$ such that 
\begin{equation} \label{lb-exact}  |m_{v,\eta_w}(\exp(Y_0+tX)\cdot z_w)| 
\geq C  {\bf w}_\eta(\exp(tX))\, {\bf d}_\eta(\exp(tX))\end{equation} 
for all  $t\in (t_n-\e, t_n+\e)$ and all $n\in \N$.  
\end{proposition}
\begin{proof} 
Consider the function  
$$F_{v,w} (t):=
\sum_{\xi\in\Xi_X} q_{v,\xi,w}(Y_0+tX) e^{\xi(Y_0+tX)  }\, .$$
As in the proof of Lemma \ref{radial bound} we have
\begin{equation}\label{O}
m_{v,\eta}(\exp(Y_0+tX)w\cdot z_0)=F_{v,w}(t) + o(e^{t\lambda_X})
\end{equation}
as  $t\to\infty$ 

Let $\Lambda\subset \lambda_X+i\R$ be as in Lemma \ref{degree along ray}.
Then
$$F_{v,w}(t)=e^{t\lambda_X}\sum_{\lambda\in\Lambda} e^{i t \im\lambda }
\sum_{\genfrac{}{}{0pt}{}{\xi\in\Xi}{\xi(X)=\lambda}}
e^{\xi(Y_0)} q_{v,\xi,w}(Y_0+tX). $$
For a one-variable polynomial 
$p(t)=\sum_{j=0}^N a_j t^j$ we denote by $p(t)_j=a_j$ the $j$-th coefficient.  With $v$ and $w$ as in the final conclusion of 
Lemma \ref{degree along ray} we set
$$
c_\lambda:=\sum_{\genfrac{}{}{0pt}{}{\xi\in\Xi}{\xi(X)=\lambda}}
e^{\xi(Y_0)}q_{v,\xi,w}(Y_0+tX)_{d_X}$$
for $\lambda\in\Lambda$ and note that then $c_\lambda\neq 0$ for 
some $\lambda$.
Let 
$$ F_{v, w}^{\rm top}(t): = e^{t\lambda_X} t^{d_X}\sum_{\lambda\in \Lambda } 
c_\lambda e^{it \im \lambda}, 
$$
then it follows from Lemma \ref{Wallach extended} below
that 
$$
\limsup_{t\to\infty} |e^{-t\lambda_X } t^{-d_X} F^{\rm top}_{v,w}(t)|^2\neq 0.$$
Moreover, since the derivative of 
$\sum_{\lambda\in \Lambda } 
c_\lambda e^{it \im \lambda}$
is bounded above we deduce the existence of positive constants $C$ and
$\e$ and a sequence $t_1<t_2< \ldots $ with $t_n\to \infty$, such that 
$$|F_{v,w}^{\rm top}(t)|\ge C e^{t\lambda_X} t^{d_X}$$
for all $|t-t_n|<\epsilon$.
Now (\ref{lb-exact}) 
follows from (\ref{O}). 
\end{proof}

The following is shown in \cite{W}, Ch.~4, Appendix 4.A.1.2 (1):

\begin{lemma}\label{Wallach extended} Let $u_1, \ldots, u_m \in\R$ be distinct real numbers
and  let $c_1,\dots,c_m\in\C$.
Then
$$\sum_{j=1}^m |c_j|^2 \le \limsup_{s\to \infty} | \sum_{j=1}^m c_je^{i u_j s}|^2.$$ 
\end{lemma}

\section{Quantitative upper bounds for generalized matrix coefficients}\label{quantitative bounds}

The goal of this section is to  achieve more quantitative versions of (\ref{ub})
and (\ref{ub-exact}) in which the $v$-dependent constant 
$C$ is replaced by a continuous norm on $V$.  For this we shall need to
impose on $G/H$ that it is either absolutely spherical or wave-front.

\par A norm $p$ on the Harish-Chandra module  $V$ is called $G$-continuous, provided that the Banach-completion $V_p$ of $V$ 
defines a Banach
representation of $G$ with 
$V_p^{\rm K-fin}\simeq _{\gf, K} V$ (see \cite{BK}, Sect. 2.2).
It is a consequence of the 
Casselman embedding theorem that $G$-continuous norms exist. 

\par Fix a basis $X_1, \ldots, X_n$ of $\gf$.  Given a $G$-continuous norm $p$ on $V$ we define 
the $k$-th Sobolev norm $p_k$ of $p$ by 
$$p_k(v):= \sum_{m_1+\ldots+m_n\leq k} p(X_1^{m_1}\cdot \ldots \cdot X_n^{m_n} v) \qquad (v\in V)$$
and note that $p_k$ is $G$-continuous as well. 
The following two statements are variants of the Casselman-Wallach globalization theorem (see \cite{BK}): 
\begin{itemize}
\item For any two $G$-continuous norms $p$ and $q$ one has $V_p^\infty \simeq  V_q^\infty$. 
\item For any two $G$-continuous norms $p$ and $q$ there exist $k\in \N$ and $C>0$ such that 
$p\leq C q_k $. 
\end{itemize}
We use the notation $V^\infty :=V_p^\infty$, for any $G$-continuous norm $p$, and call $V^\infty$ the Casselman-Wallach 
globalization of $V$.

\subsection{The absolutely spherical case}

In the following lemma we assume that the orbit set
$P_{min,\C}\bs G_\C/ H_\C$ is finite. The condition that
$P_{\min}\bs G/ H$ is finite is equivalent to $G/H$ being real spherical (see \cite{KS1}),
but finiteness of $P_{min,\C}\bs G_\C/ H_\C$ is known to be  a stronger condition
(see \cite{Mat2}, Remark 7). It is fulfilled for absolutely spherical spaces, since then
$B_\C \bs G_\C/ H_\C$ is finite for every Borel subgroup $B_\C$.

\begin{lemma} \label{comparison smooth} Suppose that $P_{min,\C}\bs G_\C/ H_\C$ is finite. 
Then the space $\hf V^\infty$ is closed in $V^\infty$
for every Harish-Chandra module $V$. In particular, if $Z$ is absolutely spherical, 
then $\hf_I V^\infty$ is 
closed in $V^\infty$ for each boundary degeneration $\hf_I$ of $\hf$. 
\end{lemma}

\begin{proof} The first statement is a result from \cite{AGKL}.  
For the second statement, we observe that $Z_I$ is 
absolutely spherical whenever $Z$ is 
(see Remark \ref{Z_I abs sph}).
\end{proof}

\begin{theorem}\label{lemma ub} 
Assume that $Z$ is absolutely spherical. 
Let $(V,\eta)$ be a spherical pair and  $\Omega\subset G$ a compact subset. 

\begin{enumerate} 
\item\label{ub-one} For every $X\in \Sphere(\af_Z^-)$ there exists a $G$-continuous norm $q$
such that 
\begin{equation*}  |m_{v,\eta}(\omega \exp(tX) w\cdot z_0)|\leq q(v) 
e^{t\lambda_ X} t^{d_X}  \qquad (t\geq 1)\end{equation*}
for every $v\in V^\infty$, $\omega\in \Omega$ and $w\in \W$. 
\item\label{ub-two}  Let $\mu:=\sum_{\sigma\in S} \sigma\in \af_Z^*$. Then for all $\e>0$ there 
exists a $G$-continuous norm $q$ such that 
\begin{equation*} |m_{v,\eta}(\omega aw\cdot z_0)|\leq  q(v)a^{\Lambda_{V,\eta} - \e\mu} 
\qquad (a\in A^-_Z)\end{equation*}
for every $v\in V^\infty$, $\omega\in \Omega$ and $w\in \W$. 
\end{enumerate}
\end{theorem} 
\begin{proof}  For a $G$-continuous norm and a compact subset $\Omega\subset G$, there exists 
a constant $C>0$ such that $p(\omega v) \leq Cp(v)$ holds for all 
$\omega\in \Omega$ and all $v\in V^\infty$. In view of $m_{gv, \eta}(z) = m_{v,\eta}(g^{-1} z)$ for 
$g\in G$, this reduces the statements in 
(\ref{ub-one}) and (\ref{ub-two})  to $\omega=\1$. 
Likewise, with (\ref{qetaw}) and Cor.~\ref{Zw}
we can reduce to $w=\1$ by applying
that case
to $\eta_w$. 
 
\par (\ref{ub-one}) Let $X\in \Sphere(\af_Z^-)$ be given and
let $I:=\{ \sigma\in S\mid \sigma(X) =0\}$. 
Recall that then $X\in \af_I$ and that
$\af_I$ normalizes $\hf_I$. 
\par We infer from Lemma \ref{comparison smooth}  that 
$\hf_I V ^\infty \subset V^\infty$ is closed. 
In particular the dual of $U:= V^\infty/ \hf_IV^\infty$ is $(V^{-\infty})^{H_I}$ and hence finite 
dimensional. Thus $U$ is a finite dimensional $\af_I$-module. 

\par To simplify the exposition we pretend that all root spaces $\gf^\alpha$ 
are one dimensional. Fix a basis of root vectors $(X_{-\alpha})_\alpha $ of $\oline \uf$, 
and set $Y_{-\alpha}:= X_{-\alpha} + T_I(X_{-\alpha}) \in \hf_I$. Further let 
$Y_j$ be a basis for $\lf\cap \hf$. Then the $Y_{-\alpha}$ together with the 
$Y_j$ form a basis for the  real spherical subalgebra $\hf_I$.  Note that 
\begin{equation} \label{AdaY} 
\Ad(a) Y_{-\alpha} = a^{-\alpha} Y_{-\alpha}\qquad (a\in A_I)\end{equation}  

\par From the fact that $\hf_I V^\infty $ is closed in $V^\infty$ we also infer the following. 
Let $p$ be
a $G$-continuous norm on $V^\infty$. Then by the open mapping theorem
there exists a $G$-continuous norm $q$ such that every $v\in \hf_I V^\infty$ admits a
presentation $v=\sum_{\alpha} Y_{-\alpha} u_\alpha + \sum_j Y_j u_j$ where $u_j, u_\alpha \in V^\infty$ and
\begin{equation} \label{p to q} \sum_\alpha p(u_\alpha) +\sum_j p(u_j) \leq q (v)\, .\end{equation}

\par We now start with the proof proper. As in the beginning of the proof of Th. 3.2 in \cite{KSS}
we begin with a crude estimate. There exist a $G$-continuous norm $p$ and a constant $\mu\in \R$
such that
\begin{equation}\label{crude}
|m_{v, \eta} (\exp(tX) \cdot z_0) |\leq e^{\mu t} p(v)\qquad (v\in V^\infty, t\geq 0).
\end{equation}
Let $\lambda_0=\lambda_0(X)$ be the infimum of the set of all $\mu$
for which such an estimate is valid for some
$G$-continuous norm $p$. 

In the first step we consider vectors $v\in \hf_IV^\infty$ and  let $v=
\sum_{\alpha} Y_{-\alpha} u_\alpha + \sum Y_j u_j$
be a presentation as above.   Then 
we compute  with (\ref{AdaY}) for all $a\in A_I$
$$ m_{v,\eta} (a\cdot z_0) =  
\sum_{\alpha} \eta(a^{-1}Y_{-\alpha} u_\alpha)
  = \sum_\alpha a^\alpha \eta (Y_{-\alpha} a^{-1}u_\alpha).$$
Note that $Y_{-\alpha} + \sum_{\alpha+\beta \not \in \la I \ra} 
X_{\alpha,\beta}\in \hf$. Hence
\begin{align} m_{v,\eta} (a\cdot z_0) &=  
- \sum_\alpha\sum_{\alpha+\beta\not \in \la I\ra} a^\alpha \eta (X_{\alpha, \beta} a^{-1} u_\alpha)\cr
  &= -\sum_\alpha\sum_{\alpha+\beta\not \in \la I\ra} a^{\alpha+\beta} \eta( a^{-1} X_{\alpha, \beta} u_\alpha)
  \, .\label{m(a)}\end{align}
Notice that $a^{\alpha+\beta}\leq 1$ if $a\in A_Z^-$.
We now specialize to $a=\exp(tX)$ with $t\ge0$. Let 
\begin{equation}\label{little c} c=c(X):= \max_{\alpha+\beta \not \in \la I\ra} 
(\alpha +\beta)(X)\end{equation}
and note that that $c<0$ as $X$ determines $I$. 

Let $p_1$ be a first order Sobolev norm of $p$, and let $q_1$ be related to $p_1$ as in
(\ref{p to q}). Then by (\ref{crude}) and (\ref{m(a)})
\begin{align*} |m_{v,\eta}(\exp(tX) \cdot z_0)|&\leq e^{t(c+\mu)} \sum_{\alpha,\beta} p(X_{\alpha,\beta} u_\alpha)\cr
& \leq C e^{t(\mu+c)} \sum_\alpha p_1(u_\alpha) \leq C e^{t(\mu +c)} q_1(v) \end{align*}
with $C>0$ a constant determined from the coordinates of $X_{\alpha,\beta}$
in terms of the basis for $\gf$ used in the Sobolev norm.

This brings us to the following improved estimate for elements from
$\hf_IV^\infty$. There exists for each $p$ and $\mu$ satisfying (\ref{crude})
a $G$-continuous norm $p'$ such that 
\begin{equation}\label{mu2}  |m_{v,\eta}(\exp(tX))| \leq e^{t(\mu+c) } p'(v) \qquad (t\geq 0, 
v\in \hf_IV^\infty)\, .
\end{equation}

Let $\E=\E_X$
be the spectrum of $X$ on $U$.
We write $U=\bigoplus_{\lambda\in \E}  U[\lambda]$ for the decomposition 
into generalized $X$-eigenspaces, and correspondingly $u=\sum u_\lambda$ for $u\in U$. 
For $v\in V^\infty$ let $[v]=v+\hf_I V^\infty$ 
be its equivalence class in $U$. 
Let $$V^\infty[\lambda]=\{v\in V^\infty\mid [v]\in U[\lambda]\}$$
for $\lambda\in\E$.

\par Let $\lambda\in\E$ and let $v\in V^\infty[\lambda]$.
Define $$v_0:=v,\qquad v_i:=(X-\lambda)v_{i-1}, \quad i=1,2,\dots.$$ 
Let $d$ be the last value of $i$ for which
$[v_i]\neq 0$.
Then $X$ is represented by a
lower triangular Jordan matrix
on the invariant subspace $\C^{d+1}\simeq Span_\C\{[v_0], \ldots, [v_d]\}\subset U[\lambda]$.
We denote by $B$ the transpose of that matrix. 
{}It follows form the definition of $v_i$ that 
\begin{equation}\label{mu4}  
p(v_{i}) \leq C p_{i}(v)  \quad 
(i=0,1,\dots)
\end{equation} for all $v\in V^\infty$, with a constant $C>0$
which depends on the coordinates of $X-\lambda$ 
in terms of the basis for $\gf$ used in the Sobolev norm $p_i$.

Now consider the $\C^{d+1}$-valued functions  
$$F(t)= (m_{v_0, \eta}(\exp (tX)), \ldots, m_{v_d, \eta}(\exp (tX))$$ and 
$$R(t)=-(0,\ldots,0,m_{w, \eta}(\exp(tX))$$
where $w=v_{d+1}\in\hf_IV^\infty$. Then $F$ satisfies  the ordinary 
differential equation 
$$ F'(t) = - B F(t) + R(t)\, .$$
The general solution formula then gives 
\begin{equation} \label{solution} 
F(t)=  e^{-tB} F(0)  + e^{-tB} \int_0^t e^{sB } R(s)\  ds\, .\end{equation} 
{}From (\ref{mu2}) and (\ref{mu4}) we infer that  for each $\mu>\lambda_0$
there exists a $G$-continuous norm $p$, independent of $v$, such that 
\begin{equation*}
|F(t)|\leq  p(v) \max \{ (1+t)^{d} e^{-\re\lambda t}, e^{(\mu + \frac c2)t}\}\, . 
\end{equation*}
In particular, this estimate applies to  $m_{v,\eta}$.

We conclude that for each pair of $p$ and $\mu$ satisfying (\ref{crude}),
there exists a $G$-continuous norm $p''$ such that
\begin{equation} \label{mu5} |m_{v,\eta} (\exp(tX))|\leq  \max \{
e^{-t \re \lambda } (1+t)^{d},
 e^{(\mu + \frac c2)t}\}
p''(v)\quad (t\ge 0)
\, .\end{equation}
for all $\lambda\in\E_X$ and all $v\in V^\infty[\lambda]$.
Here $d+1$ is determined as the maximal possible length of a Jordan block of $X$ on $U[\lambda]$. 

\par Before we move on we recall the following standard  fact from functional analysis. 
Let $E$ be a Fr\'echet space and $F\subset E$ be a closed subspace, then $E/F$ is a Fr\'echet space. 
Moreover, if $(p^n)_n$ is a sequence of semi-norms which define the topology on $E$, then 
$[p^n](v+F):=\inf_{w\in F} p^n(v+w)$ is a family of semi-norms which defines the topology on $E/F$.   

Our previous discussion shows 
that for every
$G$-continuous norm $p$ on $V$ there is a $G$-continuous norm $q$
such that 
for all $v\in V^\infty$ we find 
$v_\lambda \in V^\infty[\lambda]$  with 
$[v_\lambda]=[v]_\lambda$ and
\begin{equation}\label{mu3}  p(v_\lambda) \leq q(v)  \qquad (\lambda\in \E)\, .\end{equation}
The continuity expressed by this will be used to combine estimates for
$v\in V^\infty[\lambda]$ with different $\lambda\in\E$ into an estimate for all $v\in V^\infty$.

We shall now split $\E$ into the disjoint union 
$$\E=\E_-\cup \E_0\cup \E_+$$
of elements $\lambda$ with $-\re\lambda<\lambda_0$, $-\re\lambda=\lambda_0$, and $-\re\lambda>\lambda_0$,
respectively.

For $\lambda\in\E_-$ we obtain from (\ref{mu5}) that 
there exists a $G$-continuous norm $p$ such that
\begin{equation} \label{mu7}
|m_{v,\eta} (\exp(tX))|\leq  
 e^{\nu t}
p(v)\quad (t\ge 0, v\in V^\infty[\lambda])
\, \end{equation}
for some $\nu<\lambda_0$.

Let
$\lambda\in \E_+\cup\E_0$ and let $\mu+\frac c2<\lambda_0<\mu$. 
Then 
it follows from (\ref{mu2}) that the integral 
$\int_0^\infty e^{sB} R(s) \ ds$ is absolutely convergent. 
Hence 
$$c^\infty=c^\infty(v):= \lim_{t\to \infty} e^{tB} F(t)=F(0)+\int_0^\infty e^{sB} R(s) \ ds \in \C^{d+1}$$
exists and satisfies
$|c^\infty(v)|\leq p(v)$ for a $G$-continuous norm $p$.
The solution formula (\ref{solution}) can then be rewritten as 
\begin{equation} \label{solution2}  F(t)=  e^{-tB}c^\infty -e^{-tB}\int_t^\infty e^{sB } R(s)\  ds\, .\end{equation} 
Moreover, we also obtain
from (\ref{mu2}) that
\begin{equation}\label{qwerty} 
\big|e^{-tB}\int_t^\infty e^{sB } R(s)\  ds\big| 
=
\big|\int_0^\infty e^{sB } R(s+t)\  ds\big| 
\le Cp(v) e^{(\mu+\frac c2)t}
\end{equation}
for a constant $C>0$.

In particular, the estimate (\ref{qwerty}) applies to $|m_{v,\eta}(\exp(tX)|$
when $c^\infty=0$ and yields 
\begin{equation} 
\label{mu8} 
|m_{v,\eta} (\exp(tX))|\leq e^{(\mu+\frac c2)t} p(v)\qquad (t\geq 0).
\end{equation}
It follows that $c^\infty$ cannot be zero
for all $\lambda\in\E_0\cup \E_+$ and all $v\in V^\infty[\lambda]$
as in that case
we would
reach from (\ref{mu7}) and (\ref{mu8}) a contradiction with the definition of $\lambda_0$.

Let now $v\in V^\infty[\lambda]$ be such that
$c^\infty\neq 0$, 
say $c^\infty= (c_0^{\infty},\ldots, c_k^{\infty},0, \ldots, 0)$ with 
$c_k^\infty\neq 0$. Then
$$|e^{-tB}c^{\infty}|\sim c _k^\infty e^{-t \re \lambda} 
\frac{t^{k}}{ k!} $$ for $t$ large.
Hence with (\ref{solution2}) and (\ref{qwerty})
we obtain for all $v\in V^\infty[\lambda]$
\begin{equation} 
\label{mu6} 
|m_{v,\eta} (\exp(tX))|\leq e^{-t \re \lambda } (1+t)^{d_\lambda}p(v)\qquad (t\geq 0).
\end{equation}
Here $d_\lambda$ is the maximal possible value of $k$
among all $v\in V^\infty[\lambda]$.
In addition, we have for some $v\in V^\infty[\lambda]$ a non-trivial lower bound 
\begin{equation} \label{mu10} 
|m_{v,\eta}(\exp(tX))| \geq C_0 t^{d_\lambda} e^{-t\re \lambda}
\end{equation}
for $C_0>0$ and  $t$ sufficiently large. 
It follows that
$-\re\lambda\le \lambda_0$ and hence $\lambda\in\E_0$.
In particular, we conclude that (\ref{mu8}) applies to all $\lambda\in\E_+$. 

By combining (\ref{mu7}) for $\lambda\in\E_-$, (\ref{mu8}) or (\ref{mu6})
for $\lambda\in\E_0$, and (\ref{mu8}) for $\lambda\in\E_+$, and using
(\ref{mu3}) we see that
\begin{equation} 
\label{mu9} 
|m_{v,\eta} (\exp(tX))|\leq e^{t\lambda_0} (1+t)^{d_0}p(v)\qquad (t\geq 0).
\end{equation}
for all $v\in V^\infty$, where
$d_0$ is the maximal value of $d_\lambda$ for $\lambda\in\E_0$. 
Moreover, by (\ref{mu10})
\begin{equation*} 
|m_{v,\eta}(\exp(tX))| \geq C_0 t^{d_0} e^{t\lambda_0}
\end{equation*}
for some $v\in V^\infty$.
We conclude from
(\ref{ub-exact}) that $C_0t^{d_0} e^{t\lambda_0}\le Ct^{d_X} e^{t\lambda_X}$.
and thus
(\ref{mu9}) implies the statement in (\ref{ub-one}).
Note that then $\lambda_0=\lambda_X$ by Proposition \ref{exact lb}. 

\par (\ref{ub-two})  Recall that $|m_{v,\eta}(abw\cdot z_0)|= |m_{v,\eta}(aw\cdot z_0)| b^{\Lambda_{V,\eta}}$ for all $b\in A_{Z,E}$, $a\in A_Z$ and $w\in \W$. 
Hence we may assume that $\af_{Z,E}=0$ for our purpose. 

\par For $\delta>0$ and $I\subsetneq S$ we set 
$$\Sphere_\delta(\af_Z^-\cap \af_I):=\{ X\in \Sphere(\af_Z^-)\cap \af_I\mid (\forall \alpha \in S\bs I)  - \alpha(X) \geq \delta\}\, .$$ 
Then for each $I\neq S$ and each $\delta>0$ 
we obtain  an estimate as in (\ref{ub-one}) 
with a $G$-continuous norm $q=q_\delta$ which is uniform for all $X\in \Sphere_\delta(\af_Z^-\cap\af_I)$. 
This is because in the proof of (\ref{ub-one}) only the constant $c$ from (\ref{little c}) enters 
in the quantification of $q$. Indeed, $q$ was constructed from the initial estimate (\ref{crude}), 
which is valid for all $X$, in a number of steps each of which lowered the  exponent $\mu$ by $c/2$.
In each step a new norm $p'$ was constructed from a previous one $p$,
and this construction can be done independently of $X$, as long as it remains in the bounded set  $S(\af_Z^-)$.
Since  the constant $\mu$ in the initial estimate (\ref{crude}) can be chosen independently of $X$,
and since $\lambda_0(X)=\lambda_X$ is uniformly bounded below by continuity,
the maximal number of steps depends only on $c$.

Further, for $|S\bs I|=1$, i.e.~$\af_I$ is one-dimensional, 
we note that $\Sphere_\delta(\af_Z^-\cap\af_I)= \Sphere(\af_Z^-)\cap\af_I$
if $\delta $ is small enough. Hence in this case we
have a priori estimates which are uniform for all $X\in \Sphere(\af_Z^-)\cap \af_I$.
The proof then proceeds by induction on
$\dim\af_I$. We explain the proof for the final step where 
$X\in \Sphere(\af_Z^-)\cap \af_I$  for $\af_I=\af_Z$. 
By the induction hypothesis we then have uniform estimates for all
$X\in \Sphere(\af_Z^-)\cap \af_I$
with $|I|\ge 1$. 

Let $\delta>0$ be fixed. It will be described at the end of the proof how it is chosen from the given
$\epsilon$. As explained above we have a uniform estimate for 
$X\in \Sphere_\delta(\af_Z^-)$. For the remaining $X\in \Sphere(\af_Z^-)$ 
we have $-\delta< \alpha(X)\le 0$ for 
at least one $\alpha\in S$ and hence  
we can write $X= X_1 + X_2$ 
with $X_1\in S(\af_Z^-)\cap\af_I$ for some $I\subset S$
with $|I|\ge 1$ and $X_2\in \af_Z$ small relative to $\delta$, 
i.e.~$\|X_2\|< C \delta$ with a constant $C$ independent of $\delta$. 

Observe that
$m_{v,\eta} (\exp(tX) \cdot z_0)= m_{\exp(tX_2) v,\eta}(\exp(tX_1)\cdot z_0)$
and thus we obtain from (\ref{ub-one}) that  for $t\ge1$
$$|m_{v,\eta} (\exp(tX) \cdot z_0)|\leq q(\exp(tX_2)v) e^{t \lambda_{X_1}} t^{d_{X_1}}\, . $$ 
It follows from (\ref{weta}) that
$\lambda_{X_1}\leq \Lambda_{V,\eta}(X_1)$.
Since $q$ is $G$-continuous we have 
$q(\exp(tX_2) v) \leq e^{t \|X_2\| c_q} q(v)$ with $c_q>0$ a constant depending on $q$. 
Hence (\ref{ub-two}) follows if
$$|\Lambda_{V,\eta}(X_2)|+c_q\|X_2\| < -\epsilon\mu(X),$$
and this can be attained uniformly with a proper choice of $\delta$ since
$-\mu$ is bounded below on $S(\af_Z^-)$.
\end{proof}

\begin{rmk}\label{comparison} The proof of (\ref{ub-one}) shows that $\lambda_X$
is minus the real part of an 
eigenvalue of $X$ on the finite dimensional $\af_I$-module $U=V^\infty/ (\hf_I)_w V^\infty$
for some
$w\in\W$.
\end{rmk}

\begin{rmk}\label{strong bound remark} 
In the group case $Z=G\times G/ G\simeq G$ the
existence of the norm $q$ in Theorem \ref{lemma ub}(\ref{ub-two}) makes the statement considerably 
stronger than the corresponding result in \cite{W}, Thm.~4.3.5, because it implies
a bound for matrix coefficients $m_{\tilde v,v}$ on $G$, which is simultaneously
uniform with respect to the two smooth vectors $v$ and $\tilde v$. We refer to \cite{KSS} 
Thm.~5.8, where the generalization is obtained for symmetric spaces
and subsequently explicated for the group case.
\end{rmk}

\subsection{The wave-front case} 

For this case we replace the smooth globalization $V^\infty$
by the analytic globalization $V^\omega\subset V^\infty$.
We briefly recall its construction (see \cite{KKSS2} Remark 6.6).  
Let $p$ be any $G$-continuous norm on $V$.  Then 
there is a family of analytic norms (not necessarily $G$-continuous) $(p_\e)_{\e >0}$ on $V$, such that 
$V^\omega=\varinjlim V^\e$ with $V^\e$ the completion of $V$ with respect to $p_\e$. 
These norms feature the following properties: 
\begin{itemize}  
\item[(a)] For $\e\leq \e'$ one has $p_\e\leq p_{\e'}$ and continuous  inclusion $V^{\e'}\to V^\e$. 
\item[(b)] For each compact set $\Omega\subset G$ and $\e>0$ there exist $0<\e'\leq \e$ and a constant 
$C$ such that 
$p_{\e'}(g \cdot v)\leq  C p_\e(v)$ for all $v\in V$ and $g\in \Omega$.
\item[(c)] For each $G$-continuous norm $q$ and $\e>0$, there exists a constant $C$ such that 
$q\leq Cp_\e $. 
\end{itemize}

Let $F\subset \Pi$ and let
$P_F= G_F  U_F$ be the associated parabolic subgroup
as described in Section \ref{section induced}.
For any Harish-Chandra module $V$ we recall that $V/\oline{\uf}_F V$ is a Harish-Chandra module 
for the pair $(\gf_F, K_F)$ and as such has an analytic globalization $(V/ \oline{\uf}_FV)^\omega$ as 
$G_F$-representation.

\begin{lemma} \label{comparison analytic} Let $V$ be a Harish-Chandra module. 
Then $\oline{\uf}_F V^\omega$ is closed in $V^\omega$. 
In particular,  $V^\omega/ \oline{\uf}_F V^\omega = (V/ \oline{\uf}_FV)^\omega$. 
\end{lemma}

\begin{proof} Since the analytic globalization $V^\omega$ coincides with the minimal 
globalization of Kashiwara-Schmid, 
this follows from the version 
of the Casselman comparison theorem established in \cite{Brat} Thm.~1 or \cite{BO} Thm.~1.3. 
\end{proof}

\begin{theorem}\label{lemma ub wf} Assume that $Z$ is wave-front. 
Let $V$ be a Harish-Chandra module, $(V,\eta)$ a spherical pair,
and $p$ a $G$-continuous norm on $V$.
Then there exists a constant $d\in\N$ such that the following holds
for each compact subset $\Omega \subset G$ and each $X\in \Sphere(\af_Z^-)$.

For every $\e>0$ 
there exists a constant $C>0$ such that
\begin{equation*} |m_{v,\eta}(\omega a\exp (tX) w\cdot z_0)|\leq  C p_\e(v) 
e^{t \lambda_X} t^{d_X} a^{\Lambda_{V,\eta} }  (1 +\|\log a \|)^d \end{equation*}
for all $v\in V$, $t\geq 1$, $a\in A_Z^-$, $\omega\in \Omega$ and $w\in \W$. 
\end{theorem}

\begin{proof} By using the property (b) for the norms $p_\e$ we reduce to $\omega=\1$ as in 
the proof of Theorem 
\ref{lemma ub}. Likewise we can assume that $w=1$. 
 
\par Each element $X\in \Sphere(\af_Z^-)$ determines a set
$I:=\{ \sigma\in S\mid \sigma(X) =0\}$. As we assume that $Z$ is wave-front 
we obtain a parabolic $\pf_{F}=\pf_{F_I}$
from Proposition \ref{sandwich1}.  In particular, there exists an element $Y\in \af^-$ such that 
$Y+\af_H =X$ and $F=\{\alpha\in \Pi \mid \alpha(Y)=0\}$. Moreover, for every $\sigma=\alpha+\beta\in {\mathcal M}$ with 
$\alpha\in \Sigma_\uf \bs \la F\ra $ we record
\begin{equation}\label{negative}  (\alpha +\beta)(X) = (\alpha+\beta)(Y)<0\, .\end{equation}

The proof is quite analogous to the one for 
Theorem \ref{lemma ub}(\ref{ub-one}).
We confine ourselves to the steps where it differs. The main difference 
consists of replacing $U=V^\infty/\hf_I V^\infty$
by $U:=V^\omega/\oline{\uf}_F V^\omega$, and applying
Lemma \ref{comparison analytic} in place of 
Lemma \ref{comparison smooth}. 
We observe that as $U$ is a Harish-Chandra module 
for $\gf_F$, it admits a finite decomposition into generalized 
eigenspaces for $\af_F$.

\par As before we pretend that all 
root spaces $\gf^\alpha$ 
are one dimensional and  fix a basis of root vectors 
$(X_{-\alpha})_\alpha $ of $\oline {\uf_F}$. 
The expression (\ref{AdaY}) 
is just replaced by the corresponding
expression for $\Ad(a)X_{-\alpha}$ where $a\in A$.

\par 
The inequality (\ref{p to q}) 
has the following analytic counterpart.
For all $\e>0$ there exist $\e'>0$ and a constant $C_\e>0$ 
such that every $v\in \oline{\uf}_F V$ admits
a presentation $v = \sum_\alpha X_{-\alpha} u_\alpha$  with $u_\alpha\in V$ and
\begin{equation} \label{p to q 2} 
\sum_{\alpha} p_{\e'}(u_\alpha) \leq C_\e p_{\e} (v)\, .\end{equation}
This follows from Lemma \ref{comparison analytic} by the argument of
\cite{KKSS2} Remark 6.6
(note that
$\e$ and $\e'$ need to be interchanged in (6.5) of that remark).

We replace the initial estimate (\ref{crude}) by the following, which is
derived from it with property (c) for the norms $p_\e$.
For some $\lambda\in\af_Z^*$, some $\mu\in\R$, and
all $\e'>0$ there exist a constant $C$ such that
\begin{equation}\label{new crude}
|m_{v,\eta} (a \exp(tX)\cdot z_0)|\leq C e^{t \mu} a^\lambda p_{\e'}(v)  \qquad (t\geq 0, a\in A_Z^-).
\end{equation}

Let $v\in {\oline {\uf}_F} V$.
With (\ref{T}) and
the above presentation of $v$ 
we obtain as in (\ref{m(a)}) 
\begin{equation}
m_{v,\eta} (a\cdot z_0) =-\sum_\alpha\sum_{\beta} a^{\alpha+\beta} \eta( a^{-1} X_{\alpha, \beta} u_\alpha)
  \, .\label{mc wf}
  \end{equation}
for all $a\in A_Z$. Here $\alpha+\beta\in\mathcal{M}$ and 
hence $a^{\alpha+\beta}\leq 1$ if $a\in A_Z^-$.
In addition $\alpha\notin\la F\ra$, and hence by (\ref{negative}) there exists $c<0$ such that
$(\alpha +\beta)(X)\le c$
for all terms in (\ref{mc wf}). 

{}From the crude estimate (\ref{new crude}) we then obtain 
$$ 
|m_{v,\eta}(a\exp(tX)\cdot z_0)| \leq C a^\lambda e^{t(\mu+c) } 
\sum_{\alpha,\beta}p_{\e'}( X_{\alpha, \beta} u_\alpha) \quad (t\geq 0, a\in A^-_Z). 
$$
Since $p_{\e'}$ is analytic there exists $C'>0$ such that
$p_{\e'}( X_{\alpha, \beta} u) \le C'
p_{\e'}( u) $ for all $u\in V$, 
and from (\ref{p to q 2}) we finally
derive 
the following improved estimate, analoguous to (\ref{mu2}). 

There exists for each $\mu$ satisfying (\ref{new crude}) and for each $\e>0$
a constant $C$ such that 
\begin{equation}\label{mu2 wf}  |m_{v,\eta}(a\exp(tX)\cdot z_0)| \leq C a^\lambda e^{t(\mu+c) } p_\e (v) \quad (t\geq 0, a\in A^-_Z)
\end{equation}  
for all $v\in \oline{\uf}_F V$.

Following the previous proof, with 
$\E=\E_X$ now denoting the spectrum of $X$ on the $\af_F$-finite module $U=V/\oline{\uf}_F V$, we arrive at the following
bound
\begin{equation*}  |m_{v,\eta}(a\exp(tX) \cdot z_0)|\leq  
C p_\e(v) 
a^\lambda e^{t\lambda_ X} t^{d_X}  \end{equation*}
for all $t\geq 1$, $\omega\in \Omega$, $ w\in \W$ and $a\in A_Z^-$.
Using the coordinates from (\ref{I-coordinates}) for $\af_Z^-$ we 
can now iterate with $X=-\omega_1,\dots,-\omega_s$ and 
optimize $\lambda$ to $\Lambda_{V,\eta}$ at the cost of a possible logarithmic term.
\end{proof}

\section{Discrete series}

All  $\eta\in (V^{-\infty})^H$ to be considered 
in the sequel are requested to be $A_{Z,E}$-eigenvectors, 
say $a\cdot\eta=a^\chi\eta$.
Recall from
(\ref{Lambda on edge}) that then
$\Lambda_{V,\eta}|_{\af_{Z,E}}=\re\chi$.

\par For $Z$ unimodular and an irreducible Harish-Chandra module $V$ we introduced in \cite{KKSS2} 
Definition 5.3 the notion of a {\it $Z$-tempered  pair}
$(V, \eta)$
and showed in Proposition 5.5 that 
this condition is satisfied by almost all pairs that contribute
to the Plancherel decomposition of $Z$. 
With \cite{KKSS2} Remark 5.4
$(V,\eta)$ is $Z$-tempered if and only if 
for all $v\in V$ there exists a constant $C_v>0$ and an $m \in \N$  such that 
\begin{equation} \label{tempinq} | m_{v, \eta} (\omega a w \cdot z_0)|\leq C_v a^{\rho_Q}  (1 +\|\log a\|)^m \end{equation}
for all $\omega\in \Omega$, $w \in \W$ 
and $a\in A_Z^-$.  Here 
$\rho_Q:=\rho_\uf \in \af_Z^*$ is defined by 
$2\rho_Q(X)=\tr \ad_\uf (X)$, and $\Omega$ denotes a compact set
for which (\ref{poldechat}) is valid. 

\begin{rmk} In (\ref{tempinq}) we can replace the constant $C_v$ by $q(v)$ with $q$ a  $G$-continuous
norm (see \cite{KKSS2}, Definition 5.3). Hence the notion of a tempered pair $(V,\eta)$ only depends on $V^\infty$ and 
not on the particular choice 
of the maximal compact subgroup $K$.\end{rmk}

We now give a criterion for temperedness in terms of $\Lambda_{V,\eta}$.
Let us first note that by (\ref{chi on m}) and
(\ref{Lambda on edge}) the temperedness (\ref{tempinq})
implies 
\begin{equation} \label{temp''}  (\Lambda_{V,\eta} - \rho_Q)|_{\af_{Z,E}}=0\,.
\end{equation}

\begin{theorem}\label{Lambda condition for tempered} 
Let $Z$ be a unimodular real spherical space
and $(V,\eta)$ an $H$-spherical pair.  Then {\rm (1)} implies {\rm(2)}
for the following assertions: 

\begin{enumerate}
\item\label{tempered} $(V,\eta)$ is tempered.
\item\label{Lambda inequality} 
$(\Lambda_{V,\eta} - \rho_Q)\big|_{\af_Z^-}\leq 0$.
\end{enumerate} 
Moreover, if $Z$ is wave-front, {\rm (1)} and  {\rm(2)}
are equivalent. 
\renewcommand{\labelenumi}{\theenumi}

\end{theorem}

\begin{proof}  It is  immediate from 
Theorem \ref{lemma ub wf} that 
(\ref{Lambda inequality}) implies (\ref{tempered})
when $Z$ is wave-front.
For the converse we use 
the lower bound in 
Proposition \ref{exact lb} 
with $Y_0=0$ and $X=-\omega_j$, together with (\ref{temp''}) and
\begin{equation}\label{Lambdaj}
\lambda_{-\omega_j}= \Lambda_{V,\eta}(-\omega_j), \quad
(1\leq j\leq s).\qedhere\end{equation}
\end{proof} 

\par Recall the coordinates (\ref{I-coordinates}) of $\af_Z$. In the context of
Theorem \ref{Lambda condition for tempered} we shall say that the pair $(V,\eta)$
satisfies the 
{\it strong inequality}  if
\begin{equation} \label{disc}  (\Lambda_{V,\eta} - \rho_Q)(\omega_j) > 0
\qquad (1\leq j\leq s) \end{equation} 
holds, together with (\ref{temp''}). 
This terminology is motivated by the relation of (\ref{disc}) to
square integrability, which will be shown in Theorem  \ref{CDS}.
For that we need to recall some 
results on half densities.

 \begin{rmk}  The strong inequality implies temperedness  for spaces which are absolutely  spherical 
or wave-front. In the latter case 
this is part of Theorem \ref {Lambda condition for tempered}.  Using Theorem \ref{lemma ub} instead of Theorem \ref{lemma ub wf}
in the above proof, the assertion follows in the absolutely spherical case as well.
\end{rmk}

\subsection{Densities on a homogeneous space}\label{densit}

In this section we consider a closed subgroup $H$ of a Lie group $G$ and 
the associated homogeneous space $Z=G/H$. In (\ref{defi Delta_Z})
we attached to $Z$ the modular 
character $\Delta_Z =|\det \Ad_{\gf/\hf}|^{-1}$ on $H$.

Let $p=\dim Z$
and fix $0\neq \Omega\in \bigwedge^p (\gf/ \hf)^*$ and observe that 
$$ \Ad^*(h) |\Omega| = \Delta_Z(h) |\Omega|\, .$$
We view $\bigwedge^p (\gf/ \hf)^*$ as a subspace of $\bigwedge^p \gf^*$ 
and extend $\Omega$ to a $G$-invariant differential $p$-form on $G$ via
$$ \Omega (g)  (d\lambda_g(\1)X_1, \ldots, d\lambda_g (\1) X_p)= \Omega(X_1, \ldots, X_p)$$
for $X_1, \ldots, X_p\in \gf$ and with $\lambda_g(x) =gx$ for $x\in G$, the left displacement by $g$ on $G$. 

Given a character $\xi$ of $H$ we set 
$$C^\infty(G; \xi):= \{ f\in C^\infty(G)\mid f(\cdot h) = \xi(h)^{-1} f(\cdot ) \ \hbox{for}\ h\in H\}\, .$$
We write $C_c^\infty(G;\xi)$ for the subspace of $C^\infty(G; \xi)$ which consists 
of functions which are compactly supported modulo $H$.

Now observe if $f\in C^\infty(G;\Delta_Z)$ 
then the density   
$$ |\Omega|_f(g)(d\lambda_g(\1)X_1, \ldots, d\lambda_g (\1) X_p)= f(g) |\Omega|(X_1, \ldots, X_p)$$
factors to a smooth density on $Z$ and every smooth density on $Z$ is of this type. 
In other words there is a natural identification 
\begin{equation*} 
C^\infty(G; \Delta_Z) \to \Dens^\infty(Z), \ \  f \mapsto 
|\Omega|_f,\end{equation*}
with the space $\Dens^\infty(Z)$ of
smooth densities of $Z$. This identification is observed to be $G$-equivariant.

\par In a moment we shall need certain spaces of 
smooth half-densities, which are defined as follows. 
A character $\xi: H \to \C^*$ for which
\begin{equation}\label{chi=sqrt}
|\xi|=\Delta_Z^{1/2},
\end{equation}
will be called a {\it normalized unitary} 
character. For such characters we define
on $C_c^\infty(G; \xi)$ 
by 
$$\la f, h\ra:= \int_Z |\Omega|_{f\oline{h}}\, $$ 
an inner product, which is preserved by the left regular action 
of $G$. 
We denote by $L^2(Z; \xi)$ 
the Hilbert completion of $C_c^\infty(G;\xi)$ with 
respect to this inner product,
and note that it is a unitary $G$-module for the left regular representation.
With the terminology of \cite{Mackey} it is the representation induced by
the unitary character $\Delta_Z^{-1/2}\xi$ of $H$.

With that, we can define the $\xi$-{\it twisted discrete series} for $G/H$ as the set
of (equivalence classes of) 
irreducible representations of $G$ which embed into $L^2(Z; \xi)$. 

\subsection{A criterion for discrete series}

Let us return to the set-up from before Section \ref{densit}
with $Z=G/H$ real spherical and unimodular. By a slight change 
from previous notation we denote in this section by $\hat H$ the 
{\it connected} Lie subgroup of $G$
with Lie algebra $\hat \hf$. Likewise we set $\hat Z= G/\hat H$. 
Unlike $Z$ the space $\hat Z$ is not necessarily unimodular. 
We write $\hat \Delta $ for the modular function of $\hat {Z}$. 

\begin{lemma} \label{lemdelta}
Let $ha \in \hat{H} $ with $a\in A_{Z,E}$ and $h\in H$. Then 
$$\hat\Delta(ha) = a^{-2\rho_Q}\, .$$
\end{lemma}

\begin{proof} As $Z$ is unimodular and $H$ is a normal subgroup of $\hat{H}$ we 
deduce that $H\subset \ker \hat \Delta$. 
Hence 
$$\hat \Delta(ha)= |\det \Ad_{\gf/ \hat{\hf} } (a)|^{-1}\, .$$
Finally observe that $\hf$ is isomorphic to $(\lf \cap \hf) +\oline{\uf}$ as an 
$A_{Z,E}$-module. The assertion follows. 
\end{proof}

 We extend the character $e^{-\chi}$ of $A_{Z,E}$ to
a character $\xi$ of $\hat H$, which is trivial on $H$, and 
note that it then follows from $a\cdot\eta=e^{\chi}\eta$ that
$m_{v,\eta}\in C^\infty(\hat Z;\xi)$. 
It follows from Lemma \ref{lemdelta} that $\xi$ is normalized unitary 
if and only if $\re\chi=\rho_Q$ (compare also (\ref{sqrt psi})).

\begin{theorem}\label{CDS} Let $Z$ be unimodular, $V$ a Harish-Chandra module, and let
$\eta\in (V^{-\infty})^H$ be an element 
with eigencharacter $e^\chi$ on $A_{Z,E}$, and let $\xi=e^{-\chi}$.
Then \ref{2nd} $\Rightarrow$ \ref{1st}
for the assertions
\renewcommand{\labelenumi}{\theenumi}
\begin{enumerate}
\renewcommand{\theenumi}{(\roman{enumi})}
\item\label{2nd} $\xi$ is normalized unitary
and $m_{v,\eta}\in L^2(\hat Z; \xi)$ 
for all $v\in V$. 
\item\label{1st} $(V,\eta)$ satisfies the strong inequality (see (\ref{disc})),
\end{enumerate}
and if $Z$ is absolutely spherical 
or wave-front, then \ref{1st} $\Rightarrow$ \ref{2nd}.
\end{theorem}

It is a consequence of \ref{2nd} that (the completion of) $V$ belongs to the
$\xi$-twisted discrete series of $\hat Z$. 

\begin{proof} Assume \ref{1st}. It follows from  (\ref{temp''}) 
that  $\re\chi=\Lambda_{V,\eta}=\rho_Q$ on $\af_{Z,E}$. Hence 
$\xi$ is normalized unitary and
$|\Omega|_v:=|\Omega|_{|m_{v,\eta}|^2}$ is a density 
on $\hat Z$ for each $v\in V$.  We need to show that it is integrable.  
Recall from (\ref{poldec}), applied to $\hat Z=G/\hat H$, that  
\begin{equation}\label{poldec hatZ}
\hat Z= \F'' K A_{\hat Z}^- \W \cdot \hat z_0
\end{equation}
where $\F''$ is a finite set. 
Keep in mind that $A_{\hat Z}^-= A_Z^-/A_{Z,E}$ and that we realize $A_{\hat Z} \subset A_Z$. 
Let $M_H=\big[Z_{K\cap H} (A_{\hat Z})\big]_e$ and note that $M_H \W \subset \W \hat H$. 
Hence the the polar map 
$$\Phi:  \F''\times K/M_H  \times A_{\hat Z}^-\times \W\to \hat Z\, $$
associated with (\ref{poldec hatZ}) is defined. 
We note that:
\begin{itemize}
\item$\Phi$ has generically invertible differential;
\item $\Phi$ has generically finite fibers,
\end{itemize}
and that the second item follows from the first, as $\Phi$ is, up to finite cover, 
the restriction of a dominant algebraic map
between complex varieties of the same dimension. 
\par Let $|\omega|_v:=\Phi^*(|\Omega|_v)$ be the pull-back of the density $|\Omega|_v$. 
 It follows from the observations above that $|\omega|_v$ is integrable
 if and only if $|\Omega|_v$ is.
 Let us now compute this pull back explicitely. 
 For a subspace $U\subset \gf$ we fix a basis 
$u_1, \ldots, u_p$ and set $\omega_U = u_1 \wedge \ldots\wedge u_p \in \bigwedge^p\gf$. 
Let $U$ be the orthogonal complement of $\mf_H$ in $\kf$. 
Define a function $J$ on $A_{\hat Z}$ by 
\begin{equation}\label{J(a)}
J(a) | \omega_\gf| = |\omega_U \wedge\omega_{\af_{\hat Z} } \wedge [\Ad(a) \omega_{\hat \hf}]|\, .\end{equation}
 Then there exists a constant $C>0$ (depending on the normalization of measures)
 such that 
 \begin{equation*} 
|\omega| _v (w', kM_H , a, w) = C |m_{v,\eta}(w'k aw\cdot z_0)|^2
J(a) |d| (a) \ |d|(kM_H)\, ,\end{equation*}
for $w'\in\F''$ and $w\in\W$, where 
$|d|(kM_H)$ and $|d| (a)$ are the Haar densities on $K/M_H$ and $A_{\hat Z}$ respectively.
We deduce from (\ref{J(a)}) and
(\ref{t-compr}) that $J(a) \leq C a^{-2\rho_Q}$ for all $a \in A_{\hat Z}^-$. 

If $Z$ is absolutely spherical 
or wave-front the integrability of $|\omega|_v$ now
follows from Theorem \ref{lemma ub}(\ref{ub-two}) and Theorem \ref{lemma ub wf}.
Hence \ref{2nd} is valid in that case.
\par  Assume \ref{2nd}.  We argue by contradiction and 
assume that (\ref{disc}) fails, i.e.~there exists $\omega_j$ such that 
$$ (\Lambda_{V,\eta} - \rho_Q)(\omega_j) \leq  0\, .$$
It suffices to show that $|\Omega|_v $ is not integrable for 
some $v$, 
when pulled back to $ K \times A_{\hat Z}^-$ as in the first part of the proof. 
\par Suppose that $v\in V$ generates an irreducible $K$-module $V_\tau$ 
with dimension $d_\tau$. Then Schur-orthogonality implies that 
$$\int_{K/M_H} |m_{v,\eta}(kaw\cdot z_0)|^2 \ d (kM_H)\geq 
\frac1{d_\tau}  |m_{v,\eta}(aw\cdot z_0)|^2 $$
and therefore 
\begin{equation} \label{lbw} \int |\omega_v| \geq \frac C{d_\tau} 
\int_{A_{\hat Z}^-}  |m_{v,\eta}( aw\cdot z_0)|^2 J(a)\  da\, 
.\end{equation} 
Next we need a lower bound on the Jacobian function $J(a)$.
We employ again (\ref{J(a)}) and (\ref{t-compr}), and deduce
for $a_0\in A_{\hat Z}^{-}$ sufficiently far from walls that
$$ J(a_0\exp(-t\omega_j)) \geq C e^{2t\rho_Q(\omega_j)}$$
for all $t\geq 0$ and a positive constant $C$. Hence it follows from (\ref{lbw}), 
Proposition \ref{exact lb} and the Fubini theorem that $|\Omega|_v$ is not square integrable for some $v$.
\end{proof}

\begin{rmk}\label{str ineq indept of K}
In case $Z$ is wave-front or absolutely spherical, 
the above theorem shows that if $(V,\eta)$ satisfies the strong inequality, 
then this is also the case if we replace $V$ by the space of $K$-fixed
vectors in $V^\infty$ for any other regular choice of  $K$.
\end{rmk}

\subsection{The interlaced case}

Recall that $H$ is said to be interlaced by $P$ if $U_P\subset H\subset P$,
and that in this case $Z_P=G_P/H_P\simeq P/H$ for the Levi-induced spherical space.

\begin{lemma}\label{induced}
If $H$ is interlaced by $P$ and $V$ embeds into
$L^2(Z,\xi)$ for some normalized unitary character $\xi$
with trivial restriction to $U_P$, then there exists a 
twisted discrete series $\sigma$ of $Z_P$ such that $V$ embeds into
the representation of $G$
unitarily induced from $\sigma\otimes 1$ of $P=G_PU_P$. 
\end{lemma}

\begin{proof} We may assume $V$ is irreducible.
Note that $L^2(Z,\xi)$ is unitary equivalent with the 
induced representation $\mathrm{ind}_H^G(\xi\otimes\Delta_Z^{-1/2})$. Induction by stages (see \cite{Mackey})
implies that this is unitary equivalent with 
$\mathrm{ind}_P^G\,\mathrm{ind}_H^P(\xi\otimes\Delta_Z^{-1/2})$.
Since $\xi$ is trivial on $U_P$, the
representations
$\mathrm{ind}_H^P(\xi\otimes\Delta_Z^{-1/2})$ 
and $\mathrm{ind}_{H_P}^{G_P}(\xi|_{H_P}\otimes\Delta_Z^{-1/2})
\otimes 1$
of
$P=G_PU_P$ are unitary
equivalent. Let
$$\xi_P=\xi|_{H_P}\otimes\Delta_Z^{-1/2}
\otimes\Delta_{Z_P}^{1/2}$$
then $L^2(Z_P,\xi_P)\simeq
\mathrm{ind}_{H_P}^{G_P}(\xi\otimes\Delta_Z^{-1/2})$. 
Hence $V$ embeds into the unitary representation
$\mathrm{ind}_{P}^{G} (L^2(Z_P,\xi_P)\otimes 1)$,
and this implies that it embeds in 
$\mathrm{ind}_{P}^{G} (\sigma\otimes 1)$
for some irreducible subrepresentation $\sigma$ of 
$L^2(Z_P,\xi_P)$.
\end{proof}

\section{The tempered spectrum of $Z$}\label{tempspec}

Let $(V,\eta)$ be a spherical pair. In this section we use the power series expansions of $m_{v,\eta}$ from Section 
\ref{power series} to define maps 
$$  
\eta \mapsto \eta_I^{\mu,w}, \qquad
(V^{-\infty})^H \to (V^{-\infty})^{(H_w)_I}, 
$$
where $\mu$ runs through certain elements of $\E$, the leading 
exponents of the $m_{v,\eta}$.
These maps can be thought of as  Radon-transforms for a single 
spherical representation -- analytically a Radon transform 
would be the map which takes a rapidly decaying function on $G/H$ to  
a function on $G/(H_w)_I$ by integrating out fibers 
$(H_w)_I/H \cap (H_w)_I$.
\par The coefficients of the power series expansion of $m_{v,\eta}$ are 
indexed by $\E \times \N_0[S]$.
Roughly speaking, discarding all the coefficients not belonging to 
$\{\mu\} \times \N_0[I]$ and evaluating the partial expansion 
at $w\in \W$  produces an $(\hf_w)_I$-invariant linear functional 
$\eta_I^{\mu,w}$ on $V$ (see Lemma \ref{etaI}). 
The difficult part is then to see that $\eta_I^{\mu,w}$ is continuous, i.e.~extends continuously to $V^\infty$.  Under the assumption 
that $Z$ is wave-front or absolutely spherical, this is achieved in 
Lemma \ref{continuous and nonzero} using the sharp 
bounds from Section \ref{quantitative bounds}.
\par Since $A_I$ acts on the finite dimensional space 
$(V^{-\infty})^{(\hf_w)_I}$, we obtain a finite dimensional $A_I$-module
$W$-generated by $\eta_I^{\mu,w}$.  Appropriate $A_I$-eigenvectors which 
we call $\eta$-optimal, then 
give the tempered embedding
theorem \ref{tet} for $(V,\eta)$ a tempered pair.

\subsection{Leading coefficient maps}

Let $(V,\eta)$ be an $H$-spherical pair as considered in Section \ref{power series},
and let $\Xi\subset\af^*_{Z,\C}$ be the associated set of exponents 
(see (\ref{defXi})).
We shall attach leading coefficient maps
$$\Lc_I: V \to C^\infty(U_AA_Z^-), $$
to each subset $I\subset S$. 

Let $\Xi_I=\{\xi|_{\af_I}\mid \xi\in\Xi\}$ be the set of exponents along $A_I$.
In analogy with ${\mathcal E}_{\rm lead}$ we let
${\mathcal E}_{{\rm lead},I}$ consist of those exponents $\mu\in\Xi_I$ for which
$\mu-\sigma|_{\af_I}\notin \Xi_I$ for all $\sigma\in \N_0[S]\bs \la I\ra$
(see (\ref{defElead})).

Then 
\begin{equation} \label{leadI} \emptyset\neq {\mathcal E}_{{\rm lead},I}\subset {\mathcal E}_{{\rm lead}}|_{\af_I}\, 
\end{equation}
and $$\Xi_I\subset {\mathcal E}_{{\rm lead},I}+\N_0[S]\bs\la I\ra.$$
For each $\mu\in {\mathcal E}_{{\rm lead},I}$ and each $w\in\mathcal W$ we define
$$ \Lc_I^{\mu,w}(v)(a):= 
\sum_{\xi\in\Xi,\, \xi|_{\af_I}=\mu} 
q_{v, \xi,w} (\log a) \,
a^{\xi} \, $$
for $v\in V$ and $a\in U_AA_Z^-$, where $q_{v,\xi,w}$ is defined by (\ref{asympXi}).

\begin{lemma}\label{TAYLOR} Let $\mu\in {\mathcal E}_{{\rm lead},I}$ and $w\in\W$.
The linear map $\Lc_I^{\mu,w} : V \to C^\infty(U_AA_Z^-)$ is $\af_Z$-equivariant.
\end{lemma}

\begin{proof} This follows from the uniqueness of the coefficients
in the asymptotic expansions of $m_{v,\eta}(aw\cdot z_0)$ and $\Lc_I^{\mu,w}(v)(a)$
(see also (\ref{c-relations})).
\end{proof}

Denote by $V^*$ the algebraic dual of $V$ and define 
$$\Lb_I^{\mu,w}: A_I \to V^*; \ \Lb_I^{\mu,w} (a) (v):= \Lc_I^{\mu,w}(v)(a)\, .$$
Notice that 
$$\Lb_I^{\mu,w}(a) (v) = a^\mu p_{v, \mu, w} (\log a)\qquad (a\in A_I)$$
for a polynomial $p_{v,\mu,w}$ on $\af_I$.

\begin{lemma}\label{etaI} Let $\mu\in{\mathcal E}_{{\rm lead},I}$ and 
$w\in\mathcal W$.
Then ${\rm im}\ \Lb_I^{\mu,w}\subset (V^*)^{(\hf_w)_I}$. 
\end{lemma} 

\begin{proof} 
For the proof we assume first that $w=\1$ is trivial. Recall that 
$$ \hf_I = (\lf \cap \hf) + 
\{ \oline X + T_I(\oline X)\mid \oline X \in \oline\uf\}$$ 
for the linear operator $T_I: \oline \uf \to   
(\lf \cap \hf)^{\perp_\lf}\oplus \uf$ given by (\ref{T_I}). 
\par First it is clear that $m_{Xv,\eta}|_{A_Z^-}=0$ for all 
$X\in \hf\cap \lf = \hf_I \cap \lf$. This implies readily that $\Lb_I^{\mu,\1}(a) \in (V^*)^{\lf\cap \hf}$
for all $a\in A_I$.  
\par Next for $\alpha\in \Sigma_\uf$ and $X_{-\alpha}\in\gf^{-\alpha}$, 
we note that 
\begin{equation} \label{T-relation} 
T(X_{-\alpha}) - T_I(X_{-\alpha}) = \sum_{\alpha+ \beta\notin \la I\ra} X_{\alpha, \beta}\end{equation}
with the notation of (\ref{T}) and (\ref{T_I}). 
 
Let  $X= X_{-\alpha} +T_I(X_{-\alpha})\in \hf_I$, and note that since
$T_I(X_{-\alpha})$ is a sum of root vectors $X_{\alpha, \beta}\in\gf^\beta$
with $\alpha+\beta\in \la I\ra$, we have $\Ad(a)X=a^{-\alpha}X$ for $a\in A_I$.
Hence for $v\in V$
\begin{equation*} m_{X\cdot v, \eta}(a\cdot z_0)= 
\eta(a^{-1} \cdot X\cdot v)= 
a^{\alpha} \eta(X a^{-1}\cdot v). \end{equation*}
Let $Y=T(X_{-\alpha})-T_I(X_{-\alpha})\in \uf$.
Then $X+Y\in\hf$ and for $a\in A_I$ we obtain from above
\begin{equation*} 
m_{X\cdot v, \eta}(a\cdot z_0) = - a^\alpha \eta(Y a^{-1}\cdot v)
=  -a^\alpha m_{(\Ad(a) Y)\cdot v, \eta} (a\cdot z_0)\, .\end{equation*}  
Inserting the root vector expansion (\ref{T-relation}) we obtain
\begin{equation} \label{mXveta1}
m_{X\cdot v, \eta}(a\cdot z_0)= -\sum_{\alpha+ \beta\notin \la I\ra} 
a^{\alpha+\beta}m_{X_{\alpha,\beta}\cdot v, \eta} (a\cdot z_0)\, .\end{equation}  
By inserting for the matrix coefficients on the right hand side
of (\ref{mXveta1}) their expansions according to (\ref{asympXi}),
we see 
that on $A_I\cap U_A A_Z^-$ the function $m_{X\cdot v, \eta}(a\cdot z_0)$ allows an 
expansion with exponents which are of the form $\xi|_{\af_I}$ with $\xi=\xi'+\sigma$
where $\xi'\in\Xi$ and $\sigma=\alpha+\beta\in \N_0[S]\bs \la I\ra$.
By the definition of ${\mathcal E}_{{\rm lead},I}$ we get $\xi|_{\af_I}\neq \mu$ and hence we conclude that 
$\Lb_I^{\mu,\1}(a)(Xv)=0$ for $a\in A_I$. 
\par Finally, the case of an arbitrary $w$ proceeds along the same lines 
in view of the description of $\hf_w$ in (\ref{Tw}). 
\end{proof}

Let $W$ be a finite dimensional vector space and $\mu \in (\af_I)_\C^*$. Then we call $f: A_I \to W$ a 
$\mu$-polynomial map provided that for all $\lambda\in W^*$ there exists a polynomial $p_\lambda$ on $\af_I$ such that 
$(\lambda\circ f)(a) = a^\mu p_\lambda(\log a)$ for all $a\in A_I$.  Notice that there is a natural action of $A_I$ on the space of all
$\mu$-polynomial maps and accordingly we obtain an action of $\U(\af_I)$. 
The following lemma is then straightforward from the definitions. 

\begin{lemma} \label{mu polynomial} Let $f: A_I\to W$ be a $\mu$-polynomial map. 
\begin{enumerate} 
\item If $f\neq 0$, then there exists an $u\in \U(\af_I)$ such that $ u\cdot f\neq 0$ and $(u\cdot f)(a) = a^\mu( u \cdot f)(\1)$ for all 
$a\in A_I$.
\item Suppose in addition that $W_0\subset W$ is a subspace with $\operatorname {im}f \subset W_0$. 
Then $\operatorname{im} (u\cdot f) \subset W_0$ for all $u\in \U(\af_I)$. 
\end{enumerate}
\end{lemma}
 
 We now apply this to the specific situation where $W=(V^*)^{(\hf_w)_I}$. Observe that 
 $W\simeq (V/ (\hf_w)_I V)^*$ is a finite dimensional $A_I$-module by Lemma \ref{finite generation}, and that 
in Lemmas \ref{TAYLOR} and \ref{etaI} we have established that 
 $\Lb_I^{\mu,w}: A_I \to W$ is an $A_I$-equivariant $\mu$-polynomial map. 
 There are two invariant 
 subspaces which are in our interest:  $(V^{-\infty})^{(\hf_w)_I}\subset (V^{-\omega})^{(\hf_w)_I}$. 
\par We already know from Lemma \ref{mu polynomial}
that there exists an element $u\in \U(\af_I)$ such that $u\cdot \Lb_I^{\mu,w}$ transforms 
as a character.  Since we want to keep track of certain symmetries we need to construct $u$ explicitly via an 
iterative procedure. This will be done in (\ref{eta I mu w}). To prepare for
that let $X\in \af_I$ be such that $\sigma(X)<0$ for all $\sigma\in S\bs I$ and sufficiently generic such that
$\im\lambda(X)\neq \im\lambda'(X)$ for all pairs $\lambda\neq\lambda'$ in $\E_{{\rm lead},I}$
with  $\re\lambda(X)=\re\lambda'(X)$.
Recall 
from Section \ref{optimal bound} the 
constants $\lambda_X$ and $d_X$. Then $\re\xi(X)\le \lambda_X$ for all 
$\xi\in\Xi$, and for each $\xi\in\Xi$ with $\re\xi(X)=\lambda_X$
we have $\xi|_{\af_I}\in {\mathcal E}_{{\rm lead},I}$. By Lemma \ref{degree along ray}
the polynomial 
$t\mapsto q_{v, \xi,w} (Y+t X)$ 
has degree at most $d_X$ for all $Y\in\af_I$. 
Hence for each $\mu\in {\mathcal E}_{{\rm lead},I}$ 
with $\re\mu(X)=\lambda_X$
an assignment
$$\Lb_{I,X}^{\mu,w}: A_I \to (V^*)^{(\hf_w)_I}$$
is defined by
\begin{equation} \label{Lb def1} \Lb_{I,X}^{\mu,w}(a) (v) =\lim_{t\to \infty}    
e^{-\mu(tX)} t^{-d_X}  \Lb_I^{\mu,w} (a\exp(tX))(v).
\end{equation}

 \begin{lemma} \label{continuous and nonzero} Let $X\in \af_I$ be as above and $\mu\in {\mathcal E}_{{\rm lead},I}$ 
 such that  $\re\mu(X)=\lambda_X$. Suppose that $Z$ is either absolutely spherical or wave-front.
  Then the following assertions hold for 
 the $\mu$-polynomial map $\Lb_I^{\mu,w}: A_I \to (V^*)^{(\hf_w)_I}$: 
  \begin{enumerate} 
 \item \label{L continuous}$\operatorname{im} \Lb_{I,X}^{\mu, w} \subset 
 (V^{-\infty})^{(\hf_w)_I}$ for all  $w$ and $\mu$ as above.
 \item \label{L nonzero} $\Lb_{I,X}^{\mu, w}(\1)\neq 0$ for some  $w$ and $\mu$. 
 \end{enumerate}
\end{lemma} 

\begin{proof} (1)  Suppose first that $Z$ is absolutely spherical. Fix $a\in A_I$ and $v\in V$. 
It follows from (\ref{O})
that
\begin{equation*} m_{v,\eta} (\exp(t X)aw \cdot z_0)
=\sum_{\genfrac{}{}{0pt}{}{\mu\in {\mathcal E}_{{\rm lead},I}}{\re\mu(X)=\lambda_X}}
\Lb_I^{\mu,w}(a\exp(tX))(v) + o(e^{\lambda_X t})\,.\end{equation*}
With $\eta_\mu:= \Lb_{I,X}^{\mu,w}(a)$ we then deduce from Theorem \ref{lemma ub}(\ref{ub-one})  and (\ref{Lb def1}) that 
\begin{equation}\label{remainder I}  \big|\sum_{\genfrac{}{}{0pt}{}
{\mu\in \E_{{\rm lead}, I}}{\re\mu(X)=\lambda_X}}
e^{i\im \mu(X)t}\eta_\mu(v) + R_v(t)\, \,\big|\leq q(v)\, \end{equation}
for all $t\geq 1$, with remainder $R_v(t)\to 0$ for $t\to \infty$
and with a $G$-continuous norm $q$ on $V$.
As we have chosen $X$ generic enough, we can use the oscillatory argument from \cite{W}, Appendix 4.A.1.2(1) to deduce
$$\sum_\mu |\eta_\mu(v)|^2 \le \limsup_{t\to\infty} \big|\sum_{\mu}e^{i\im \mu(X)t}\eta_\mu(v) \big|^2$$
and hence
$|\eta_\mu(v)|\le q(v)$ for each $\mu$.
This proves (\ref{L continuous}) for $Z$ absolutely spherical.

If $Z$ is wave-front, then we proceed similarly
as above but use Theorem \ref{lemma ub wf} 
instead of Theorem \ref{lemma ub}(\ref{ub-one}) while carrying an additional $A_Z$-variable along: First we obtain as above for 
$b \in A_I$ that $\eta_b:=\Lb_{I,X}^{\mu, w}(b)\in (V^{-\omega})^{(\hf_w)_I}$.  Let $W_\mu \subset (V^{-\omega})^{(\hf_w)_I}$
be the finite dimensional space generated by $\operatorname{im}\Lb_{I,X}^{\mu, w}$. Observe that the action of $A_I$ on $W_\mu$ is by 
a single Jordan block corresponding to $\mu$. Hence with the notions of Theorem \ref{lemma ub wf} 
we find a $d\in\N$ such that 
$$|m_{v,\eta_b}(\omega a w'\cdot z_{w,I})| \leq  C_\e p_\e (v) b^\mu a^{\Lambda_{V,\eta}}  (1 + \|\log a\|+ \|\log b\|)^d$$
for all $v\in V$,  $a\in A_Z^-$, $b\in A_I$,  $\omega \in \Omega$, $w'\in \W$ and all $\e>0$.
It follows from \cite{KKSS2} Lemma 3.1 and Prop. 3.4  that 
there exists a weight ${\bf w}: Z_{w,I}\to \R^+$ such that 
$$ {\bf w} (abw' \cdot z_{w,I})\geq   b^\mu a^{\Lambda_{V,\eta}}  ( 1 + \|\log a\|+\|\log b\|)^d $$
for all $a\in A_Z^-$, $b\in A_I$ and $w'\in \W$. 
Note that $A_{Z_{w,I}}^- = A_I A_Z^-$. In particular, employing the polar decomposition for $Z_{w,I}$,  it follows that
$$ q(v):= \sup_{z\in  Z_{w,I}}   |m_{v,\eta_b}(z)| {\bf w}^{-1}(z) <\infty   \qquad (v\in V^\omega)\, .$$
We claim that $q$ is a $G$-continuous norm on $V$. To see that we first note that the $G$-action on the Banach space $E=L^\infty (Z, {\bf w}^{-1} dz)$ 
is by locally equicontinuous operators (see \cite{BK}, Lemma 2.3) and hence the $G$-continuous vectors $E_c\subset E$ form a
$G$-stable closed subspace of $E$, i.e.~the action of $G$ on $E_c$ defines a Banach-representation. Now $V\subset E_c$
as $V$ consists even of smooth vectors.  Hence $q$ is a $G$-continuous norm on $V$. 
As $|\eta_b(v)| \leq {\bf w} (z_{w,I}) q(v)$ for all $v\in V$, we deduce that 
$\eta_b\in (V^{-\infty})^{(\hf_w)_I}$. 

\par To establish (\ref{L nonzero}) we note that for each $\lambda\in \lambda_X+i\R$
$$e^{-\lambda t}\sum_{\genfrac{}{}{0pt}{}{\mu\in {\mathcal E}_{{\rm lead},I}}{\mu(X)=\lambda}}
\Lb_I^{\mu,w}(\exp(tX))(v)=
\sum_{\genfrac{}{}{0pt}{}{\xi\in \Xi}{\xi(X)=\lambda}}q_{v,\xi,w}(tX).$$
It follows from Lemma \ref{degree along ray} that this
polynomial in $t$
has degree $d_X$ for some $\lambda\in\lambda_X+i\R$ and some $v$ and $w$. 
Hence $\Lb_I^{\mu,w}(\exp(tX))(v)$ has degree $d_X$ for some $\mu\in {\mathcal E}_{{\rm lead},I}$
with $\mu(X)=\lambda$.
This implies $\Lb_{I,X}^{\mu, w}(\1)\neq 0$. 
 \end{proof} 
 
 We can equally implement $\Lb_{I,X}^{\mu,w}$ by a differential operator. To see that we define 
for $Y\in \af_I$ and $f$ a differentiable function on $A_I$: 
$$ \partial _Y f(a) := (Y\cdot f)(a)= \frac{d}{dt}\Big|_{t=0} f(a\exp(tY))\, .$$
Further for $\gamma\in \af_{I,\C}^*$ we define a 
first order operator 
$$\partial_{\gamma, Y} := 
e^{\gamma(\log (\cdot))} \circ \partial_Y \circ e^{-\gamma(\log(\cdot))} =\partial_Y-\gamma(Y)
.$$
The following is then immediate:

\begin{lemma} \label{AI continuous} Let $\mu$, $w$ be as in (\ref{Lb def1}).
Then 
$$ (\partial_{\mu,X})^{d_X} \Lb_I^{\mu,w}=d_X!\, \Lb_{I,X}^{\mu, w} \, .$$
\end{lemma}

 For the rest of this article we assume that $Z$ is either absolutely spherical or wave-front.

\begin{lemma} \label{eta0IwI}
Let $w_I\in \W_I$ and let $w\in\W$ correspond to $w_I$ as in Lemma \ref{wwI},
and let $\mu\in {\mathcal E}_{{\rm lead},I}$ 
with $\re\mu(X)=\lambda_X$.
There exists $m_w\in M$ such that
\begin{equation*} 
m_w w_I \cdot \Lb_{I,X}^{\mu,\1} (a)  = \Lb_{I,X}^{\mu,w}(a) 
\qquad (a\in A_I) \, .  \end{equation*}
\end{lemma}

\begin{proof}  We give the proof for $Z$ absolutely spherical - 
the proof for $Z$ wave-front 
is analogous.

At first let $w\in\W$ be arbitrary and $a_s = \exp(sX)$. Let $C_I \subset A_I$ be a compact subset. 
Then (\ref{remainder I}) implies: 
\begin{equation} \label{remainder I2}\sum_{\nu \in\E_{{\rm lead}, I} \atop \re \nu(X) 
=\re \mu(X)}  e^{s  (\nu-\mu)(X)} \Lb_{I,X}^{\nu,w}(a) 
= s^{-d_X}a_s^{-\mu} \left(aa_s \cdot \eta_w\right) + o({1\over s})\, .\end{equation} 
The $o({1\over s})$ has to be understood in the following way: Let $q=q(X)$ be the
$G$-continuous norm from Theorem \ref{lemma ub}(\ref{ub-one}) and $q^*$ its dual norm. 
Then (\ref{remainder I2}) means that 
\begin{equation} \label{remainder I3}  \lim_{s\to \infty}
q^*\Big( s^{-d_X}a_s^{-\mu}  \left(aa_s \cdot \eta_w\right) - \sum_{\nu \in\E_{{\rm lead}, I} \atop \re \nu(X) 
=\re \mu(X)}  e^{s  (\nu-\mu)(X)} \Lb_{I,X}^{\nu,w}(a) \Big)=0 \end{equation} 
uniformly for $a\in C_I$. 

\par Moreover, note that Lemma
\ref{continuous and nonzero} implies that 
\begin{equation} \label{osc-bound} \sup_{a\in C_I, s\in\R}  q^*\Big( \sum_{\nu \in\E_{{\rm lead}, I} \atop \re \nu(X) 
=\re \mu(X)}  e^{s  (\nu-\mu)(X)} \Lb_{I,X}^{\nu,w}(a) \Big)<\infty.\end{equation}

Now let $w$ correspond to $w_I$. From Lemma \ref{wwI} we obtain that
$w_I a_s = u_s m_s b_s  w h_s$ with $u_s\in U$, $m_s\in M$, $b_s \in A_Z$ and $h_s\in H$ 
such that $u_s\to \1$ and $b_sa_s^{-1}\to\1$ for $s\to \infty$.  Moreover, we can assume that 
 $(m_s)_{s\geq s_0}$ is convergent, say with limit $m_w^{-1}\in M$. 
 Then 
$$ a_s w \cdot \eta = a_s b_s^{-1} m_s^{-1}u_s^{-1} w_I a_s \cdot \eta\, .$$
If we thus set $g_s:= a_s b_s^{-1} m_s^{-1}u_s^{-1}$, then $(g_s)_{s\geq s_0}$ is convergent to 
$m_w$ and
\begin{equation} \label{gsw}  a_s w \cdot \eta = g_sw_I a_s \cdot \eta\, .\end{equation} 

\par From (\ref{remainder I3}) we have 
$$ \lim_{s\to \infty}
q^*\Big( s^{-d_X}a_s^{-\mu}  (aa_s \cdot \eta ) - \sum_{\nu \in\E_{{\rm lead}, I} \atop \re \nu(X) 
=\re \mu(X)}  e^{s  (\nu-\mu)(X)} \Lb_{I,X}^{\nu,\1}(a) \Big)=0 $$
uniformly for $a\in C_I$. 
As $q^*$ is $G$-continuous and  $\Lb_{I,X}^{\nu,\1}(a)$ is $H_I$-invariant we hence 
obtain 
$$ \lim_{s\to \infty}
q^*\Big( s^{-d_X}a_s^{-\mu}  (w_Ih aa_s) \cdot \eta  - \sum_{\nu \in\E_{{\rm lead}, I} \atop \re \nu(X) 
=\re \mu(X)}  e^{s  (\nu-\mu)(X))} [\Lb_{I,X}^{\nu,\1}(a)]_{w_I} \Big)=0$$
uniformly for $h$ in compacta of $H_I$ and $a\in C_I$. 
Now note that there exists $h_a\in H_I$, depending continuously on $a\in A_I$, such that 
$w_I h_a a = aw_I$ (see Lemma \ref{F-set}).  Hence we get 
$$ \lim_{s\to \infty}
q^*\Big( s^{-d_X}a_s^{-\mu}  (aw_Ia_s) \cdot \eta  - \sum_{\nu \in\E_{{\rm lead}, I} \atop \re \nu(X) 
=\re \mu(X)}  e^{s  (\nu-\mu)(X))} [\Lb_{I,X}^{\nu,\1}(a)]_{w_I} \Big)=0$$
uniformly for $a\in C_I$. 
Further by applying $g_s$, which in the limit commutes with $a$, we deduce with (\ref{osc-bound}) 
and the $G$-continuity of $q^*$,  that 
\begin{equation}\label{remainder I4} \lim_{s\to \infty}
q^*\Big( s^{-d_X}a_s^{-\mu}  (ag_sw_Ia_s) \cdot \eta  - \sum_{\nu \in\E_{{\rm lead}, I} \atop \re \nu(X) 
=\re \mu(X)}  e^{s  (\nu-\mu)(X))} m_w[\Lb_{I,X}^{\nu,\1}(a)]_{w_I} \Big)=0\end{equation}
uniformly for $a\in C_I$.  If we combine (\ref{remainder I4})  with (\ref{gsw}) and  (\ref{remainder I3}) we arrive at 
\begin{equation} \label{stoinfty} \lim_{s\to \infty} q^* \Big( \sum_{\nu \in\E_{{\rm lead}, I} \atop \re \nu(X) =\re \mu(X)}  
e^{s  (\nu-\mu)(X)}\big( \Lb_{I,X}^{\nu,w}(a) - m_w[\Lb_{I,X}^{\nu,\1}(a) ]_{w_I}\big)\Big)
\end{equation} 
uniformly for $a\in C_I$. 

\par We are now in the position to apply Lemma \ref{Wallach extended}. 
For that note that it is no loss of generality to assume that the $G$-continuous norm $q$ from 
Theorem \ref{lemma ub} is Hermitian (see \cite{BK}). Let $\Hc_0$ be the Hilbert 
completion of the Harish-Chandra module 
 $(V^{-\infty})^{K-{\rm fin}}$ with respect to $q^*$. With $\Hc:= L^2(C_I, \Hc_0)$
the lemma follows now from (\ref{stoinfty}) and Lemma \ref{Wallach extended}, which is 
easily extended to a Hilbert space valued setting.
\end{proof}

In the next step we produce out of $\Lb_{I,X}^{\mu,w}$ a continuous functional which transforms under the action 
of $A_I$ as a character. Note that 
$$\Lb_{I,X}^{\mu, w} (a)(v) = p_{v,\mu, w, X} (\log a) a^\mu$$
with $p_{v,\mu,w,X}$ a polynomial on $\af_I$ which is constant along all affine rays $t\mapsto Y + tX$ in 
$\af_I$.  
Now if all $p_{v,\mu,w,X}$ are constant, then we set $\eta_I^{\mu,w}:=\eta_{I,X}^{\mu,w}$.
Otherwise we set $X_1:=X$ and choose $0\neq X_2\in \af_I$ orthogonal to $X_1$
and let $d_2$ the maximum of the degrees of the polynomials $t \to p_{v,\mu,w,X_1}(Y+tX_2)$
for $Y\in \af_I$, $v\in V$. 
Then define for $a\in A_I$ and $v\in V$
$$\Lb_{I,X_1,X_2}^{\mu,w} (a) (v) = \lim_{t\to \infty} t^{-d_2} e^{-t\mu(X_2)}  \Lb_{I,X_1}^{\mu, w} (a \exp(tX_2))(v)\, .$$
As in Lemma \ref{AI continuous} we have 
$$\Lb_{I,X_1,X_2}^{\mu,w}  = {1\over d_2 !}(\partial_{\mu,X_2} )^{d_2} \ \Lb_{I,X_1}^{\mu, w} \, .$$
With Lemmas \ref{mu polynomial} and \ref{continuous and nonzero} 
we thus get $\Lb_{I,X_1,X_2}^{\mu,w} (a) \in (V^{-\infty})^{(\hf_w)_I}$. 
We continue that iteratively and arrive at an element $\eta_I^{\mu,w}\in (V^{-\infty})^{(\hf_w)_I} $
defined by
\begin{equation} \label{eta I mu w} \eta_I^{\mu, w}(v) =\Lb_{I, X_1, \ldots, X_{\dim \af_I}}^{\mu,w}(\1)(v)\end{equation}
with the additional property that it transforms under $A_I$ as a character
\begin{equation} \label{mu character} \eta_I^{\mu,w}(av) =a^\mu \eta_I^{\mu,w} (v)\qquad (a\in A_I, v\in V)\, .\end{equation}

Note that $\eta_I^{\mu,w}\neq 0$ if and only if  $L_{I,X}^{\mu, w}(\1)\neq 0$.
In particular, there 
exists $\mu\in \E_{{\rm lead}, I}$ and $w\in\W$ such that the functional $\eta_I^{\mu,w}$ is non-zero. 
Further we note that 
\begin{equation} \label{mc-eta-I} m_{v, \eta_I^{\mu,w}}(aw\cdot z_{0,I}) = 
\sum_{\genfrac{}{}{0pt}{}{\xi\in \Xi}{ \xi|_{\af_I} =\mu}}
 c_{v,\xi,w} (\log a) a^\xi
\qquad (a\in A_Z^- A_I , v\in V) \end{equation} 
with polynomials $c_{v,\xi,w}$ which do not depend on $\af_I$. 

\begin{lemma} \label{eta1IwI}
Let $w_I\in \W_I$ and let $w\in\W$ correspond to $w_I$ as in Lemma \ref{wwI}.
There exists $m_w\in M$ such that
\begin{equation} \label{wwIeta} 
m_w\cdot (\eta_I^{\mu,\1})_{w_I} = \eta^{\mu,w}_I.
  \end{equation}
\end{lemma}

\begin{proof} In view of the construction of $\eta^{\mu,w}_I$, this is now immediate from Lemma
\ref{eta0IwI} and Lemma \ref{AI continuous}.
\end{proof}

\subsection{The tempered embedding theorem}

Let $(V,\eta)$ be an $H$-spherical tempered pair. Then 
$\rho_Q(\omega_i)\le  
\Lambda_{V,\eta}(\omega_i)$  for $i=1,\dots,s$ by Theorem \ref{Lambda condition for tempered}.
For
each leading exponent $\lambda$ we have 
$\Lambda_{V,\eta}(\omega_i)\le 
\re\lambda(\omega_i)$ by definition, and hence 
\begin{equation}\label{squeeze Lambda}
\rho_Q(\omega_i)\le \Lambda_{V,\eta}(\omega_i)\le 
\re\lambda(\omega_i).
\end{equation}
We attach 
the following set of spherical roots to $\lambda$
\begin{equation}\label{defi I_eta}
I_{\eta,\lambda} :=\{ \sigma_i\in S : 
\rho_Q(\omega_i)<\re\lambda(\omega_i)\}\,.
\end{equation}

\begin{lemma} \label{I-ind} For $\lambda\in \E_{\rm lead}$ 
and $I=I_{\eta, \lambda}$ the following assertions
hold: 
\begin{enumerate}
\item \label{one}
$\rho_Q|_{\af_I} =\Lambda_{V,\eta}|_{\af_I}=
\re \lambda|_{\af_I}$. 
\item \label{three}For all $\lambda'\in \E_{{\rm lead}}$ and $\sigma_i \notin I$ one has $\re (\lambda'-\lambda)(\omega_i)\geq 0$. 
\item\label{two}  $\lambda|_{\af_I} \in \E_{{\rm lead}, I}$.
\item\label{four}  If $\lambda'\in \E_{{\rm lead}}$ and $\re\lambda'|_{\af_I}=\re\lambda|_{\af_I} $
then $I_{\eta,\lambda'}\subset I_{\eta,\lambda}$.
\end{enumerate}
\end{lemma}

\begin{proof}  (\ref{one}) 
follows from (\ref{squeeze Lambda}) and (\ref{defi I_eta})
since $\af_I$ is spanned by the $\omega_i$ for
$\sigma_i\notin I$. 
Next (\ref{three}) follows from (\ref{one})
together with (\ref{squeeze Lambda}).
Moving on to (\ref{two}) we argue by contradiction and assume that 
$\lambda|_{\af_I}\notin\E_{{\rm lead}, I}$. Then  
there exist an $\gamma\in \E_{{\rm lead},I}$ and a 
$\sigma\in \N_0[S]\bs \N_0[I]$ such that 
$ (\lambda-\sigma)|_{\af_I} =\gamma\, .$
Since $\sigma|_{\af_I}\neq 0$ this contradicts 
(\ref{three}). Finally, (\ref{four})  is obvious from the definition.
\end{proof}

Let $\min_\eta:=\min\{ \# I_{\eta,\lambda} \mid \lambda\in \E_{\rm lead}\} $.

\begin{lemma}\label{eta-optimal} Let $(V,\eta)$ be an $H$-spherical tempered pair. 
Then there exists a $\lambda\in \E_{\rm lead}$ such that 
for $I=I_{\eta, \lambda}$ and $\mu=\lambda|_{\af_I}$ the following assertions hold: 
\begin{enumerate} 
\item\label{unus} $\# I= \min_\eta$. 
\item\label{duo} $\eta_I^{\mu, w}\neq 0$ for some $w\in\W$.
\end{enumerate} 
\end{lemma}

\begin{proof}  Let $\gamma\in \E_{\rm lead}$ be such that (\ref{unus}) holds for
$I:=I_{\eta,\gamma}$. 
According to Lemma \ref{I-ind}(\ref{two}) we have $\gamma|_{\af_I}\in \E_{{\rm lead}, I}$.
For each $X\in \Sphere(\af_Z^-)$ we have $\re\gamma(X)\le \lambda_X\le \Lambda_{V,\eta}(X)$.
If in addition $X\in\af_I$ then Lemma \ref{I-ind}(\ref{one}) implies
$\re\gamma(X)=\lambda_X= \Lambda_{V,\eta}(X)$. 

Let $X\in \af_I$ be the sufficiently generic element used in the construction (above Lemma 
\ref{continuous and nonzero}) of $\Lb_I^{\mu,w}$. 
By this construction and the one above Lemma \ref{eta1IwI},
we infer that there 
exists a $\lambda \in \E_{\rm lead}$ with  $\re \lambda(X) =\lambda_X$ such that for some $w\in\W$
we have
$\eta_I^{\mu,w}\neq 0$ where $\mu:=\lambda|_{\af_I}$. 
It follows from (\ref{unus}) and Lemma
\ref{I-ind}(\ref{four}) that $I_{\lambda,\eta}=I$.  Hence the assertion follows. 
\end{proof}

Pairs $(\lambda, w)\in \E_{\rm lead}\times \W$ which satisfy the conditions of Lemma \ref{eta-optimal} will 
be called {\it $\eta$-optimal}.

\par Recall that $\af_{Z_I}=\af_Z$ and that $\af_{Z_I, E} =\af_I$.   For $\sigma\in I$ we define 
$\omega_{\sigma, I}\in\af_Z$ by 
\begin{itemize}
\item $\sigma'(\omega_{\sigma,I})= \delta_{\sigma \sigma'}$ for $\sigma'\in I$. 
\item $\omega_{\sigma, I}\perp \af_I$. 
\end{itemize}
We note that $\omega_{\sigma, S}$ coincides with the previously defined $\omega_\sigma$. In general
\begin{equation} \label{ooI} \omega_{\sigma, I} -\omega_\sigma \in\af_I \qquad (\sigma\in I)\, .\end{equation}

\par Let now $(\lambda,w)\in \E_{\rm lead}\times \W$ be $\eta$-optimal, $I=I_{\eta, \lambda}$ and $\mu=\lambda|_{\af_I}$. 
We then define a linear functional $\Lambda_{V,\eta, I}\in \af_Z^*$ by the requirements: 
\begin{itemize}
\item $(\Lambda_{V,\eta, I} -\rho_Q)|_{\af_I}=0$. 
\item $\Lambda_{V,\eta, I} (\omega_{\sigma, I}) = \min\{ \re \gamma(\omega_{\sigma, I})\mid \gamma\in \E_{\rm lead}, \re \gamma|_{\af_I} =\mu\}$ 
 for $\sigma\in I$. 
\end{itemize}

\begin{lemma}\label{discI} Let $(V,\eta)$ be a $Z$-tempered  and $(\lambda,w)\in \E_{\rm lead}\times\W$
be an associated  $\eta$-optimal pair. Then 
\begin{equation}\label{eta discrete} (\Lambda_{V,\eta, I} -\rho_Q)(\omega_{\sigma, I})>0 \qquad (\sigma \in I)\, .\end{equation}
\end{lemma}

\begin{proof} In view of the definition of $\Lambda_{V,\eta, I}$ and (\ref{ooI}) it is sufficient to check that 
$(\Lambda_{V,\eta, I} -\rho_Q)(\omega_\sigma)>0$ for all $\sigma\in I$.  Let now $\gamma\in \E_{{\rm lead }}$ be such that 
$\re \gamma|_{\af_I}  =\mu$.  
Then $I_{\eta, \gamma} =I$ by Lemma \ref{I-ind} (\ref{four}) and the minimality of $\# I$. Hence 
$(\re \gamma - \rho_Q)(\omega_\sigma)>0$ by the definition of $I_{\eta, \gamma}$. 
\end{proof}

\begin{theorem}\label{tet} Suppose that $Z$ is either absolutely spherical or of wave-front type. 
Let $(V, \eta)$ be an $H$-spherical
tempered pair with $V$ irreducible. 
Then there exists for every $\eta$-optimal pair $(\lambda, w)\in \E_{\rm lead}\times \W$
a boundary degeneration $Z_{w,I} =G/ (H_w)_I$ of $Z_w=G/H_w$,
a normalized unitary character $\xi$ of $(H_w)_IA_I$,
and a $(\gf, K)$-embedding 
$$ V\hookrightarrow L^2(\hat Z_{w,I}; \xi)$$
where $\hat Z_{w,I}:=G/(H_w)_IA_I$.
In particular, $V$ is unitarizable.
\end{theorem}

\begin{proof} Let $(\lambda,w)\in \E_{\rm lead}\times\W$ be  $\eta$-optimal and let 
$I=I_{\eta,\lambda}$,  $\mu=\lambda|_{\af_I}$. 
Then 
$\eta_I:= \eta_I^{\mu, w}\neq 0$, and recall 
that it is a continuous $(H_w)_I$-invariant functional. Hence after replacing $H$ by $H_w$ 
we may assume that $w=\1$.

As $V$ is irreducible, the map 
$$ \Phi: V^\infty \to C^\infty (Z_I), \ \ v \mapsto m_{v, \eta_I}$$
defines a $G$-equivariant injection.  Notice that by  (\ref{mu character}) all elements in
${\rm im}\  \Phi$ transform under the normalized unitary 
character $a^\mu$  by the right $A_I$-action. 

Recall the set $\W_I=\{ w_{1, I}, \ldots, w_{p, I}\}$. For $1\leq i \leq p$ let $w_i\in \W$
correspond to $w_{i,I}$ 
as in Lemma \ref{wwI}. 
From (\ref{wwIeta}) we get $m_j\cdot (\eta_I)_{w_{j,I}}= \eta_I^{\mu, w_j}$ for $j=1,\dots, p$,
and some $m_j\in M$.  Hence
\begin{equation}\label{metaIv}
m_{\eta_I,v}(a w_{j,I} \cdot z_{0,I})= m_{\eta^{\mu, w_j}_{I},v}(m_ja \cdot z_{0,w_j,I})
\end{equation}
for all $a\in A_Z$. Here $z_{0,w_j,I}$ denotes the origin of $G/(H_{w_j})_I$.
It follows from (\ref{mc-eta-I}), applied to the matrix coefficient on the right hand side
of (\ref{metaIv}), 
that all exponents belong to $\Xi$
and restrict to $\mu$ on $\af_I$. This being the case for all
$j=1,\dots, p$ we then infer from 
the definition of $\Lambda_{V,\eta, I}$
that
$$\Lambda_{V,\eta_I}(\omega_{\sigma,I}) \geq \Lambda_{V,\eta,I}(\omega_{\sigma,I}),\qquad(\sigma\in I)$$
for the exponent defined by (\ref{Lambdadef}) for
$\eta_I$.

The theorem now follows by applying Theorem \ref{CDS} to $Z_I$ (see also Lemma \ref{Z_I is unimodular}),
together with the inequality (\ref{eta discrete}).
\end{proof}

\begin{rmk}\label{chofP}  After a change of $P_{\min}$ (and $A$) it is possible to get an optimal pair $(\lambda, w)$ such that 
$w=\1$. 
In order to see that, note 
$$m_{v, \eta_w}(a\cdot z_w) =m_{w^{-1}v, \eta} (w^{-1} a w\cdot z_0)\, .$$
We are allowed to change $(A,K,P_{\min} )$ to $\Ad(w^{-1})(A,K,P_{\min})$ without effect on the 
developed theory of power series expansions.  The asserted 
reduction to $w=\1$ follows. 
\end{rmk}

\begin{cor} \label{temp-thm-wavefront} Assume that $Z$ is wave-front and let 
$(V,\eta)$, $(\lambda, w)$ be as above. Then there exist
\begin{enumerate}
\item a parabolic
subgroup $P\supset \oline{Q}$, 
\item a Levi decomposition $P=G_PU_P$,
\item
a real spherical subgroup $H'_P\subset G_P$ containing $(G_P\cap H_w)_e$,
\item
a twisted discrete series representation $\sigma$
of $G_P/H'_P$,
\end{enumerate}
such that $V$ embeds equivariantly
into the parabolically induced unitary representation 
$\operatorname{Ind}_{P}^G (\sigma\otimes 1)$.
\end{cor}

\begin{proof} Let $I$ as in Theorem \ref{tet}, and 
let $P=\oline{P_F}$ be as in Corollary \ref{sandwich2}, 
i.e.~$\oline {U_F} \subset H_I \subset {\oline P_F}$. In view of Lemma
\ref{hf} we have  $\oline {U_F} \subset (H_w)_I \subset \oline{P_F}$ as well. 
Hence the assertion follows 
with $H_P'=(G_F\cap (H_w)_I)_e$ from Lemma \ref{induced}.
\end{proof}

\begin{ex} Let $G/H$ be a symmetric space as in Example \ref{symmetric space case},
then $Z$ is wave-front.  

\par After possible change of $P$ to $w^{-1}Pw$ (see Remark \ref{chofP}) we may assume that 
there is a $\lambda\in \E_{\rm lead}$ such that $(\lambda, \1)$ is an $\eta$-optimal pair. 
After a  further change of $P$ to $hPh^{-1}$ for some $h\in H$ (see \cite{Matsuki}, Thm.~1) we may 
assume that  $P$ is a $\sigma\theta$-stable parabolic subgroup. 
Then $H_P'=(G_P^\sigma)_e$ is a symmetric subgroup in
$G_P$. Hence Corollary \ref{temp-thm-wavefront} implies 
Delorme's result from \cite{D} (up to components of $G_P^\sigma$).
The result of Langlands cited in the introduction then also follows.
\end{ex}

\end{document}